\documentclass[11pt]{article}

\usepackage[letterpaper,margin=1in]{geometry}

\usepackage[utf8]{inputenc}
\usepackage[T1]{fontenc}
\usepackage{graphicx}
\usepackage{amsfonts}
\usepackage{amsopn}
\usepackage{amsmath}
\usepackage{amssymb}
\usepackage{mathtools}
\usepackage{amsthm}
\usepackage{xcolor}
\usepackage{hyperref}
\hypersetup{colorlinks=true, allcolors=blue}
\usepackage{url}
\usepackage{cleveref}
\usepackage{scalerel,stackengine}
\usepackage{xspace}
\usepackage{cite}
\usepackage{xargs}

\usepackage{algorithmicx}
\usepackage{algorithm}
\usepackage{algpseudocode}

\makeatletter
\patchcmd{\ALG@step}{\addtocounter{ALG@line}{1}}{\refstepcounter{ALG@line}}{}{}
\newcommand{\ALG@lineautorefname}{Step}
\makeatother

\usepackage[shortlabels]{enumitem}

\makeatletter

\newcommand{\articletype}[1]{}

\newcommand{\@affilblock}{}
\newcommand{\@emailblock}{}
\newcommand{\@keywordblock}{}
\newcommand{\affil}[1]{\renewcommand{\@affilblock}{#1}}
\newcommand{\email}[1]{\renewcommand{\@emailblock}{#1}}
\newcommand{\keywords}[1]{\renewcommand{\@keywordblock}{#1}}

\usepackage{orcidlink}
\newcommand{\orcid}[1]{\,\orcidlink{#1}}

\let\article@author\author
\renewcommand{\author}[1]{%
  \article@author{%
    #1%
    \ifx\@affilblock\@empty\else
      \\[4pt]{\normalfont\small \@affilblock}%
    \fi
    \ifx\@emailblock\@empty\else
      \\[2pt]{\normalfont\small \texttt{\@emailblock}}%
    \fi
  }%
}

\makeatother

\date{\today}

\makeatletter
\AtBeginDocument{%
  \let\article@abstract\abstract
  \let\endarticle@abstract\endabstract
  \renewenvironment{abstract}{%
    \maketitle
    \article@abstract
  }{%
    \ifx\@keywordblock\@empty\else
      \par\medskip\noindent\textbf{Keywords:}\ \@keywordblock
    \fi
    \endarticle@abstract
  }%
}
\makeatother

\newcommand{\ack}[1]{\section*{Acknowledgments}#1}

\newcommand{\data}[1]{\section*{Data availability}#1}

\DeclareMathAlphabet{\mathup}{OT1}{\familydefault}{m}{n}

\DeclareSymbolFont{yhlargesymbols}{OMX}{yhex}{m}{n}
\DeclareMathAccent{\wideparen}{\mathord}{yhlargesymbols}{"F3}
\xdefinecolor{green}{rgb}{0.04, 0.85, 0.32}
\xdefinecolor{darkgreen}{rgb}{0.24, 0.7, 0.44}
\xdefinecolor{mint}{rgb}{0.24, 0.71, 0.54}
\xdefinecolor{officegreen}{rgb}{0.0, 0.5, 0.0}
\xdefinecolor{napiergreen}{rgb}{0.16, 0.5, 0.0}
\xdefinecolor{maroon}{rgb}{0.65,0.06,0.17}
\xdefinecolor{blue}{rgb}{0,0.2,0.6}
\xdefinecolor{phthalogreen}{rgb}{0.07,0.21,0.14}
\definecolor{grassgreen}{RGB}{92,135,39}

\newtheorem{theorem}{Theorem}[section]

\newtheorem{lemma}{Lemma}[section]
\newtheorem{proposition}{Proposition}[section]
\newtheorem{remark}{Remark}[section]
\newtheorem{assumption}{Assumption}[section]

\crefname{theorem}{theorem}{theorems}
\Crefname{theorem}{Theorem}{Theorems}

\crefname{corollary}{corollary}{corollaries}
\Crefname{corollary}{Corollary}{Corollaries}

\crefname{lemma}{lemma}{lemmas}
\Crefname{lemma}{Lemma}{Lemmas}

\crefname{proposition}{proposition}{propositions}
\Crefname{proposition}{Proposition}{Propositions}

\crefname{remark}{remark}{remarks}
\Crefname{remark}{Remark}{Remarks}

\crefname{assumption}{assumption}{assumptions}
\Crefname{assumption}{Assumption}{Assumptions}

\newcommand{\llim}{\nearrow}  
\newcommand{\rlim}{\searrow}  
\newcommand{\Diag}[1]{\mathsf{Diag}\left(#1\right)}                 
\newcommand{\del}[2]{\frac{\partial{#1}}{\partial{#2}}}             
\newcommand{\delll}[3]{\frac{\partial^2{#1}}{\partial{#2}\,\partial{#3}}} 
\newcommand{\mat}[1]{\mathbf{{#1}}}                                 
\renewcommand{\vec}[1]{\mathbf{{#1}}}                                 

\DeclareMathOperator*{\argmin}{arg\,min}  
\DeclareMathOperator*{\argmax}{arg\,max}  

\newcommand{\proj}{\mathit{P}}       
\newcommand{\Proj}[2]{\proj_{#1}{\left(#2\right)}}       

\newcommand*{\opt}{^{\mkern-1.5mu\mathrm{opt}}}               
\newcommand*{\tran}{^{\mkern-1.5mu\mathsf{T}}}                
\newcommand*{\adj}{^{\mkern-1.5mu\mathsf{*}}}                 
\newcommand*{\inv}{^{\mkern-1.5mu\mathsf{-1}}}                
\newcommand*{\pseudoinv}{^{\mkern-1.5mu\mathsf{\dagger}}}     
\newcommand*{\Pseudoinv}[1]{\left(#1\right)\pseudoinv}        
\newcommand{\domain}{\mathcal{D}}                             
\newcommand{\trace}{\mathrm{Tr}}                              
\newcommand{\Trace}[1]{\trace \left(#1\right)}                
\newcommand{\abs}[1]{\left| {#1} \right|}                  
\newcommand{\norm}[1]{\left\| {#1} \right\|}                  
\newcommand{\sqnorm}[1]{\left\| {#1} \right\|^2}              
\newcommand{\wnorm}[2]{\left\| {#1} \right\|_{#2}}            

\newcommand\restr[2]{{ \left.\kern-\nulldelimiterspace        
                     {#1}\vphantom{\big|} \right|_{#2}}}

\newcommand{\Rnum}{\mathbb{R}}  
\newcommand{\xcont}{u}                             

\newcommand{\y}{\mathbf{y}}                        
\newcommand{\obs}{\y}                              

\newcommand{\param}{\vec{\theta}}                  
\newcommand{\iparam}{\param}                       
\newcommand{\iparprior}{\iparam_{\rm pr}}          
\newcommand{\iparb}{\iparprior}                    
\newcommand{\iparpost}{\iparam_{\rm post}^\obs}    
\newcommand{\ipara}{\iparpost}                     

\newcommand{\model}{\Model}

\newcommand{\inoise}{\vec{\varepsilon}}

\newcommandx{\noise}[1][1={}]{\inoise^{\rm #1}}

\newcommand{\Nobs}{\textsc{N}_{\rm obs}}                        
\newcommand{\Nens}{\textsc{N}_{\rm ens}}                        
\newcommand{\nobs}{n_{t}}                                  
\newcommand{\nobstimes}{\nobs}                             
\newcommand{\Nsens}{n_{\rm s}}                         
\newcommand{\Cparamprior}{\mat\Gamma_{{\rm pr}}}                
\newcommand{\Cparampost}{\mat\Gamma_{{\rm post}}}               
\newcommand{\Cobsnoise}{\mat{\Gamma}_{ {\rm noise}}}            
\newcommand{\Cparampriormat}{\Cparamprior}                      
\newcommand{\Cparampostmat}{\Cparampost}                        

\newcommand{\Fcont}{\mathcal{F}}                                 
\newcommand{\F}{\mathbf{F}}                                      

\newcommand{\obj}{\mathcal{U}}            
\newcommand{\stochobj}{\Upsilon}          

\newcommand{\utilityfunction}{\mathcal{U}}

\ExplSyntaxOn
  \NewDocumentCommand{\PP}{ s O{} >{\SplitArgument{1}{|}}m }
   {
    \mathbb{P}
    \IfBooleanTF{#1}
     { \PPauto #3 }
     { \PPfixed {#2} #3 }
   }

  \NewDocumentCommand{\PPauto}{mm}
   {
    \left(
    \IfNoValueTF{#2} { #1 } { #1 \;\middle|\; #2 }
    \right)
   }

  \NewDocumentCommand{\PPfixed}{mmm}
   {
    \mathopen{#1(}
    \IfNoValueTF{#3} { #2 } { #2 \mathrel{#1|} #3 }
    \mathclose{#1)}
   }
\ExplSyntaxOff

\ExplSyntaxOn
  \NewDocumentCommand{\PC}{ s O{} >{\SplitArgument{1}{|}}m }
   {
    \mathsf{cov}
    \IfBooleanTF{#1}
     { \PCauto #3 }
     { \PCfixed {#2} #3 }
   }

  \NewDocumentCommand{\PCauto}{mm}
   {
    \left(
    \IfNoValueTF{#2} { #1 } { #1 \;\middle|\; #2 }
    \right)
   }

  \NewDocumentCommand{\PCfixed}{mmm}
   {
    \mathopen{#1(}
    \IfNoValueTF{#3} { #2 } { #2 \mathrel{#1|} #3 }
    \mathclose{#1)}
   }
\ExplSyntaxOff

\newcommand{\Prob}{\mathbb{P}}                                   
\newcommand{\CondProb}[2]{\PP*{#1\! |\! #2} }  
\newcommand{\Var}{\mathsf{var}}                                  
\newcommand{\Cov}{\mathsf{cov}}                                  
\newcommand{\brCov}[1]{\Cov{\Bigl(#1\Bigr)} }                                  
\newcommand{\brVar}[1]{\Var\left(#1\right)}                      
\newcommand{\CondVar}[2]{\Var{\left(#1|#2\right)}}                                  
\newcommand{\GM}[2]{\mathcal{N}\!\left( {#1}, {#2}\right)}       

\newcommand{\Expect}[2]{\mathbb{E}_{#1}{\Bigl[ #2 \Bigr]} }     
\newcommand{\CondExpect}[2]{\mathbb{E}{\Bigl[ #1|#2\Bigr]} }     
\newcommand{\baseline}{b}

\newcommand{\Cobj}{T_{\!\obj}}

\newcommand{\FisherI}{\mathcal{I}}    
\newcommand{\hyperparam}{\vec{p}}     
\newcommand{\design}{\vec{\zeta}}                                       
\newcommand{\designdomain}{\mathcal{X}}               
\newcommand{\designsum}{Z}                                       
\newcommand{\designsumval}{z}                                          
\newcommand{\designsumset}{\mathcal{Z}}                                

\newcommand{\R}{ \mathbf{R} }

\newcommand{\Model}{\mathcal{M}}

\newcommand{\pyoed}{{PyOED}\xspace}

\begin{document}

\articletype{Paper} 

\title{Probabilistic Approach to Black-Box Binary Optimization with Budget Constraints: Application to Sensor Placement}
\author{Ahmed Attia\orcid{0000-0001-5940-9247}}

\affil{
  Mathematics and Computer Science Division,
  Argonne National Laboratory, Lemont, IL, USA
}

\email{aattia@anl.gov}


\begin{abstract}
 This paper presents a fully probabilistic approach for solving optimal experimental design problems under budget constraints. The experimental design is viewed as a random variable and is associated with a parametric conditional distribution that inherently models the budget constraints. The original optimization problem is replaced with an optimization over the expected value of the original objective, which is then optimized over the distribution parameters. The resulting optimal parameter (policy) is used to sample the feasible region of binary space to produce estimates of the optimal solution(s) of the original optimization problem. In this work we extend the family of conditional Bernoulli models to model the random variable conditioned by the total number of nonzero entries, that is, the budget constraint. This approach (a) is generally applicable to binary optimization problems with nonstochastic black-box objective functions and budget constraints; (b) employs conditional probabilities to model and sample only the feasible region and thus considerably reduces the computational cost compared with employing soft constraints; and (c) does not employ soft constraints and thus does not require tuning of a regularization parameter, for example to promote sparsity, which is generally challenging. The proposed approach is verified numerically using an optimal sensor placement experiment based on an advection-diffusion forward model in a parameter identification setup.
\end{abstract}

\section{Introduction}
  \label{sec:Introduction}
  This paper develops a probabilistic, black-box approach for solving binary optimization
  problems with budget constraints, with optimal sensor placement for inverse
  problems as the primary application.
  We motivate the problem and review relevant background in
  \Cref{subsec:scientific_simulations,subsec:background_OED},
  before presenting the method and its theoretical properties.

  Model-based optimal experimental design (OED) \cite{FedorovLee00,Pazman86,Pukelsheim93}
  is the general mathematical formalism for optimizing the configuration 
  of an experimental or observational setup. 
  Examples include optimizing the placement of sensors or measurement instruments
  and the scheduling of data acquisition for physical experiments.
  Model-based OED has seen a recent surge of interest in the field of sensor placement for 
  model-constrained Bayesian inverse problems \cite{huan2024optimal,alexanderian2020optimal}. 

  \subsection{Predictive Simulations and Inverse Problems}
    \label{subsec:scientific_simulations}
    Consider the following abstract formulation. The mathematical model
    $\model\left(\iparam; \noise[\model] \right) \!=\! 0$
    (e.g., a PDE) governs the evolution of the physical phenomenon of interest,
    with $\iparam$ denoting the model parameter, for example, initial condition,
    state, or both.
    The model is driven by a stochastic forcing $\noise[\model]$ 
    with a probability law that accounts for the discrepancies that affect model 
    predictability \cite{Hairer09}.
    The relation between observations $\obs$ and the 
    model parameter $\iparam$ is typically characterized by the forward problem
    $ \obs  = \Fcont(\iparam) + \noise[\obs] $,
    where the forward operator $\Fcont$, also known as the \emph{parameter-to-observable map}, 
    maps the parameter $\iparam$ onto the spatiotemporal observation space and
    $\noise[\obs]$ quantifies the observational uncertainties. 

    A fundamental computational science problem is the \emph{inverse problem}
    (commonly called data assimilation, DA, in atmospheric sciences):
    the retrieval of the model parameter $\iparam$
    from noisy observations $\obs$ given the model dynamics
    \cite{vogel2002computational,bertero2020introduction,kaipio2006statistical,tarantola2005inverse}.
    DA has gained special consideration because of the computational complexity 
    of simulations targeted by DA algorithms, 
    lack of physical knowledge or representation, 
    and computational 
    limitations \cite{ghil1991data,bouttier2002data,evensen2009data,kalnay2003atmospheric,attia2016reducedhmcsmoother,attia2015hmcfilter}. 
    Common DA approaches include 
    variational DA \cite{asch2016data,bannister2017review};
    ensemble methods such as
    Kalman filtering  
    \cite{evensen2009data,chen2003bayesian,fillion2020iterative} and
    particle filtering 
    \cite{van2009particle,chen2003bayesian,holm2020massively,morzfeld2018variational}; and 
    Markov chain Monte Carlo methods 
    \cite{attia2015hmcsmoother,attia2016reducedhmcsmoother,attia2018ClHMCAtmos}.

    DA algorithms generally aim to estimate the unknown parameter $\iparam$ 
    with associated uncertainties
    typically by estimating the moments of the posterior.
    For example, they can calculate a maximum a posteriori (MAP) estimate  
    of the unknown parameter $\iparam$ 
    and estimate the posterior covariance 
    to quantify the uncertainties in such an estimate. 
    Given a prior $\Prob(\iparam)$, 
    Bayesian inversion employs Bayes' theorem to characterize the knowledge 
    about $\iparam$ through the posterior: 
    \begin{equation}\label{eqn:inversion_posterior}
      \CondProb{\iparam}{\obs} 
        \propto \CondProb{\obs}{\iparam} \Prob(\iparam) \,,
    \end{equation}
    where $\CondProb{\obs}{\iparam}$ is the data likelihood. 
    For Gaussian observation errors $\vec{\delta} \sim \GM{\vec{0}}{\Cobsnoise}$,
    Gaussian prior $\iparam \sim \GM{\iparb}{\Cparampriormat}$, 
    and linear forward operator $\Fcont\equiv\F$,
    the posterior $\CondProb{\iparam}{\obs} = \GM{\ipara}{\Cparampostmat}$ is Gaussian,
    with 
    \begin{equation}\label{eqn:Posterior_Params}
      \Cparampostmat = \left(\F \adj \Cobsnoise^{-1} \F 
        + \Cparampriormat^{-1} \right)^{-1} \,, \quad
      \ipara = \Cparampostmat \left( \Cparampriormat^{-1} \iparb 
        + \F\adj \Cobsnoise^{-1}\, \obs \right) \,,
    \end{equation}
    where $\F\adj$ is the forward operator adjoint. 
    When the model $\Fcont$ is nonlinear and/or the observational errors are non-Gaussian, 
    the Bayesian inverse problem can still be solved, for example, 
    by following a Laplace approximation approach in which the posterior is approximated 
    with a Gaussian centered around the MAP estimate $\iparam^{\rm map}$, 
    that is employing the tangent linear model 
    $\F:=\partial\model(\iparam)|_{\iparam=\iparam^{\rm map}}$. 
    Alternative approaches that accommodate nonlinearities and unconventional uncertainties
    include ensemble-based methods.

    DA methods rely on highly informative data to accurately infer the unknown parameter for 
    more accurate simulations and prediction.
    Thus, optimizing data acquisition, especially under budgetary or computational limitations, 
    e.g., limited number of sensors, is of utmost importance to large-scale inverse problems.
    Model-based OED enables tackling this problem.

  \subsection{Model-Based Optimal Experimental Design}
  \label{subsec:background_OED}
    In model-based OED, an experimental design $\design\in\designdomain$ encodes 
    the decision variables such as sensor locations, data acquisition frequency,
    the trajectory of a drone or a satellite, 
    real-time data acquisition in a streaming environment, or 
    static/dynamic settings of an experiment. 
    OED problems generally seek a design that ``optimizes'' some
    ``utility function'' $\utilityfunction$ that quantifies
    the quality of the design $\design$ and depends on $\design$
    and other arguments originating in the inference problem \eqref{eqn:inversion_posterior}.
    The utility function is typically defined as a scalar summary of either
    the Fisher information matrix $\mathcal{I}$, which is maximized,
    or the posterior covariance $\Cparampostmat$ \eqref{eqn:Posterior_Params},
    which is minimized to obtain an optimal design.
    Scalar summaries includes the trace (A-optimal design), the determinant
    (D-optimal design), and the maximum eigenvalue (E-optimal design)
    of the posterior covariance of the inferred parameters,
    which are ideal for linear inverse problems.
    Evaluating $\utilityfunction$ typically requires numerous model simulations
    and is thus computationally expensive.
    Choices of the utility function in nonlinear 
    settings include the 
    expected information gain (EIG) 
    which is maximized to obtain an optimal design, 
    and the alphabetic criteria (e.g., A-, D-, E-optimal designs) combined with 
    Laplace approximation of the posterior around the MAP estimate; see e.g., \cite{chowdharyNonlinearRobust}.

    Of special interest is the problem of optimal resource allocation 
    and optimal sensor placement in large-scale inverse 
    problems \cite{bannister2017review,attia2016PhD_advanced_sampling,attia2016reducedhmcsmoother}. 
    Given $\Nsens$ candidate sensor locations,
    this problem is usually formulated as a model-constrained OED problem \cite{Pukelsheim93,Pazman86,Ucinski00}:
    \begin{equation}\label{eqn:binary_optimization_of_interest}
      \design\opt \in \argmax_{\design\in\designdomain}{\obj(\design) }
        \quad \textrm{s.t.} \quad  \design \in \{0, 1 \}^{\Nsens} \,, \quad
           \wnorm{\design}{0} \in \designsumset \subseteq \{0, 1, \ldots, \Nsens\}  \,,
    \end{equation}
    where $\design$ is an experimental design
    with a binary encoding of candidate sensor locations;
    $\design_i=1$ means placing a sensor, and $\design_i=0$ means not placing a sensor at the
    $i$th candidate location.
    Here, $\wnorm{\design}{0} := \sum_{i=1}^{\Nsens} \design_i$ denotes the $\ell_0$-norm
    (i.e., the number of active sensors), and
    $\designsumset \subseteq \{0, 1, \ldots, \Nsens\}$ is a prescribed set of admissible budget values
    (e.g., $\designsumset = \{z\}$ for a fixed budget $z$, or $\designsumset = \{1,\ldots,z\}$ for an
    upper-bound constraint).
    The feasible design domain $\designdomain\subseteq \{0, 1\}^{\Nsens}$ is the set of all binary
    vectors satisfying the budget constraint $\wnorm{\design}{0}\in\designsumset$,
    which is critical, for example, when the sensors are
    expensive to deploy and/or operate.
    The majority of OED problems can be written as maximization or equivalent 
    minimization problems. For example, in regression problems
    an A-optimal design can be defined as the maximizer of the 
    Fisher information matrix trace, or equivalently the minimizer of 
    its inverse assuming the information matrix is full rank \cite{FedorovLee00}.
    Care must be taken for inverse problems, however, because the forward operator is
    typically low-rank, making the Fisher information matrix rank-deficient.
    This is a primary motivation for Bayesian formulations in OED for inverse problems.
    In our work, we consider both maximization and minimization equally and leave the choice of the
    OED criterion (the utility function) for the user to define.

    Traditional approaches for solving the budget-constrained binary optimization problem \eqref{eqn:binary_optimization_of_interest}
    include \emph{greedy/myopic} \cite{koval2024non} and
    \emph{design-space relaxation} \cite{attia2022optimal,yu2018scalable} methods.
    While they can provide useful results for small-scale problems,
    relaxation methods solve an alternative (continuous) problem
    and typically require a soft penalty parameter to promote sparsity
    \cite{alexanderian2014optimal,attia2022optimal},
    whose tuning is generally challenging.
    Greedy methods, on the other hand, are generally limited to binary designs, fixed budgets,
    and do not inherently account for uncertainties or
    scale well with increasing inference and/or design dimensionality.

    This paper describes a fully probabilistic approach for solving budget-constrained
    binary sensor placement problems \eqref{eqn:binary_optimization_of_interest} which 
    addresses these limitations. 
    Specifically, the proposed probabilistic approach  
    regards the utility function $\utilityfunction$ as a black-box
    real-valued deterministic objective function, 
    and optimizes over the parameters of a probability distribution on the feasible binary designs.
    This can be viewed as a relaxation in the Bernoulli parameter space rather than in the design space.
    Unlike design-space relaxation, the proposed approach requires no soft penalty parameter and every
    sample drawn from the distribution is a feasible binary design by construction,
    since the budget constraint is encoded directly in the probability model.

    The proposed framework is formulated as a general-purpose black-box optimizer:
    the utility $\utilityfunction$ is treated as a real-valued deterministic function,
    with no gradient or structural information assumed.
    This is appropriate for the many settings where such information is unavailable ---
    for example, nonlinear inverse problems where the EIG must be estimated by nested
    Monte Carlo \cite{huan2024optimal}, legacy simulation codes with no adjoint,
    robust OED formulations with nonsmooth ensemble objectives \cite{attia2023robust},
    or E-optimal design where the maximum-eigenvalue criterion is non-differentiable
    at points of repeated eigenvalues.
    In contrast, for \emph{linear-Gaussian} inverse problems where the posterior covariance
    and Fisher information matrix admit closed-form expressions, structure-aware methods
    \cite{alexanderian2014optimal,AlexanderianSaibaba17} may be more efficient;
    in such cases the proposed approach trades that efficiency for generality.
    Moreover, \eqref{eqn:binary_optimization_of_interest} is an NP-hard combinatorial
    problem that arises broadly --- in machine learning and computer vision
    \cite{boykov2001fast,keuchel2003binary,shi2000normalized,davenport20141,shi2013multi,wang2015learning},
    natural and social sciences \cite{he2016joint,chan2011convex,yuan2013truncated},
    healthcare \cite{elton2019deep,yang2019machine}, and engineering
    \cite{wolsey2014integer,papalexopoulos2022constrained} --- and the framework
    applies to all such settings without modification.

    \noindent{}\textbf{Contributions.}
      This work directly extends the probabilistic binary optimization framework
      introduced in \cite{attia2022stochastic} (Attia, Leyffer, and Munson, SISC 2022),
      of which the current author is a co-author,
      and its robust extension in \cite{attia2023robust}.
      The prior work \cite{attia2022stochastic} 
      handles unconstrained binary optimization
      and enforces budget constraints only via soft penalty terms added to the objective,
      which requires expensive and often difficult tuning of the regularization parameter.
      The present paper addresses this limitation by developing conditional probability models
      that intrinsically encode the budget constraint,
      thereby replacing the combinatorial feasibility requirement with a well-defined
      probability distribution over the feasible set.
      The specific contributions of this work, which are not present in \cite{attia2022stochastic},
      are summarized as follows.
      \begin{enumerate}[leftmargin=*]
        \item We present an extensive treatment of a class of conditional distributions
          suitable for modeling budget constraints in binary optimization,
          namely, the Poisson binomial and the conditional Bernoulli models. 
          The existing theory and computational methods for this class of probability models,
          however,
          are very sparse and are limited to non-degenerate probabilities.
          We considerably extend these models to properly handle degenerate probabilities,
          discuss their first- and second-order moments,
          develop analytical forms of their derivatives, and develop bounds
          on derivatives 
          needed for the convergence analysis of the proposed
          approach.
          In addition to the crucial role these derivatives play in the proposed approach,
          they are also suitable in general for model fitting,
          e.g., maximum likelihood estimation.
        \item We propose a fully probabilistic approach that employs the developed
          probability models
          and we provide a complete algorithmic statement
          for solving binary optimization problems with
          black-box objectives and \emph{hard} budget constraints,
          without soft penalty parameters.
          This approach can be used as a plug-and-play tool for solving a wide
          range of sensor placement optimization and decision-making problems where the objective
          can be treated as a black box.
        \item We discuss the convergence of the proposed approach and analyze its performance,
          strengths, and  limitations.
        \item We provide an open-source implementation 
          through
          \pyoed \cite{attia2024pyoed}, and we provide 
          extensive numerical analysis for the proposed approach
          using an OED optimal sensor placement experiment. 
      \end{enumerate}

    \noindent{}\textbf{Paper outline.} 
    \Cref{sec:background} briefly reviews the probabilistic approach to unconstrained
    binary optimization from \cite{attia2022stochastic} and identifies
    the soft-constraint limitation that motivates the present work.
    \Cref{sec:probability_models} develops the conditional probability models  
    required for the proposed probabilistic optimization approach, which is presented in 
    \Cref{sec:probabilistic_optimization}.
    Numerical results are discussed in  \Cref{sec:numerical_experiments}, and 
    conclusions are made in \Cref{sec:conclusions}.

  \section{The Probabilistic Approach to Binary Optimization}
    \label{sec:background}
    This section provides a concise overview of the probabilistic binary optimization
    framework from \cite{attia2022stochastic}, which forms the foundation of the present work.
    We highlight the limitation that motivates our extension: the inability to enforce
    hard budget constraints without soft penalty terms.
    Readers familiar with \cite{attia2022stochastic} may proceed directly to
    \Cref{sec:probability_models}.

    Inspired by policy optimization in reinforcement learning
    \cite{BertsekasTsitsiklis96},
    \cite{attia2022stochastic} introduced a probabilistic approach for unconstrained
    binary optimization, subsequently extended to robust (max-min) optimization
    in \cite{attia2023robust}.
    While this approach is well-suited for black-box objective functions,
    budget constraints are enforced only via soft penalty terms appended to the objective,
    requiring expensive and often difficult tuning of the regularization parameter.
    This is especially challenging in robust settings, where traditional tuning methods such as
    the L-curve \cite{lawson1995solving} can fail due to the bilevel structure of the problem.
    Heuristic approaches have also been proposed \cite{ryu2023heuristic}, but these are
    problem-specific, still explore the full probability space, and are not guaranteed to
    yield feasible solutions.
    The present paper resolves this limitation by developing conditional probability models
    that encode the budget constraint by construction;
    see \Cref{subsec:modeling_constraints}. 

    The original probabilistic approach \cite{attia2022stochastic} addresses
    the unconstrained binary optimization problem
    \begin{equation}\label{eqn:binary_optimization_unconstrained}
      \design\opt \in \argmax_{\design}{\obj(\design) }
        \quad \textrm{s.t.} \quad
          \design=(\design_1, \ldots, \design_{\Nsens})\tran \in \{0, 1\}^{\Nsens}
          \,.
    \end{equation}
    The central idea is to treat $\design$ as a random variable following a multivariate
    Bernoulli distribution with parameter $\hyperparam\in[0,1]^{\Nsens}$,
    whose joint probability mass function (PMF) is
    \begin{equation}\label{eqn:joint_Bernoulli_pmf_prod}
      \CondProb{\design}{\hyperparam}
        := \prod_{i=1}^{\Nsens} {\hyperparam_i ^ {\design_i}\,
          \left(1-\hyperparam_i\right)^{1-\design_i}}
          \,,\quad \design_i \in \{0, 1\} \,,\quad \hyperparam_i \in [0,1] \,,
    \end{equation}
    and to replace \eqref{eqn:binary_optimization_unconstrained} with a stochastic
    optimization problem over $\hyperparam$,
    \begin{equation}\label{eqn:stochastic_optimization}
      \hyperparam\opt \in \argmax_{\hyperparam \in [0,1]^{\Nsens}}
        \stochobj(\hyperparam)
      := \Expect{\design\sim\CondProb{\design}{\hyperparam}}{\obj(\design)}
        \,.
    \end{equation}
    The stochastic objective $\stochobj(\hyperparam)$ is optimized via a policy-gradient
    (REINFORCE-type \cite{BertsekasTsitsiklis96}) stochastic gradient ascent,
    where the gradient is estimated using $\Nens$ Monte Carlo samples
    $\{\design[j]\}_{j=1}^{\Nens} \sim \CondProb{\design}{\hyperparam}$:
    \begin{equation}\label{eqn:kernel_stochastic_gradient_Bernoulli}
      \nabla_{\hyperparam} \stochobj(\hyperparam)
      =
      \Expect{}{
        \obj(\design) \nabla_{\hyperparam} \log  \CondProb{\design}{\hyperparam}
      }
      \approx
        \frac{1}{\Nens} \sum_{j=1}^{\Nens}
          \obj(\design[j]) \nabla_{\hyperparam} \log  \CondProb{\design[j]}{\hyperparam} \,.
    \end{equation}
    A key property is that this gradient estimator requires only black-box evaluations of $\obj$:
    the cost of the stochastic gradient step is dominated by the cost of evaluating $\obj$
    at the $\Nens$ sampled designs, making surrogate or reduced-order model acceleration
    directly applicable \cite{FrangosMarzoukWillcoxEtAl11}.
    We refer the reader to \cite{attia2022stochastic} for the full derivation,
    convergence analysis, and algorithmic details of this unconstrained framework.

    \subsection{Modeling budget constraints}
      \label{subsec:modeling_constraints}
      Enforcing hard constraints requires modeling $\design$ as a conditional random
      variable whose support is restricted to the feasible region.
      Specifically, to employ \eqref{eqn:stochastic_optimization} for
      solving \eqref{eqn:binary_optimization_of_interest}, $\design$ must follow a
      conditional distribution
      $\CondProb{\design}{\hyperparam;\, \design\in\designdomain\subseteq\{0, 1\}^{\Nsens}}$.
      Deriving such a conditional distribution is nontrivial, and different constraint
      types generally require different models.
      In this work we focus on budget constraints suitable for decision-making, 
      sensor placement, and OED. 
      The requisite conditional probability models are developed in \Cref{sec:probability_models},
      and the probabilistic optimization problem employing these probability models
      is then presented in \Cref{sec:probabilistic_optimization}.

  \section{Probability Models} 
    \label{sec:probability_models}
    This section is self-contained, and 
    the reader interested only in the proposed optimization approach can 
    skip to  \Cref{sec:probabilistic_optimization} and 
    refer back to 
    this section as needed.
    Here we develop the probability models:
    \begin{equation}\label{eqn:probability_models}
      \CondProb{\designsum}{\hyperparam}\,; \quad
      \CondProb{\design}{\designsum;\, \hyperparam}\,; \qquad 
      \designsum=\wnorm{\design}{0}=\sum_{i=1}^{\Nsens}{\design_i} \,,
    \end{equation}
    where  
    $\design=(\design_{1}, \ldots, \design_{\Nsens})\tran$ is 
    a collection of binary random variables (Bernoulli trials)
    with success probabilities 
    $\hyperparam_{j}\in[0, 1], j=1, \ldots,\Nsens$.

    First, $\CondProb{\designsum}{\hyperparam}$ describes 
    the probability of achieving a total number of successes (number of entries equal to $1$) 
    in the $\Nsens$ trials. 
    Since the success probabilities are not necessarily 
    identical, this distribution is described by the Poisson binomial (PB)
    model \cite{chen1997statistical}.
    Second, $\CondProb{\design}{\designsum;\, \hyperparam}$ 
    describes the probability of an instance of the binary variable $\design$, that is, a collection 
    of the $\Nsens$ binary trials, conditioned by their sum. 
    This distribution can be defined by extending the conditional Bernoulli (CB) 
    model \cite{chen1997statistical}.
    
    In the rest of this section we focus on 
    the PB (\Cref{subsec:Poisson-Binomial}) 
    and CB (\Cref{subsec:Conditional-Bernoulli}) models, respectively.
    For brevity, we suppress the dependence on $\hyperparam$ from all distributions
    and restore it explicitly only where the context requires it.
    Additionally, the proofs of theorems and lemmas 
    are provided in \Cref{app:probability_models_proofs}.

    \subsection{Preliminary relations}
    \label{subsec:Preliminary-Relations}
      The following combinatorial 
      \emph{R-function}
      \begin{equation}\label{eqn:R_function_and_Bernoulli_weights}
        R(k, A) := \sum_{\substack{B\subseteq A \\ |B|=k}} \prod_{i \in B} w_i \,, \quad 
        w_i = \frac{\hyperparam_i}{1-\hyperparam_i}\,, \,\,
        i=1, \ldots,\Nsens \,; \,\,
        A \subseteq S:=\{1, 2, \ldots \Nsens\} \,,
      \end{equation}
      is typically used to define the PMFs of PB and CB models \eqref{eqn:probability_models}.
      Evaluating \eqref{eqn:R_function_and_Bernoulli_weights} by enumerating
      all possible combinations is computationally intractable for large $\Nsens$.
      It can be computed efficiently via the following recurrence
      relation \cite{chen1997statistical,kroese2013handbook}:
      \begin{subequations}\label{eqn:R_function_and_Bernoulli_weights_recurrence_relations}
        \begin{align}
          \label{eqn:R_function_and_Bernoulli_weights_recurrence_relation_2_formula}
            R(k, A) &= R(k, A\setminus \{k\}) + w_k R(k-1, A\setminus \{k\}) \,;\qquad k=2, \ldots, \abs{A}  \\
            \label{eqn:R_function_and_Bernoulli_weights_recurrence_relation_1_IC}
            R(1, A) &= \sum_{\substack{B\subseteq A \\ |B|=1}} \prod_{i \in B} w_i
           = \sum_{\substack{B\subseteq \{\{1\}, \ldots,\{\Nsens\} \}}} \prod_{i \in B} w_i
           = \sum_{i \in A} w_i
           = \sum_{i\in A} \frac{\hyperparam_i}{1-\hyperparam_i}  \\
            R(0, A)&=1\,;\qquad 
              R(k, A)=0\,\, \forall \, k<0 \text{ or } k>\abs{A} \,.
        \end{align}
      \end{subequations}
      %

      \subsubsection{Inclusion probability.}
      \label{subsubsec:inclusion_probability}
        %
        The $k$th-order inclusion (coverage) probability 
        is the probability that the set of unique 
        $k$ units $\{i_1, \ldots, i_k\}$ is included in a sample of size 
        $\designsumval$ drawn without replacement from a population of 
        size $\Nsens$ \cite{sunter1986solutions}, and is given by
        \begin{equation}
          \pi_{i_1, \ldots, i_k} 
            := \left( \prod_{t=1}^{k} w_{i_t} \right)
            \frac{
               R(\designsumval\!-\!k,\, S\!\setminus\! \{i_1, \ldots, i_k\})
            }{
              R(\designsumval,\, S)
            } \,;\quad i_1, \ldots, i_k\in S=\{1, \ldots, \Nsens\} 
            \,,
        \end{equation}
        where $A\!\setminus\!\{B\}$ is the set difference excluding elements of $B$ from $A$. 
        For example, the first- and second-order inclusion probabilities, 
        respectively, are given by
        \begin{subequations}\label{eqn:first_second_inclusion}
          \begin{align}
            \pi_i &= w_{i}
              \frac{R(\designsumval\!-\!1,\, S\!\setminus\!  \{i\}) }{ R(\designsumval,\, S) } \,; \\ 
            \pi_{i, j} &= w_{i} w_{j}
              \frac{R(\designsumval\!-\!2,\, S\!\setminus\!  \{i, j\}) }{ R(\designsumval,\, S) } \,,\,\,
                i\neq j,\,\, i,j\in S \,.  
          \end{align}
        \end{subequations}
        %

    \subsection{Poisson binomial distribution}
    \label{subsec:Poisson-Binomial}
      %
      The PB model 
      is given by \cite{chen1997statistical} 
      \begin{equation}\label{eqn:Poisson_Binomial_PMF}
        \Prob(\designsum=\designsumval)= 
          \sum_{\substack{B\subseteq A \\ |B|=\designsumval}} 
            \prod_{i \in B} \hyperparam_i 
            \prod_{j \in B^c} (1 \!-\! \hyperparam_j)  
            = R(\designsumval, S) \prod_{i=1}^{\Nsens} (1\!-\!\hyperparam_i) 
            = \frac{R(\designsumval, S) }{\prod_{i=1}^{\Nsens} (1\!+\!w_i) } 
            \,, 
      \end{equation}
      where  
      $
        \designsum 
          :=\wnorm{\design}{0}=\sum_{i=1}^{\Nsens}{\design_i}
      $. 
      The PB model \eqref{eqn:Poisson_Binomial_PMF} is defined only for non-degenerate
      probabilities $\hyperparam_i\in(0,1)$ and is not addressed in the literature for
      degenerate cases $\hyperparam_i\in\{0,1\}$.
      We extend it here to handle degenerate probabilities, which is essential for the
      proposed optimization approach since the boundary values $\hyperparam_i\in\{0,1\}$
      are naturally reached during optimization and require evaluation of both the PMF
      and its one-sided derivatives. 
      To this end, we start with \Cref{lemma:PB_Model_Degenerate}.

      \begin{lemma}\label{lemma:PB_Model_Degenerate}
        The PB model \eqref{eqn:Poisson_Binomial_PMF} satisfies the following identities.
        \begin{subequations}\label{eqn:PB_Model_Degenerate}
          \begin{align}
            \label{eqn:PB_Model_Degenerate_1}
            \CondProb{\designsum\!=\!\designsumval}{\hyperparam_i\!=\!0} 
              &= 
              \frac{
                R(\designsumval, S\setminus\{i\}) 
              }{
                \prod_{\substack{j=1\\j\neq i}}^{\Nsens} (1+w_j)
              } 
              \,,  \\
            \label{eqn:PB_Model_Degenerate_2}
            \CondProb{\designsum\!=\!\designsumval}{\hyperparam_i\!=\!1} 
              &= 
              \frac{
                R(\designsumval-1, S\setminus\{i\}) 
              }{
                \prod_{\substack{j=1\\j\neq i}}^{\Nsens} (1+w_j)
              }
              \,,
              \\
            \label{eqn:PB_Model_Degenerate_3}
            \lim_{\hyperparam_i \rlim 0 } \Prob({\designsum=\designsumval})
              &= \CondProb{\designsum=\designsumval}{\hyperparam_i=0} 
              \,, \\
            \label{eqn:PB_Model_Degenerate_4}
            \lim_{\hyperparam_i \llim 1 } \Prob({\designsum=\designsumval})
              &= \CondProb{\designsum=\designsumval}{\hyperparam_i=1} 
              \,.
          \end{align}
        \end{subequations}
      \end{lemma}
      \begin{proof}
        See \Cref{app:probability_models_proofs}.
      \end{proof}

      \Cref{lemma:PB_Model_Degenerate} shows that 
      the PB model is continuous at
      the corners of the domain of the parameter, 
      that is $\hyperparam \in\{0, 1\}^{\Nsens}$.
      Moreover, it provides a rigorous definition of the PMF for the 
      degenerate case where the weights $w_i$ employed in the R-function in 
      \eqref{eqn:Poisson_Binomial_PMF} 
      are not defined. 
      \Cref{theorem:PB_Model_Full} employs \Cref{lemma:PB_Model_Degenerate} to generalize 
      the PB model to allow for degenerate success probabilities.

      \begin{theorem}\label{theorem:PB_Model_Full}
      \begin{subequations}\label{eqn:PB_Model_Full}
        Consider a Bernoulli random vector $\design\in\{0, 1\}^{\Nsens}$,  
        and let 
        \begin{equation}\label{eqn:indices_set}
          \begin{aligned}
            S &= \{1, \ldots, \Nsens\}\,, \qquad \qquad
            O = \{ i\in S \,|\, \hyperparam_i\!=0\}\,, \\
            I &= \{ i\in S \,|\, \hyperparam_i=1\}\,, \qquad \,
            V = S\setminus\{I\cup O\} \,, 
          \end{aligned}
        \end{equation}
        then the PMF of the PB model over the success probability domain 
        $\hyperparam\in[0, 1]^{\Nsens}$ 
        is 
        \begin{equation}\label{eqn:PB_Model_Full_PMF} 
          \Prob({\designsum=\designsumval}) 
          =
            R(\designsumval-\abs{I}, V) 
              \prod_{\substack{j\in V}} (1-\hyperparam_j)
          =
          \frac{  R(\designsumval-\abs{I}, V) 
              }{\prod_{\substack{j\in V}} (1+w_j)
              }
            \,; \quad
              \designsum 
                =\sum_{i=1}^{\Nsens}\design_i
              \,.
        \end{equation}
      \end{subequations}
      \end{theorem}
      \begin{proof}
        See \Cref{app:probability_models_proofs}.
      \end{proof}

    \subsection{Conditional Bernoulli distribution}
    \label{subsec:Conditional-Bernoulli}
      The PMF of the CB model for non-degenerate success probabilities
      $\hyperparam_i\in(0, 1);\, i=1, \ldots, \Nsens$
      is given by \cite{chen2000general} 
      \begin{equation}\label{eqn:CB_Model_PMF}
        \CondProb{\design}{\designsum\!=\!\designsumval}
          = \frac{
            \prod\limits_{i=1}^{\Nsens}{{w_i}^{\design_i}}
          }{
            \sum\limits_{\substack{\vec{d}\in\{0,1\}^{\Nsens} \,; 
              \wnorm{\vec{d}}{0}=\designsumval}} 
                \prod_{i=1}^{\Nsens}{w_i}^{d_i} 
        }
        = \frac{
            \prod\limits_{i=1}^{\Nsens}{{w_i}^{\design_i}}
          }{
            R(\designsumval, S)
          } 
          \,; 
          \quad
          \designsum
            =\sum_{i=1}^{\Nsens}\design_i 
          \,.
      \end{equation}
      \begin{lemma}\label{lemma:CB_Model_Degenerate}
        The CB model \eqref{eqn:CB_Model_PMF} satisfies the following identities.
        \begin{subequations}
          \begin{align}
            \label{eqn:CB_Model_Degenerate_1}
            \CondProb{\design}{\hyperparam_i\!=\!0;\designsum\!=\!\designsumval} 
              &= 
                \begin{cases}  
                  \frac{
                    \prod_{\substack{j=1\\j\neq i}}^{\Nsens} w_j^{\design_j} 
                  }{
                    R(\designsumval, S\!\setminus\!\{i\}) 
                  } 
                  \!&,\, \design_i\!=\!0
                  \\
                  0 \!&,\, \design_i\!=\!1
                \end{cases} 
                \,, \\
             \label{eqn:CB_Model_Degenerate_2}
            \CondProb{\design}{\hyperparam_i\!=\!1;\designsum\!=\!\designsumval} 
              &= 
                \begin{cases}  
                  0 
                  &,\, \design_i\!=\!0
                  \\
                  \frac{
                    \prod_{\substack{j=1\\j\neq i}}^{\Nsens} w_j^{\design_j} 
                  }{
                    R(\designsumval\!-\!1, S\setminus\{i\}) 
                  } 
                  &,\, \design_i\!=\!1
                \end{cases} \,,
              \\
            \label{eqn:CB_Model_Degenerate_3}
            \lim_{\hyperparam_i \rlim 0 } 
              \CondProb{\design}{\designsum=\designsumval} 
              &= \CondProb{\design}{\hyperparam_i=0;\designsum=\designsumval} 
                \,, \\
            \label{eqn:CB_Model_Degenerate_4}
            \lim_{\hyperparam_i \llim 1 } 
              \CondProb{\design}{\designsum=\designsumval} 
               &= \CondProb{\design}{\hyperparam_i=1;\designsum=\designsumval} 
              \,.
          \end{align}
        \end{subequations}
      \end{lemma}
      \begin{proof}
        See \Cref{app:probability_models_proofs}.
      \end{proof}

      \Cref{theorem:CB_Model} applies  
      \Cref{lemma:CB_Model_Degenerate} recursively to 
      handle degenerate probabilities in a CB model.
      Specifically, $\designsumval$ is reduced by the number of 
      entries in $\hyperparam$ equal to $1$, and 
      $\Nsens$ is reduced by 
      removing degenerate ($0/1$ probability) entries from $\design,\hyperparam$.
      \begin{theorem}\label{theorem:CB_Model}
        The PMF of the CB model over the closed 
        interval $\hyperparam\in[0, 1]^{\Nsens}$ 
        is 
        \begin{equation}\label{eqn:CB_Model_Full_PMF}
          \CondProb{\design}{\designsum\!=\!\designsumval}
            =
            \begin{cases}
              \frac{
                \prod\limits_{j\in V} {w_j}^{\design_j}
                }{
                  R(\designsumval - |I|, V)
                } \quad &;
                \design_j\!=\!\hyperparam_j \,\, \forall j\in \{I\!\cup\! O\},\, \text{and } 
                \sum\limits_{j\in V} \design_j \!=\! \designsumval\!-\!|I| \,,
                \\
              \qquad 0  &; o.w.,
            \end{cases}
        \end{equation}
        where 
        $
          \designsum 
          =\sum_{i=1}^{\Nsens}{\design_i},\,  
          w_j = \frac{\hyperparam_j}{1-\hyperparam_j}
        $, 
        and 
        $S,\, O,\, I,\, V$ are given by \eqref{eqn:indices_set}. 
      \end{theorem}
      \begin{proof}
        See \Cref{app:probability_models_proofs}.
      \end{proof}
      \begin{proposition}\label{proposition:CB_Model_moments}
        Let $\design\sim\CondProb{\design}{\designsum=\designsumval}$ be modeled by the CB model
        \eqref{eqn:CB_Model_Full_PMF}, 
        then
        \begin{equation} \label{eqn:CB_Model_moments_1}
          \Expect{}{\design_{i}} = \pi_i 
            \,; \,\,\,
            \Expect{}{ \design_{i} \design_{j}}
              =  \begin{cases} 
                \pi_{i}& ; i\!=\!j \\
                \pi_{ij}&; i\!\neq\! j 
              \end{cases}
            \,; \,\,\,
            \brCov{\design_{i}, \design_{j}}
              =  \begin{cases} 
                \pi_{i} - \pi_{i}^2&; i\!=\!j \\
                \pi_{i, j} - \pi_{i}\pi_{j}&; i\!\neq\! j 
              \end{cases} \,,
        \end{equation}
        where $\pi$ and $\pi_{i, j}$ are the first- and the second-order
        inclusion probabilities \eqref{eqn:first_second_inclusion}.
      \end{proposition}
      \begin{proof}
        See \Cref{app:probability_models_proofs}.
      \end{proof}

    \subsection{Generalized conditional Bernoulli  model}
    \label{subsec:Generalized-Conditional-Bernoulli}
      \Cref{theorem:CB-Multiple-Sums} introduces the generalized CB (GCB) 
      model which replaces the scalar $\designsumval$ in 
      the CB model 
      \eqref{eqn:CB_Model_Full_PMF} with a set of admissible values
      $\designsum\in\designsumset:=\{\designsumval_{1},\ldots,\designsumval_{m}\}$.
      \Cref{theorem:CB-Multiple-Sums} enables 
      evaluating the PMF and the moments of the GCB model by employing
      the PB and the CB models.

      \begin{theorem}\label{theorem:CB-Multiple-Sums} 
        Consider the set 
        $
          \designsumset:=\{\designsumval_{1}, \ldots, \designsumval_{m}\}
            \subseteq \{0, 1, \ldots, \Nsens\} .
        $
        Then,
        \begin{subequations}\label{eqn:GCB_Model-identities}\allowdisplaybreaks
          \begin{align} 
            \label{eqn:GCB_Model_PMF}
            \CondProb{\design}{\designsum\in\designsumset}
              &= \frac{
                1
                }{
                \sum\limits_{\designsumval\in \designsumset} 
                  \Prob({\designsumval})
                }
                \sum\limits_{\designsumval\in \designsumset} 
                  \CondProb{\design}{\designsum=\designsumval} 
                    \Prob({\designsumval})
                \,,  \\
            \label{eqn:GCB_Model-Function-Expect}
            \CondExpect{f(\design)}{\designsumset}
              &=
              \frac{1}{
                \sum\limits_{\designsumval\in \designsumset} 
                  \Prob({\designsumval})
                }
                \sum\limits_{\designsumval\in \designsumset} 
                  \CondExpect{f(\design)}{\designsumval}
                  \Prob({\designsumval})  \,, \\
            \label{eqn:GCB_Model-Function-Var}
            \CondVar{f(\design)}{\designsumset}
              &\!=
                  \frac{
                    \sum\limits_{\designsumval\in \designsumset}\! 
                      \CondExpect{f(\design)^2}{\designsumval}
                      \Prob({\designsumval})  
                  }{
                    \sum\limits_{\designsumval\in \designsumset} 
                      \Prob({\designsumval})
                  } 
                  -
                  \frac{
                    \sum\limits_{i=1}^{m} 
                    \sum\limits_{j=1}^{m} \!
                      \CondExpect{f(\design)}{\designsumval_i}
                      \CondExpect{f(\design)}{\designsumval_j}
                      \Prob({\designsumval_i}) 
                      \Prob({\designsumval_j}) 
                  }{
                    \left(
                      \sum\limits_{\designsumval\in \designsumset} 
                        \Prob({\designsumval})
                    \right)^2
                  }  
            ,
          \end{align}
        \end{subequations}
        where 
        $f:\{0, 1\}^{\Nsens}\rightarrow \Rnum$, 
        $\CondProb{\design}{\designsumset} \equiv \CondProb{\design}{\designsum\in\designsumset}$, 
        and 
        $ \Prob({\designsumval})\equiv \Prob({\designsum=\designsumval})$ and  
        $ \CondProb{\design}{\designsumval} \equiv \CondProb{\design}{\designsum=\designsumval}$ are the PMFs of the 
        PB \eqref{eqn:PB_Model_Full} 
        and the 
        CB \eqref{eqn:CB_Model_Full_PMF} 
        models, respectively.
      \end{theorem}
      \begin{proof}
        See \Cref{app:probability_models_proofs}.
      \end{proof}
      
      \subsection{Gradient information} 
      \label{subsec:probability_models_gradients}
        Here we describe the derivatives of the probability
        models in \Cref{sec:probability_models} with respect to their parameters.
        These derivatives yield the score function
        $\nabla_{\hyperparam}\log\CondProb{\design}{\cdot}$,
        which is the key quantity required for evaluating the stochastic gradient
        of the probabilistic objective in \Cref{sec:probabilistic_optimization}.

        \begin{proposition}\label{proposition:probability_models_derivatives}
          The models PB \eqref{eqn:PB_Model_Full}, 
          CB \eqref{eqn:CB_Model_Full_PMF}, and 
          GCB \eqref{eqn:GCB_Model_PMF} 
          satisfy,
          \begin{subequations}\label{eqn:probability_models_derivatives}\allowdisplaybreaks

            \begin{equation} \label{eqn:PB_Model_Full_Gradient} 
              \del{\Prob({\designsum\!=\!\designsumval})}{\hyperparam_i}
                =
                \begin{cases}
                  \left(
                      R\left(z\!-\!\abs{I}\!-\!1, V\right) 
                      -  R\left(z\!-\!\abs{I}, V\right) 
                    \right)
                    \prod\limits_{\substack{j\in V}} (1\!-\!\hyperparam_j)
                  &; \,  0 \leq \hyperparam_i < 1 \,,  \\
                  \left(
                      R\left(z\!-\!\abs{I}\!, V\right) 
                      -  R\left(z\!-\!\abs{I}+\!1, V\right) 
                  \right)
                  \prod\limits_{\substack{j\in V}} (1\!-\!\hyperparam_j)
                  \,&; \,  \hyperparam_i = 1 \,,  \\
                \end{cases}
                \,,
            \end{equation}

            \begin{equation}\label{eqn:CB_Model_Full_Gradient}
              \del{
                  \CondProb{\design}{\designsum\!=\!\designsumval}
                }{\hyperparam_i}
                =
                \begin{cases}
                  \CondProb{\design}{\designsum\!=\!\designsumval}
                    \frac{(1+w_i)^2}{w_i}
                    \Bigl(
                      \design_i
                      \!-\!
                      \frac{
                        w_i \, R(\designsumval - |I|-1, V\setminus \{i\} )
                      }{
                        R(\designsumval - |I|, V\setminus \{i\})
                      }
                    \Bigr)
                  &; \, 0 < \hyperparam_i < 1 \,,  \\[1.0em]
                  \left.
                  \begin{array}{ll}
                    \frac{
                      - R(\designsumval - |I|-1, V)
                    }{
                      \left(
                        R(\designsumval - |I|, V)
                      \right)^2
                    }
                      \prod\limits_{j\in V} w_j^{\design_j}
                    \,\,&; \design_i=0 \\[1.0em]
                    \frac{
                        1
                      }{
                          R(\designsumval - |I|, V)
                      }
                      \prod\limits_{j\in V} w_j^{\design_j}
                    \,\,&;  \design_i=1
                  \end{array}
                  \right\}
                  &; \,  \hyperparam_i = 0 \,,  \\[1.0em]
                  \left.
                  \begin{array}{ll}
                    \frac{
                      -1
                    }{
                        R(\designsumval - |I|, V)
                    }
                    \prod\limits_{j\in V} w_j^{\design_j}
                    \,\,&; \design_i=0 \\[1.0em]
                    \frac{
                      R(\designsumval - |I|+1, V)
                    }{
                      \left(
                        R(\designsumval - |I|, V)
                      \right)^2
                    }
                      \prod\limits_{j\in V} w_j^{\design_j}
                    \,\,&;  \design_i=1
                  \end{array}
                  \right\}
                  &; \, \hyperparam_i = 1 \,,
                \end{cases}
                \,, 
            \end{equation}

            \begin{equation} \label{eqn:GCB_Model-LogProb-Grad}
            \nabla_{\hyperparam} \CondProb{\design}{\designsum\in\designsumset}
              =\CondProb{\design}{\designsum\in\designsumset} \left(
                \frac{
                  \sum\limits_{\designsumval\in \designsumset}
                    \CondProb{\design}{\designsum=\designsumval}
                    \nabla_{\hyperparam}\Prob({\designsumval})
                  \!+\!
                  \sum\limits_{\designsumval\in \designsumset}
                    \Prob({\designsumval})
                    \nabla_{\hyperparam}\CondProb{\design}{\designsum=\designsumval}
                }{
                  \sum\limits_{\designsumval\in \designsumset}
                    \CondProb{\design}{\designsum=\designsumval}
                    \Prob({\designsumval})
                }  
                -
                \frac{
                  \sum\limits_{\designsumval\in \designsumset} 
                  \nabla_{\hyperparam} \Prob({\designsumval})
                }{
                  \sum\limits_{\designsumval\in \designsumset} 
                    \Prob({\designsumval})
                } 
              \right)
              \,. 
            \end{equation}

          \end{subequations}

        \end{proposition}
        \begin{proof}
          See \Cref{app:probabilistic_optimization_proofs}.
        \end{proof}

      \subsection{Sampling} \label{subsec:sampling}
        Here we discuss sampling from the probability models in
        \Cref{sec:probability_models}.
        Efficient sampling is essential for the proposed approach, as each iteration
        of \Cref{alg:probabilistic_binary_optimization} requires drawing a batch of
        feasible binary designs from the current CB or GCB model to estimate the
        stochastic gradient.
        \subsubsection{Sampling the PB model.}
        \label{subsubsec:Poisson_binomial_sampling}
        The PB model \eqref{eqn:PB_Model_Full}  
        is sampled by weighted random sampling 
        $\designsum\in\{0, \ldots, \Nsens\}$, with replacement, with weights 
        \eqref{eqn:PB_Model_Full}.
        \subsubsection{Sampling the CB model.}
          \label{subsubsec:CB_Model_sampling}
          For completeness, here we summarize an approach for sampling the CB model; 
          see \cite[Appendix A]{heng2020simple} as described by \Cref{alg:CB_Model_sampling}.  

          \begin{algorithm}[htbp!]
            \caption{
              Generate a sample from the CB 
              model \eqref{eqn:CB_Model_Full_PMF}.
            }\label{alg:CB_Model_sampling}
            
            \begin{algorithmic}[1] 
           
              \Require{
                Distribution parameters 
                $\hyperparam,\, \designsumval$, and sample size $\Nens$.
              }
              \Ensure{
                A sample 
                $ \{
                    \design^{(i)}\in\{0, 1\}^{\Nsens}
                      \sim \CondProb{\design}{\hyperparam;\,\designsum\!=\!\designsumval}
                    | i=1, \ldots, \Nens
                \}
                $ drawn from \eqref{eqn:CB_Model_Full_PMF}
              }

              \State Calculate the matrix $q$  using \eqref{eqn:q_matrix}

              \State Initialize a sample $S=\{\}$.

              \For {$i$ $\gets 1$ to $\Nens$}
                
                \State Initialize $\design^{(i)}\in\{0, 1\}^{\Nsens}$ 
                
                \For {$j$ $\gets 1$ to $\Nsens$}
                  \State Calculate $p_j$ using \eqref{eqn:CB_sequential_prob}
                  \State Sample a uniform random value $u_i \sim \mathcal{U}(0, 1)$ 
                  \State \textbf{if} {$u_j >= p_j$} \textbf{then} {$\design^{(i)}_j \gets 1$} \textbf{else} {$\design^{(i)}_j \gets 0$}
                \EndFor

                \State Update $S \gets S \cup \{\design^{(i)}\}$
              \EndFor
         
              \State \Return {$S$ }
              
            \end{algorithmic}
          \end{algorithm}
          \Cref{alg:CB_Model_sampling} samples 
          $\design_i, \, i=1, \ldots, \Nsens$ sequentially 
          by employing the inclusion probability of the $j$th entry, 
          \begin{subequations}
            \begin{equation} 
              \CondProb{\design_j=1 }{\design_1, \ldots, \design_{j-1},\, \designsum=\designsumval} 
                = 
                \frac{
                  q(\designsumval-\sum_{m=1}^{j-1}\design_{m}  \,,\,\,  j+1) \, \hyperparam_j
                }{
                  q(\designsumval+1-\sum_{m=1}^{j-1}\design_{m} \,,\,\,  j)
                  }\,,\label{eqn:CB_sequential_prob}  
              \end{equation}
              where for 
              $
                i = 1, \ldots, \designsumval+1 
              $, 
              $
                j = 1, \ldots, \Nsens 
              $, the matrix valued function $q(i, j)$ is given by
              \begin{equation}\label{eqn:q_matrix}
              q(i, j) 
                = 
                \begin{cases}
                  \prod_{m=j}^{\Nsens} (1-\hyperparam_m) &\text{if }  i=1\,, \\
                  \hyperparam_{\Nsens}  &\text{if } i=2, j=\Nsens \,, \\
                  0   &\text{if }  i >\Nsens - j +2 \,,\\ 
                  \hyperparam_j q(i-1, j+1) + (1-\hyperparam_j) \, q(i, j+1)  &\text{otherwise}\,,
                \end{cases} ;
            \end{equation}
          \end{subequations}
        \subsubsection{Sampling the GCB model.}
          \label{subsubsec:GCB_Model_sampling}
          Generating a random sample of size $\Nens$ from the generalized
          CB model \eqref{eqn:GCB_Model_PMF} is described by \Cref{alg:GCB_Model_sampling},
          which proceeds in two steps.
          First, a set of $\Nens$ realizations of the parameter 
          $\designsum$ is drawn with replacement from the set 
          $\designsumset:=\{\designsumval_{1},  \ldots, \designsumval_{m}\}$
          by using weighted random sampling 
          with weights equal to the probabilities dictated by 
          \eqref{eqn:Poisson_Binomial_PMF}, 
          that is, 
          $
            \{
              \CondProb{\designsum=\designsumval_{1}}{\hyperparam}, 
              \ldots, 
              \CondProb{\designsum=\designsumval_{m}}{\hyperparam}
            \}
          $. 
          Second, for each realization $\designsum=\designsumval\in\designsumset$,
          use \Cref{alg:CB_Model_sampling} to 
          draw a sample from the CB model with parameters 
          $\hyperparam,\, \designsumval$. 

          \begin{algorithm}[htbp!]
            \caption{
              Generate a sample from the GCB model \eqref{eqn:GCB_Model_PMF}.
            }\label{alg:GCB_Model_sampling}
            
            \begin{algorithmic}[1] 
           
              \Require{
                Distribution parameters 
                $\hyperparam,\, \designsumset:=\{\designsumval_{1},  \ldots, \designsumval_{m}\}$, 
                and sample size $\Nens$.
              }
              \Ensure{
                A sample 
                $ \{
                    \design^{(i)}\in\{0, 1\}^{\Nsens}
                      \sim \CondProb{\design}{\hyperparam;\,\designsum\!\in\!\designsumset}
                    | i=1, \ldots, \Nens
                \}
                $ drawn from \eqref{eqn:GCB_Model_PMF}
              }

              \State Calculate weights/probabilities 
              $
                W = \{\CondProb{\designsumval_i}{\hyperparam} |i=1,\ldots,m\}
              $ using \eqref{eqn:Poisson_Binomial_PMF}
              
              \State Sample 
              $ 
                \widehat{\designsumset}=\{\widehat{\designsumval}_i|i=1, \ldots, \Nens\}
              $ with replacement from $\designsumset$ using weighted random sampling with weights $W$

              \State Extract unique sample sizes 
                $\widetilde{\designsumset}=\{\widetilde{\designsumval}_i\in\widehat{\designsumset}|i=1, \ldots, r\}$
                and calculate 
                $N:=\{N_{i}|i=1, \ldots, r\}$, the number of times each $\widetilde{\designsumval}_i$ 
                appears in $\widehat{\designsumset}$ 
              
              \State Initialize a sample $S=\{\}$.

              \For {$i$ $\gets 1$ to $r$}

                  \State Generate a sample $\widetilde{S}$ of size $N_i$ from \eqref{eqn:CB_Model_Full_PMF}
                    with parameters $\hyperparam$, $\widetilde{\designsumval}_i$
                    \Comment{Use \Cref{alg:CB_Model_sampling}}
              
                  \State Update $S \gets S \cup \widetilde{S}$
              \EndFor

              \State \Return {$S$ }
              
            \end{algorithmic}
          \end{algorithm}

  \section{Probabilistic Black-Box Binary Optimization} 
    \label{sec:probabilistic_optimization}
    Here we describe our fully probabilistic approach for solving 
    \eqref{eqn:binary_optimization_of_interest}.
    \Cref{subsec:probabilistic_optimization_equality} 
    considers 
    budget 
    equality constraint, and 
    \Cref{subsec:probabilistic_optimization_inclusion} 
    addresses the case of budget inclusion 
    (e.g., inequality) constraint.
    For clarity, the proofs 
    are provided in \Cref{app:probabilistic_optimization_proofs}.

    \subsection{Budget equality constraint}
    \label{subsec:probabilistic_optimization_equality}
      Here we consider binary optimization problems with a hard equality budget constraint,
      i.e., the design must activate \emph{exactly} $\designsumval$ out of $\Nsens$ candidates.
      This is the canonical sensor placement setting where the number of sensors is fixed.
      We replace the binary problem with a stochastic optimization over the parameters
      of the CB model (\Cref{subsec:Conditional-Bernoulli}),
      whose support is restricted to the feasible set $\{\design : \wnorm{\design}{0} = \designsumval\}$
      by construction.
      Specifically, the problem of interest is:
      \begin{equation}\label{eqn:binary_optimization_budget_equality_constraint}
        \design\opt \in \argmax_{\design}{\obj(\design) }
          \quad \textrm{s.t.} \quad   
            \design \in \{0, 1\}^{\Nsens} \,, \quad \wnorm{\design}{0} = \designsumval
            \,.
      \end{equation}

      We view $\design$ as a random variable following the CB model \eqref{eqn:CB_Model_Full_PMF}, 
      and replace \eqref{eqn:binary_optimization_budget_equality_constraint} 
      with the following stochastic optimization problem:
      \begin{subequations}\label{eqn:overall_probabilistic_equality}
      \begin{equation}\label{eqn:probabilistic_optimization_budget_equality_constraint}
        \hyperparam\opt \in \argmax_{\hyperparam \in [0, 1]^{\Nsens} }
          \stochobj(\hyperparam)
        := \Expect{
             \design\sim\CondProb{\design}{\designsum=\designsumval}
           }{\obj(\design)}
        \,,
      \end{equation}
      whose gradient (derivative with respect to the distribution parameter $\hyperparam$) 
      is given by
      \begin{equation}\label{eqn:budget_equality_constraint_exact_gradient}
        \begin{aligned}
          \vec{g}(\hyperparam) 
            &:= 
            \nabla_{\hyperparam}\, 
              \Expect{
                \design\sim \CondProb{\design}{\designsum=\designsumval}
              }{\obj(\design)} 
            = \Expect{
                \design\sim \CondProb{\design}{\designsum=\designsumval}
              }{
                \obj(\design) 
                \nabla_{\hyperparam} \log \CondProb{\design}{\designsum=\designsumval} 
              } 
            \,,
        \end{aligned}
      \end{equation}
      where independence of $\obj$ from $\hyperparam$, 
      linearity of the expectation, 
      and the log-derivative trick
      ($\nabla_{\hyperparam}\Prob = \Prob \nabla_{\hyperparam}\log\Prob$) are 
      used to obtain the last term.
      To solve \eqref{eqn:probabilistic_optimization_budget_equality_constraint} numerically, we employ 
      the following stochastic approximation 
      of the gradient \eqref{eqn:budget_equality_constraint_exact_gradient}, 
      \begin{equation}\label{eqn:budget_equality_constraint_stochastic_gradient}
        \vec{g}(\hyperparam) 
        \approx
        \widehat{\vec{g}}(\hyperparam)
          := \frac{1}{\Nens} \sum_{k=1}^{\Nens} \obj(\design[k]) \,
            \frac{
          \nabla_{\hyperparam} \CondProb{\design[k]}{\designsum=\designsumval}
        }{
          \CondProb{\design[k]}{\designsum=\designsumval} 
        }
      \end{equation}
      where 
      $
      \{
        \design[k]
          \sim\CondProb{\design}{\designsum\!=\!\designsumval}
            ;\, k\!=\!1, \ldots \Nens
      \} 
      $
      is a sample drawn from 
      the CB model 
      \eqref{eqn:CB_Model_Full_PMF}, and 
      $\nabla_{\hyperparam} 
        \CondProb{\design[k]}{\designsum=\designsumval}
      $ is given (elementwise) by \eqref{eqn:CB_Model_Full_Gradient}. 
      The stochastic gradient
      \eqref{eqn:budget_equality_constraint_stochastic_gradient} is unbiased.
      However, due to high variability of the stochastic gradient 
      \eqref{eqn:budget_equality_constraint_stochastic_gradient}, 
      in practice we always use the variance-reduced, baselined form
      \begin{equation}\label{eqn:budget_equality_constraint_stochastic_gradient_with_basesline}
        \widehat{\vec{g}}^{\rm b}
          = \frac{1}{\Nens} \sum_{k=1}^{\Nens} \left(\obj(\design[k])-\baseline\opt\right) \,
            \nabla_{\hyperparam} \log \CondProb{\design[k]}{\designsum=\designsumval}
        \,,
      \end{equation}
      where $\baseline\opt$ is the optimal baseline defined and derived in
      \Cref{subsec:variance_reduction}.
      
    \end{subequations}

    \subsection{Budget inclusion constraint}
    \label{subsec:probabilistic_optimization_inclusion}
      Here we consider the more general case in which the budget is allowed to
      take any value within a prescribed set $\designsumset \subseteq \{0, 1, \ldots, \Nsens\}$
      rather than being fixed to a single value.
      This includes inequality constraints (e.g., ``at most $\designsumval$ sensors'')
      and scenarios in which multiple budget levels are permissible.
      The probability model used here is the GCB model
      (\Cref{subsec:Generalized-Conditional-Bernoulli}), whose support spans all 
      feasible designs across
      all admissible budget values in $\designsumset$.
      Formally, the problem of interest is: 
      \begin{equation}\label{eqn:binary_optimization_budget_inclusion_constraint}
        \design\opt \in \argmax_{\design}{\obj(\design) }
          \quad \textrm{s.t.} \quad   
            \design \in \{0, 1\}^{\Nsens} \,, \quad 
            \wnorm{\design}{0} 
              \in\designsumset:=\{\designsumval_{1},\ldots,\designsumval_{m}\}
            \,.
      \end{equation}

      \begin{subequations}\label{eqn:overall_probabilistic_inclusion}
        Similar to the case of equality constraint discussed 
        in \Cref{subsec:probabilistic_optimization_equality}, 
        we replace the optimization 
        problem \eqref{eqn:binary_optimization_budget_inclusion_constraint} 
        with the following stochastic optimization problem:
        \begin{equation}\label{eqn:probabilistic_optimization_budget_inclusion_constraint}
          \hyperparam\opt \in \argmax_{\hyperparam \in [0, 1]^{\Nsens} }
            \stochobj(\hyperparam)
          := \Expect{
               \design\sim\CondProb{\design}{\designsum\in\designsumset}
             }{\obj(\design)}
          \,,
        \end{equation}
        with exact gradient $\vec{g}(\hyperparam) $ and a stochastic approximation 
        $\widehat{\vec{g}}(\hyperparam)$, respectively, given by 
        \begin{align}
          \vec{g}(\hyperparam) 
            &:= 
            \nabla_{\hyperparam}\, 
              \Expect{
                \design\sim \CondProb{\design}{\designsum\in\designsumset}
              }{\obj(\design)} 
            = \Expect{
                \design\sim \CondProb{\design}{\designsum\in\designsumset}
              }{
                \obj(\design) 
                \nabla_{\hyperparam} \log \CondProb{\design}{\designsum\in\designsumset} 
              } 
            \,,  \label{eqn:budget_equality_constraint_inclusion_gradient}  \\
          \widehat{\vec{g}}(\hyperparam)
            &= \frac{1}{\Nens} \sum_{k=1}^{\Nens} \obj(\design[k]) \,
              \frac{
                \nabla_{\hyperparam} \CondProb{\design[k]}{\designsum\in\designsumset}
              }{
                \CondProb{\design[k]}{\designsum\in\designsumset}
              }
              \,;  \quad
          \design[k]
            \sim\CondProb{\design}{\designsum\in\designsumset}
              \,,  \label{eqn:budget_inclusion_constraint_stochastic_gradient}
        \end{align}
        with
        $\nabla_{\hyperparam} \CondProb{\design[k]}{\designsum\in\designsumset}$
        given by \eqref{eqn:GCB_Model-LogProb-Grad}
        and
        $\CondProb{\design[k]}{\designsum\in\designsumset}$
        given by \eqref{eqn:GCB_Model_PMF}.
        Similar to \eqref{eqn:budget_equality_constraint_stochastic_gradient_with_basesline}, 
        in practice we always use the variance-reduced, baselined form
        \begin{equation}
          \label{eqn:budget_inclusion_constraint_stochastic_gradient_with_baseline}
          \widehat{\vec{g}}^{\rm b}
            = \frac{1}{\Nens} \sum_{k=1}^{\Nens} \left(\obj(\design[k])-\baseline\opt\right) \,
              \nabla_{\hyperparam} \log \CondProb{\design[k]}{\designsum\in\designsumset}
              \,,
        \end{equation}
        where $\baseline\opt$ is the optimal baseline defined in \Cref{subsec:variance_reduction}.

      \end{subequations}

  \subsection{Variance reduction of the stochastic gradient}
  \label{subsec:variance_reduction}

    The stochastic gradient estimators
    \eqref{eqn:budget_equality_constraint_stochastic_gradient} and
    \eqref{eqn:budget_inclusion_constraint_stochastic_gradient}
    are unbiased but may exhibit high variance, particularly early in
    optimization when the Bernoulli parameters $\hyperparam$ are far from
    degenerate.
    Subtracting a constant baseline $\baseline$ from $\obj(\design[k])$ in the
    estimator does not introduce any bias, because the score function has zero mean:
    \begin{equation}\label{eqn:score_zero_mean}
      \Expect{\design \sim \CondProb{\design}{\cdot}}{
        \nabla_{\hyperparam} \log \CondProb{\design}{\cdot}
      } = \vec{0} \,,
    \end{equation}
    where $\CondProb{\design}{\cdot}$ denotes either the CB model
    $\CondProb{\design}{\designsum=\designsumval}$ or the GCB model
    $\CondProb{\design}{\designsum\in\designsumset}$.
    This identity holds for any properly normalised probability model and is the
    fundamental property that makes baseline subtraction variance-reducing without
    introducing bias.
    Consequently, both baselined gradient estimators
    \eqref{eqn:budget_equality_constraint_stochastic_gradient_with_basesline} and
    \eqref{eqn:budget_inclusion_constraint_stochastic_gradient_with_baseline}
    are unbiased for any choice of $\baseline$.

    \begin{proposition}[Optimal scalar baseline]\label{prop:optimal_scalar_baseline}
      Among all constant baselines $\baseline \in \R$, the one that minimises
      the total variance
      $\brVar{\widehat{\vec{g}}^{\rm b}}
       = \sum_{i=1}^{\Nsens} \Var\!\left(\widehat{g}^{\rm b}_i\right)$
      of the baselined gradient estimator is
      \begin{equation}\label{eqn:optimal_scalar_baseline_theory}
        \baseline\opt
          = \frac{
              \Expect{}{
                \obj(\design)\,
                \bigl\|
                  \nabla_{\hyperparam} \log \CondProb{\design}{\cdot}
                \bigr\|^2
              }
            }{
              \Expect{}{
                \bigl\|
                  \nabla_{\hyperparam} \log \CondProb{\design}{\cdot}
                \bigr\|^2
              }
            } \,,
      \end{equation}
      where $\CondProb{\design}{\cdot}$ is either the CB or GCB model.
      The denominator equals $\Trace{\FisherI(\hyperparam)} > 0$
      for non-degenerate $\hyperparam$, where
      $\FisherI(\hyperparam) = \Expect{}{\vec{s}(\design)\,\vec{s}(\design)\tran}$
      is the Fisher information matrix of the probability model
      $\CondProb{\design}{\cdot}$ with respect to $\hyperparam$
      (not to be confused with the inference parameter $\iparam$ or the
      design-space Fisher information).
      The sample estimator of \eqref{eqn:optimal_scalar_baseline_theory} is
      \begin{equation}\label{eqn:optimal_scalar_baseline_sample}
        \widehat{\baseline}\opt
          = \frac{
              \sum_{k=1}^{\Nens}
                \obj(\design[k])\,
                \bigl\|\nabla_{\hyperparam}\log\CondProb{\design[k]}{\cdot}\bigr\|^2
            }{
              \sum_{k=1}^{\Nens}
                \bigl\|\nabla_{\hyperparam}\log\CondProb{\design[k]}{\cdot}\bigr\|^2
            } \,.
      \end{equation}
    \end{proposition}

    \begin{proof}
      Write $\vec{s}(\design) := \nabla_{\hyperparam}\log\CondProb{\design}{\cdot}$
      for the score vector and $s_i(\design)$ for its $i$-th component.
      For a single sample ($\Nens=1$), the variance of the $i$-th component of
      the baselined estimator is
      \begin{align*}
        \Var\!\bigl((\obj(\design)-\baseline)\,s_i(\design)\bigr)
        &= \Expect{}{\bigl(\obj(\design)-\baseline\bigr)^2 s_i(\design)^2}
           - \Bigl(\Expect{}{(\obj(\design)-\baseline)\,s_i(\design)}\Bigr)^2.
      \end{align*}
      Since $\Expect{}{s_i(\design)} = 0$ by \eqref{eqn:score_zero_mean},
      the second term equals $\bigl(\Expect{}{\obj(\design)\,s_i(\design)}\bigr)^2$,
      which does not depend on $\baseline$.
      Minimising over $\baseline$ is therefore equivalent to minimising
      \[
        V(\baseline)
          := \sum_{i=1}^{\Nsens}
             \Expect{}{\bigl(\obj(\design)-\baseline\bigr)^2 s_i(\design)^2}
           = \Expect{}{\bigl(\obj(\design)-\baseline\bigr)^2
               \bigl\|\vec{s}(\design)\bigr\|^2}.
      \]
      Expanding and differentiating with respect to $\baseline$:
      \[
        \frac{dV}{d\baseline}
          = -2\,\Expect{}{(\obj(\design)-\baseline)\,\bigl\|\vec{s}(\design)\bigr\|^2}
          = -2\,\Expect{}{\obj(\design)\,\bigl\|\vec{s}\bigr\|^2}
            + 2\baseline\,\Expect{}{\bigl\|\vec{s}\bigr\|^2}.
      \]
      Setting to zero and solving gives \eqref{eqn:optimal_scalar_baseline_theory}.
      This is a minimum because $d^2V/d\baseline^2 = 2\,\Expect{}{\|\vec{s}\|^2} > 0$
      for non-degenerate $\hyperparam$.
      The denominator in \eqref{eqn:optimal_scalar_baseline_theory} satisfies
      $\Expect{}{\|\vec{s}(\design)\|^2}
        = \Trace{\FisherI(\hyperparam)}$,
      since $\Expect{}{\vec{s}} = \vec{0}$ implies
      $\Expect{}{\|\vec{s}\|^2} = \Trace{\Var(\vec{s}})
       = \Trace{\FisherI(\hyperparam)}$.
      For the CB model, this trace admits the closed-form expression from
      \Cref{proposition:CB_Model_gradient_moments},
      $\Trace{\FisherI(\hyperparam)}
       = \sum_{i=1}^{\Nsens} \frac{(1+w_i)^4}{w_i^2}(\pi_i - \pi_i^2)$,
      which can replace the sample estimate for improved stability at small $\Nens$.
      When all probabilities are degenerate, $\Trace{\FisherI(\hyperparam)}=0$
      and the baseline is set to $0$.
      The sample estimator \eqref{eqn:optimal_scalar_baseline_sample} is obtained
      by replacing both expectations in \eqref{eqn:optimal_scalar_baseline_theory}
      with their sample averages over $\{\design[k]\}_{k=1}^{\Nens}$.
    \end{proof}

    \begin{remark}\label{remark:baseline_denom_correction}
      The estimator \eqref{eqn:optimal_scalar_baseline_sample} is a weighted
      average of $\obj(\design[k])$ with non-negative weights
      $\|\vec{s}(\design[k])\|^2$, so $\widehat{\baseline}\opt$ always lies
      in the convex hull of the sampled objective values.
      A naive estimator that uses the sample mean of the score in the denominator
      collapses to zero for large $\Nens$, since
      $\frac{1}{\Nens}\sum_k \vec{s}(\design[k]) \to \vec{0}$
      by the zero-mean identity \eqref{eqn:score_zero_mean};
      the estimator \eqref{eqn:optimal_scalar_baseline_sample} avoids this
      by using $\|\vec{s}\|^2$ (always non-negative) rather than $\vec{s}$ itself.
    \end{remark}

    \begin{proposition}[Optimal per-component baseline]\label{prop:optimal_percomponent_baseline}
      For each coordinate $j = 1,\ldots,\Nsens$, the constant $\baseline_j \in \R$
      that independently minimises $\Var\!\left(\widehat{g}^{\rm b}_j\right)$ is
      \begin{equation}\label{eqn:optimal_percomponent_baseline_theory}
        \baseline\opt_j
          = \frac{
              \Expect{}{
                \obj(\design)\,
                \bigl(
                  \partial_{p_j}\log\CondProb{\design}{\cdot}
                \bigr)^2
              }
            }{
              \Expect{}{
                \bigl(
                  \partial_{p_j}\log\CondProb{\design}{\cdot}
                \bigr)^2
              }
            },
          \quad j = 1, \ldots, \Nsens \,,
      \end{equation}
      where the denominator is the $j$-th diagonal entry of $\FisherI(\hyperparam)$,
      strictly positive for non-degenerate $\hyperparam_j$.
      The per-component baseline vector
      $\vec{\baseline}\opt = (\baseline\opt_1, \ldots, \baseline\opt_{\Nsens})\tran$
      satisfies
      \begin{equation}\label{eqn:percomponent_dominates_scalar}
        \sum_{j=1}^{\Nsens}
          \Var\!\left(\widehat{g}^{\rm b}_{j};\,\baseline\opt_j\right)
        \;\leq\;
        \sum_{j=1}^{\Nsens}
          \Var\!\left(\widehat{g}^{\rm b}_{j};\,\baseline\opt\right) \,,
      \end{equation}
      where $\baseline\opt$ is given by \eqref{eqn:optimal_scalar_baseline_theory},
      with equality if and only if all $\baseline\opt_j$ are equal.
      The sample estimator derived from \eqref{eqn:optimal_scalar_baseline_theory} is
      \begin{equation}\label{eqn:optimal_percomponent_baseline_sample}
        \widehat{\baseline}\opt_j
          = \frac{
              \sum_{k=1}^{\Nens}
                \obj(\design[k])\,
                \bigl(\partial_{p_j}\log\CondProb{\design[k]}{\cdot}\bigr)^2
            }{
              \sum_{k=1}^{\Nens}
                \bigl(\partial_{p_j}\log\CondProb{\design[k]}{\cdot}\bigr)^2
            } \,,
          \quad j = 1, \ldots, \Nsens \,.
      \end{equation}
    \end{proposition}

    \begin{proof}
      Write $s_j(\design) := \partial_{p_j}\log\CondProb{\design}{\cdot}$.
      The $j$-th component of the baselined gradient estimator (for $\Nens=1$) is
      $(\obj(\design) - \baseline_j)\,s_j(\design)$.
      Its variance is
      \begin{align*}
        \Var\!\bigl((\obj(\design)-\baseline_j)\,s_j(\design)\bigr)
        &= \Expect{}{(\obj(\design)-\baseline_j)^2\,s_j(\design)^2}
           - \bigl(\Expect{}{(\obj(\design)-\baseline_j)\,s_j(\design)}\bigr)^2.
      \end{align*}
      As noted in the proof of \Cref{prop:optimal_scalar_baseline}, 
      the second term equals $(\partial_{p_j}\stochobj)^2$ and is independent of
      $\baseline_j$, so minimising over $\baseline_j$ reduces to minimising
      \[
        V_j(\baseline_j)
          := \Expect{}{(\obj(\design)-\baseline_j)^2\,s_j(\design)^2}.
      \]
      Differentiating:
      \[
        \frac{dV_j}{d\baseline_j}
          = -2\,\Expect{}{(\obj(\design)-\baseline_j)\,s_j(\design)^2}
          = -2\,\Expect{}{\obj(\design)\,s_j^2}
            + 2\baseline_j\,\Expect{}{s_j^2}.
      \]
      Setting to zero gives \eqref{eqn:optimal_percomponent_baseline_theory}.
      This is a minimum since $d^2 V_j/d\baseline_j^2 = 2\Expect{}{s_j^2} > 0$
      for non-degenerate $\hyperparam_j$.

      To prove \eqref{eqn:percomponent_dominates_scalar}, note that for any
      scalar $\baseline \in \R$ and each $j$,
      the per-component baseline $\baseline\opt_j$ minimises
      $V_j(\cdot)$ by construction, so
      $V_j(\baseline\opt_j) \leq V_j(\baseline)$ for every $\baseline$.
      Summing over $j$ gives the inequality in
      \eqref{eqn:percomponent_dominates_scalar}.
      Equality holds if and only if $\baseline\opt_j$ is the same for all $j$,
      which is when the scalar optimum $\baseline\opt$ already achieves the
      coordinate-wise minimum.
      Note also that $\baseline\opt$ is the Fisher-diagonal-weighted mean of
      the per-component baselines:
      $\baseline\opt = \sum_j \FisherI_{jj}\,\baseline\opt_j \,/\, \sum_j \FisherI_{jj}$,
      which follows directly from substituting \eqref{eqn:optimal_percomponent_baseline_theory}
      into \eqref{eqn:optimal_scalar_baseline_theory}.
      The sample estimator \eqref{eqn:optimal_percomponent_baseline_sample} is
      obtained by replacing expectations with sample averages.
    \end{proof}

    \begin{remark}[]\label{rmrk:baseline} 
      Both \eqref{eqn:optimal_scalar_baseline_sample}
      and \eqref{eqn:optimal_percomponent_baseline_sample} are computed from
      the same sample already drawn to estimate the gradient, incurring no
      additional evaluations of $\obj$.
      Different coordinates $j$ may have very different signal-to-noise ratios.
      This is especially the case near convergence, when some
      $\hyperparam_j$ are close to $0$ or $1$.
      The practical variance-reducing effect of both estimators is verified
      empirically in \Cref{sec:numerical_baseline_variance}, which shows
      that the per-component baseline yields lower variance than the
      scalar baseline in the CB experiments across two representative
      initialization regimes.
    \end{remark}

    \paragraph{Choice of baseline.}
      Both Propositions are presented because they apply in different
      settings.
      The scalar result \Cref{prop:optimal_scalar_baseline} is the
      standard REINFORCE construction and is well defined for any
      probabilistic model whose score function is square integrable.
      The per-component refinement
      \Cref{prop:optimal_percomponent_baseline} applies when the
      probability model is built from a vector of independent
      per-coordinate parameters.
      This is the case for the CB and GCB models considered here, where
      each candidate location has its own Bernoulli parameter
      $\hyperparam_i$.
      For probabilistic models whose design does not have this
      one-parameter-per-coordinate structure, the per-component baseline
      is not well defined, and the scalar baseline is the appropriate
      construction.
      One example is the distribution over index sequences in
      trajectory-based experimental design
      \cite{attia2026probabilistic_path_oed}.
      In the experiments reported here the per-component baseline is the
      default, and cross-validation against the scalar baseline is
      reported in \Cref{sec:numerical_baseline_variance}.

    \subsection{Algorithmic statement}
    \label{subsec:optimization_algorithms}
      \Cref{alg:probabilistic_binary_optimization}
      details the steps
      of the proposed probabilistic approach 
      for solving
      \eqref{eqn:binary_optimization_budget_equality_constraint}
      and 
      \eqref{eqn:binary_optimization_budget_inclusion_constraint}.
      \Cref{alg:probabilistic_binary_optimization} inherits the advantages 
      of the original stochastic optimization 
      algorithm \cite[Algorithm 3.2]{attia2022stochastic}. 
      Specifically,
      the value of $\obj$ is evaluated repeatedly at instances of the binary variable 
      $\design$, which are more likely/frequently revisited as the algorithm proceeds.
      Thus, redundancy in computation is prevented by keeping track of the sampled binary 
      variables and the corresponding value of $\obj$.
      \Cref{alg:probabilistic_binary_optimization} 
      is conceptual, because the loop (Step 2) is not necessarily a finite processes. 
      In practice, we can ensure that the loop is finite by setting a tolerance of the projected gradient and/or a maximum number of iterations.

      \begin{algorithm}[htbp!]
        \caption{
          Probabilistic black-box binary optimization for solving
          \eqref{eqn:binary_optimization_budget_equality_constraint}
          or
          \eqref{eqn:binary_optimization_budget_inclusion_constraint}.
        }
        \label{alg:probabilistic_binary_optimization}
        \begin{algorithmic}[1] 
        
          \Require{Initial distribution parameter $\hyperparam^{(0)}$,
                    stepsize schedule $\eta^{(n)}$, 
                    and sample sizes $\Nens,\, N_{\rm opt}$
                    }
          \Ensure{$\design\opt$}

          \State{initialize $n = 0$}

          \While{Not Converged}
            \State{ \label{algstep:random_sample}
              Sample $\{\design[k]; k=1,\ldots,\Nens \} $ 
              \Comment{
                Use \Cref{alg:CB_Model_sampling} 
                for \eqref{eqn:binary_optimization_budget_equality_constraint}
                or
                \Cref{alg:GCB_Model_sampling} 
                for \eqref{eqn:binary_optimization_budget_inclusion_constraint}
              }
            }
          
            \State{
              Calculate  optimal baseline estimate $\baseline\opt$
              \Comment{
                Use \eqref{eqn:optimal_scalar_baseline_sample}; see \Cref{subsec:variance_reduction}
              }
            }

            \State{
              Calculate $ \widehat{\vec{g}}^{{\rm b},(n)}$
              \Comment{
                Use \eqref{eqn:budget_equality_constraint_stochastic_gradient_with_basesline}
                for \eqref{eqn:binary_optimization_budget_equality_constraint}
                or
                \eqref{eqn:budget_inclusion_constraint_stochastic_gradient_with_baseline}
                for \eqref{eqn:binary_optimization_budget_inclusion_constraint}
              }
            }

            \State{
              Update
              $ \hyperparam^{(n+1)}
                  = \hyperparam^{(n)} + \eta^{(n)} \Proj{\hyperparam, \eta}{\widehat{\vec{g}}^{{\rm b},(n)}}
              $
              \Comment{Use $\proj_{\hyperparam, \eta}$ given by \eqref{eqn:scaling_projector}}
            }
            
            \State {
              Update $n\leftarrow n+1$ 
            }

          \EndWhile
          
          \State{
            Set $\hyperparam\opt = \hyperparam^{(n)}$
          }

          \State{
              Sample $\{\design[k]; k=1,\ldots,N_{\rm opt} \} $ 
              with the optimal parameter $\hyperparam\opt$.
          }
     
          \State \Return{
              $\design\opt$: the design $\design$ with the largest
                value of $\obj$ in the sample.
          }
          
        \end{algorithmic}
      \end{algorithm}

      The parameter update in \Cref{alg:probabilistic_binary_optimization} uses the
      scaling projection operator $\Proj{\hyperparam, \eta}{\cdot}$ defined by
      \begin{equation}\label{eqn:scaling_projector}
        \Proj{\hyperparam, \eta}{g} := \rho \, g\,; \quad
        \rho := \min\bigl\{1,\, \min_{i=1,\ldots,\Nsens} \rho_i\bigr\}\,; \quad
        \rho_i =
          \begin{cases}
            \dfrac{1-\hyperparam_i}{|g_i|} & \text{if } \hyperparam_i \pm \eta g_i > 1 \\[6pt]
            \dfrac{\hyperparam_i}{|g_i|}   & \text{if } \hyperparam_i \pm \eta g_i < 0 \\[6pt]
            1                              & \text{otherwise,}
          \end{cases}
      \end{equation}
      which scales the entire gradient vector to keep
      $\hyperparam^{(n)} + \eta^{(n)}\Proj{\hyperparam, \eta}{\widehat{\vec{g}}^{(n)}}$
      inside $[0,1]^{\Nsens}$.
      The parameter update step is therefore
      \begin{equation}\label{eqn:stoch_steep_step}
        \hyperparam^{(n+1)}
          = \hyperparam^{(n)} + \eta^{(n)} \Proj{\hyperparam, \eta}{\widehat{\vec{g}}^{{\rm b},(n)}} \,,
      \end{equation}
      where $0 < \eta^{(n)} \leq 1$ is the step size (learning rate) at the $n$th
      iteration.
      The $\pm$ in \eqref{eqn:scaling_projector} covers both maximization ($+$)
      and minimization ($-$), and the plus sign in \eqref{eqn:stoch_steep_step}
      is replaced with a minus sign for minimization.

      \begin{remark}[Scaling projection operator]\label{remark:scaling_projector}
        The projector \eqref{eqn:scaling_projector} scales the \emph{entire}
        gradient vector by a single scalar $\rho \leq 1$, chosen as the largest value
        such that the updated parameter remains in $[0,1]^{\Nsens}$ for any
        $\eta^{(n)} \leq 1$.
        This has two practical consequences.
        First, it \emph{absorbs the effect of the step size}: since feasibility is
        guaranteed for any $\eta \in (0,1]$, one can use a fixed learning rate
        across all iterations without per-iteration checks or adaptive schedules
        (as done in all numerical experiments of \Cref{sec:numerical_experiments}).
        Second, it preserves the \emph{direction} of the gradient update — only
        the magnitude is reduced — so the ascent property of the step is not
        compromised.
        Strictly speaking, $\Proj{\hyperparam,\eta}{\cdot}$ is not a
        metric projection onto a convex set: it is a scalar rescaling
        of the search direction $g$ chosen so that the update
        $\hyperparam + \eta\,\Proj{\hyperparam,\eta}{g}$ remains in
        $[0,1]^{\Nsens}$ for every $\eta \in (0,1]$.
        We retain the operator notation
        $\Proj{\hyperparam,\eta}{\cdot}$ only for compactness within
        \Cref{alg:probabilistic_binary_optimization}.
        The standard alternative within the projected-gradient family
        is the Euclidean projection onto $[0,1]^{\Nsens}$, i.e.,
        component-wise truncation, which underlies projected gradient
        descent \cite[Ch.~16]{nocedal2006numerical} and is used in
        the unconstrained variant of the present approach
        \cite{attia2022stochastic}.
        Truncation distorts the direction of the gradient
        component-by-component whenever any coordinate is clipped.
        The construction adopted here preserves the direction of $g$,
        allows a fixed learning rate $\eta \in (0,1]$ without
        per-iteration tuning, and keeps iterates inside
        $[0,1]^{\Nsens}$, which is consistent with the non-degenerate
        intermediate and optimal policies observed in
        \Cref{sec:numerical_experiments}.
      \end{remark}

    \subsection{Computational considerations}
    \label{subsec:computational_considerations}
      The computational cost per iteration of
      \Cref{alg:probabilistic_binary_optimization} is dominated by two
      axes.
      The first is the evaluation of the objective $\obj(\design[k])$
      for each of the $\Nens$ sampled designs.
      The second is the sampling procedure together with the
      computation of the score
      $\nabla_{\hyperparam}\log\CondProb{\design[k]}{\designsum}$ for
      each sample.

      For the CB model under the budget-equality constraint
      \eqref{eqn:binary_optimization_budget_equality_constraint},
      the per-iteration asymptotic cost of the sampling-and-score axis
      admits an explicit breakdown.
      Building the auxiliary $q$-matrix
      \eqref{eqn:q_matrix} once per iteration costs
      $\mathcal{O}(\Nsens\,\designsumval)$.
      Drawing the $\Nens$ designs via
      \Cref{alg:CB_Model_sampling} then costs
      $\mathcal{O}(\Nens\,\Nsens)$, since each design is drawn in
      $\mathcal{O}(\Nsens)$ time using the precomputed $q$-matrix.
      Evaluating the score
      $\nabla_{\hyperparam}\log\CondProb{\design[k]}{\designsum=\designsumval}$
      via the recurrence
      \eqref{eqn:R_function_and_Bernoulli_weights_recurrence_relations}
      requires, for each sampled design, the leave-one-out values
      $R(\designsumval-1, V\setminus\{i\})$ and
      $R(\designsumval, V\setminus\{i\})$ for $i = 1, \ldots, \Nsens$.
      The full table $R(k, V)$ for $k = 0, \ldots, \designsumval$ is
      built once per design at cost $\mathcal{O}(\Nsens\,\designsumval)$,
      from which all $\Nsens$ leave-one-out values are recovered by
      inverting the recurrence.
      The per-component assembly of the score then proceeds in
      $\mathcal{O}(\Nsens)$ work per index in the current implementation,
      giving $\mathcal{O}(\Nsens^2)$ per design and
      $\mathcal{O}(\Nens\,\Nsens^2 + \Nens\,\Nsens\,\designsumval)$
      over the full sample.
      Forming the per-component baseline
      \eqref{eqn:optimal_percomponent_baseline_sample} and the
      baselined gradient is two passes over the score matrix at
      $\mathcal{O}(\Nens\,\Nsens)$.
      The scaling projector
      \eqref{eqn:scaling_projector} and the parameter update
      contribute $\mathcal{O}(\Nsens)$.
      The total algorithmic cost per iteration is therefore
      $\mathcal{O}\!\left(\Nens\,\Nsens\,(\Nsens + \designsumval)\right)$,
      which simplifies to $\mathcal{O}(\Nens\,\Nsens^2)$ in the typical
      regime $\designsumval \le \Nsens$.
      The GCB model under the inclusion constraint
      \eqref{eqn:binary_optimization_budget_inclusion_constraint}
      adds a multiplicative factor of $|\designsumset|$ in the
      score/log-PMF terms because the relevant $R$-tables must be
      assembled for each budget
      $\designsumval \in \designsumset$.
      The objective-evaluation axis contributes
      $\mathcal{O}(\Nens \cdot \Cobj)$ where $\Cobj$ is the cost of a
      single evaluation of $\obj(\design)$.
      This contribution is application-specific and data-parallel across
      the $\Nens$ samples.
      For the OED problems considered in this work the per-evaluation
      forward solve is sufficiently expensive that
      $\Nens\,\Cobj$ dominates the algorithmic cost
      $\mathcal{O}(\Nens\,\Nsens^2)$ in the regimes studied in
      \Cref{sec:numerical_experiments}; this regime is typical for
      simulation-based OED applications.
      An empirical confirmation of this breakdown is given in the context of 
      the
      advection--diffusion OED experiments given in 
      \Cref{sec:numerical_experiments}.

    \subsection{Relation between binary and probabilistic formulations}
    \label{subsubsec:binary_vs_probabilistic_formulation}
      The relation between 
      the original binary optimization problems 
      and the proposed probabilistic formulations
      is stated by \Cref{theorem:Binary_to_Probabilistic}.
      
      \begin{theorem}\label{theorem:Binary_to_Probabilistic}
        The optimal solutions of 
        \eqref{eqn:binary_optimization_budget_equality_constraint} and 
        \eqref{eqn:probabilistic_optimization_budget_equality_constraint}
        are such that 
        \begin{equation}
          \argmax\limits_{\substack{
              \design\in\{0, 1\}^{\Nsens} ,\,  
              \wnorm{\design}{0} = \designsumval
            } 
          } \obj(\design)
          \subseteq
          \argmax_{\hyperparam \in [0, 1]^{\Nsens} }
        \Expect{
             \design\sim\CondProb{\design}{\designsum=\designsumval}
           }{\obj(\design)}  \,;
        \end{equation}
        and if the solution $\design\opt$ of 
        \eqref{eqn:binary_optimization_budget_equality_constraint}
        is unique, then 
        $\design\opt=\hyperparam\opt$, where $\hyperparam\opt$ is the 
        unique optimal solution of \eqref{eqn:probabilistic_optimization_budget_equality_constraint}.
        Similarly, 
        the optimal solutions of the two problems 
        \eqref{eqn:binary_optimization_budget_inclusion_constraint} and 
        \eqref{eqn:probabilistic_optimization_budget_inclusion_constraint}
        are such that 
        \begin{equation}
          \argmax\limits_{\substack{
              \design\in\{0, 1\}^{\Nsens} \\
              \wnorm{\design}{0} \in \designsumset \subseteq\{0, 1\ldots, \Nsens\}
            } 
          } \obj(\design)
          \subseteq
          \argmax_{\hyperparam \in [0, 1]^{\Nsens} }
        \Expect{
             \design\sim\CondProb{\design}{
               \designsum \in \designsumset \subseteq\{0, 1\ldots, \Nsens\}
             }
           }{\obj(\design)}  \,;
        \end{equation}
        and if $\design\opt$ uniquely solves 
        \eqref{eqn:binary_optimization_budget_inclusion_constraint}
        then 
        $\design\opt=\hyperparam\opt$ 
        the unique 
        solution of \eqref{eqn:probabilistic_optimization_budget_inclusion_constraint}.
      \end{theorem}
      \begin{proof}
        The proof follows directly from \cite[Lemma~3.2 and Proposition~3.1]{attia2022stochastic}.
        The argument carries over to the constrained setting by noting that the CB and GCB models
        restrict the support of $\CondProb{\design}{\hyperparam}$ to the feasible region
        $\{\design : \wnorm{\design}{0} = \designsumval\}$
        (respectively $\{\design : \wnorm{\design}{0} \in \designsumset\}$).
        With this substitution, the inclusion of the binary optimal set in the probabilistic
        optimal set, and the uniqueness implication, follow identically to
        \cite[Proposition~3.1]{attia2022stochastic}.
      \end{proof}

      \Cref{theorem:Binary_to_Probabilistic} shows that the
      optimal set of the probabilistic optimization problem
      includes the optimal set of the original binary optimization problem.
      By sampling the underlying probability distribution (e.g., CB or GCB) with parameter
      set to the optimal solution of the probabilistic optimization problem, one obtains a set of
      binary realizations of $\design$ with at least a near-optimal objective value as suggested in
      \cite{attia2022stochastic} and as shown in \Cref{sec:numerical_experiments}.

    \subsection{Convergence analysis}
      \label{subsec:convergence_analysis}
      The convergence analysis requires the following assumption on the utility function.

      \begin{assumption}\label{assumption:utility_boundedness}
        The utility function $\obj: \{0,1\}^{\Nsens} \to \Rnum$ is real-valued and bounded,
        i.e., there exists $M < \infty$ such that $|\obj(\design)| \leq M$
        for all $\design \in \{0,1\}^{\Nsens}$.
      \end{assumption}

      \noindent
      This assumption is mild and holds in all standard OED settings:
      the utility is evaluated at a finite number of binary designs
      ($2^{\Nsens}$ in the unconstrained case, $\binom{\Nsens}{\designsumval}$ in
      the constrained case), so boundedness is automatic.
      Boundedness of $\obj$ ensures that the stochastic gradient estimator
      \eqref{eqn:kernel_stochastic_gradient_Bernoulli} has finite variance,
      which is required for convergence of the stochastic gradient ascent algorithm.

      Proving convergence of \Cref{alg:probabilistic_binary_optimization} in expectation to
      the solution of the corresponding binary optimization problem requires developing
      bounds on the first- and second-order derivatives of the exact gradient and ensuring
      that the corresponding stochastic gradient approximation is an unbiased estimator
      and has bounded variance \cite{attia2022stochastic}.

      Evaluations of the PMF and the gradients of the CB model \eqref{eqn:CB_Model_Full_PMF}, 
      and the GCB model \eqref{eqn:GCB_Model_PMF}, rely primarily on 
      the weights \eqref{eqn:R_function_and_Bernoulli_weights} of the non-degenerate entries.
      For non-degenerate 
      probabilities 
      the weights 
      satisfy the following 
      upper bounds, 
      \begin{equation}\label{eqn:R_function_and_Bernoulli_weights_bounds}
        \max\limits_{i=1,\ldots,\Nsens} \!w_i 
          = \frac{
              \max\limits_{i=1,\ldots,\Nsens} \hyperparam_i 
            }{
              1- \max\limits_{i=1,\ldots,\Nsens} \hyperparam_i 
            }; \,\, 
        \max\limits_{i=1,\ldots,\Nsens} \frac{(1+w_i)^2}{w_i} 
          =  \frac{1}{\widehat{\hyperparam}(1-\widehat{\hyperparam})} 
          , \,\, 
            \widehat{\hyperparam} := 
            \max\limits_{i=1,\ldots,\Nsens} \{\hyperparam_i, 1\!-\!\hyperparam_i\}
          \,,
      \end{equation}
      which enables
      bounding 
      the derivatives of the CB
      model described by \Cref{lemma:CB_Model_derivatives_bounds}.

      \begin{lemma}\label{lemma:CB_Model_derivatives_bounds}
      \begin{subequations}\label{eqn:CB_Model_derivatives_bounds}
        The 
        non-degenerate CB 
        model \eqref{eqn:CB_Model_PMF} satisfy the following 
        bounds:
        \begin{align}
          \label{eqn:CB_Model_derivatives_bounds_1}
          \sqnorm{\nabla_{\hyperparam} \CondProb{\design}{ \designsum=\designsumval}}
            &\leq 
            \sqnorm{\nabla_{\hyperparam} \log \CondProb{\design}{ \designsum=\designsumval}}
            \leq 
            \Nsens\, C^2
            \,,\\
          \label{eqn:CB_Model_derivatives_bounds_2}
          \Expect{}{
            \sqnorm{
              \nabla_{\hyperparam} \log \CondProb{\design}{ \designsum=\designsumval}  
            }
          }
          &= 
            \brVar{ 
              \nabla_{\hyperparam} \log \CondProb{\design}{ \designsum=\designsumval} 
            }
              \leq \frac{\Nsens}{4}\, C^2
              \,,  \\
          \label{eqn:CB_Model_derivatives_bounds_3}
            \abs{
              \delll{
                \CondProb{\design}{\designsum=\designsumval} 
              }{\hyperparam_i}{\hyperparam_j}
            }
            &\leq 2 C^2 
            \,,
        \end{align}
        \end{subequations}
        where 
        $C= \max\limits_{i=1,\ldots,\Nsens} \frac{(1+w_i)^2}{w_i}$.
        Moreover, there is a finite constant $\widehat{C}$ such that 
        the derivatives \eqref{eqn:CB_Model_Full_Gradient} of the degenerate 
        CB 
        model \eqref{eqn:CB_Model_Full_PMF} satisfy 
        \begin{subequations}\label{eqn:CB_Model_Full_derivatives_bounds}
          \begin{align}
          \sqnorm{\nabla_{\hyperparam} \CondProb{\design}{\designsum=\designsumval}}
            &\leq \widehat{C}
            \,, \\
          \Expect{}{
            \sqnorm{
              \nabla_{\hyperparam} \log \CondProb{\design}{\designsum=\designsumval}  
            }
          }
            &\leq \widehat{C}
            \,, \\
            \abs{
              \delll{
                \CondProb{\design}{\designsum=\designsumval} 
              }{\hyperparam_i}{\hyperparam_j}
            }
            &\leq \widehat{C}
            \,.           
          \end{align}
        
        \end{subequations}
      \end{lemma}
      \begin{proof}
        See \Cref{app:probabilistic_optimization_proofs}.
      \end{proof}

      \Cref{theorem:probabilistic_optimization_derivatives_bounds_exact} 
      employs \Cref{lemma:CB_Model_derivatives_bounds} to
      show that 
      the gradient of $\stochobj$ is bounded in norm (gradient bound)
      and the gradient difference at any two points is bounded by a finite constant
      (Hessian entries are bounded), which together imply that $\stochobj$ is
      Lipschitz smooth and guarantee convergence
      of a steepest-descent approach
      for solving \eqref{eqn:probabilistic_optimization_budget_equality_constraint}
      to a local optimum.
      \begin{theorem}\label{theorem:probabilistic_optimization_derivatives_bounds_exact}
        Let $\design \in \designdomain \subset \{0, 1\}^{\Nsens} $ be modeled by 
        the 
        CB model \eqref{eqn:CB_Model_PMF}, 
        then the 
        derivatives of the 
        objective $\stochobj$ in
        \eqref{eqn:probabilistic_optimization_budget_equality_constraint} 
        satisfy the  
        bounds:
        \begin{subequations}\label{eqn:exact_gradient_bounds}
        \begin{eqnarray} 
            \norm{ \nabla_{\hyperparam}\,\stochobj(\hyperparam) } 
            &\leq &
            M\, \binom{\Nsens}{\designsumval} C \sqrt{\Nsens}
            \,, \quad \hyperparam \in [0,1]^{\Nsens} \,,
              \label{eqn:exact_gradient_norm_bound}  \\
            \norm{ 
              \nabla_{\hyperparam}\,\stochobj(\hyperparam[1]) - 
              \nabla_{\hyperparam}\,\stochobj(\hyperparam[2]) 
            } &\leq& 2 \, M\, \sqrt{
            \binom{\Nsens}{\designsumval} \Nsens C
          }  
            \,, \quad \hyperparam[1], \hyperparam[2] \in [0,1]^{\Nsens} \,,
              \label{eqn:exact_gradient_Lipschitz} \\
            \abs{ \delll{\stochobj}{\hyperparam_i}{\hyperparam_j} } 
              &\leq& \frac{3\, M\, C}{2}  \sqrt{\binom{\Nsens}{\designsumval}} \,,
            \label{eqn:Hessian_entries_bound}
        \end{eqnarray}
        \end{subequations}
        where $C= \max\limits_{i=1,\ldots,\Nsens} \frac{(1+w_i)^2}{w_i}$,
        and $\binom{\Nsens}{\designsumval}$ is the cardinality
        of the feasible domain,
        and 
        $M = \max\limits_{\design\in \designdomain}{\{ \left|\obj(\design)\right| ,\, \wnorm{\design}{0}=\designsumval\}}$.
        Moreover, if $\design$ is in the degenerate case, that is, when $\design$ follows
        \eqref{eqn:CB_Model_Full_PMF}, then there is a finite constant 
        $\widetilde{C}<\infty$ such that 
        \begin{equation}\label{eqn:exact_gradient_bounds_degenerate}
          \norm{ \nabla_{\hyperparam}\,\stochobj(\hyperparam) } \leq \widetilde{C} \,;
          \qquad 
          \norm{ 
              \nabla_{\hyperparam}\,\stochobj(\hyperparam[1]) - 
              \nabla_{\hyperparam}\,\stochobj(\hyperparam[2]) 
          } \leq \widetilde{C} \,;
          \qquad 
          \abs{ \delll{\stochobj}{\hyperparam_i}{\hyperparam_j} } \leq\widetilde{C} \,.
        \end{equation}
      \end{theorem}
      \begin{proof}
        See \Cref{app:probabilistic_optimization_proofs}.
      \end{proof}

      From \Cref{theorem:probabilistic_optimization_derivatives_bounds_exact} it follows 
      that similar bounds can be developed for the stochastic objective 
      \eqref{eqn:probabilistic_optimization_budget_inclusion_constraint}.
      This follows 
      from 
      \Cref{theorem:CB-Multiple-Sums} by noting that for the GCB model \eqref{eqn:GCB_Model_PMF}, 
      the probabilities, derivatives, and first- and second-order moments are weighted 
      linear combinations of the corresponding values of the CB 
      model \eqref{eqn:CB_Model_Full_PMF}. 
      Note that the value of the upper bound $\widetilde{C}$ in 
      \eqref{eqn:exact_gradient_bounds_degenerate} itself is irrelevant here 
      because only the 
      existence of this finite constant guarantees convergence of a gradient-based optimization
      algorithm involving the exact gradient of the stochastic objective $\stochobj$.
      Convergence of the proposed \Cref{alg:probabilistic_binary_optimization}
      requires the stochastic estimate of the gradient to be 
      unbiased with bounded variance. 
      This is shown by \Cref{theorem:probabilistic_optimization_derivatives_bounds_stochastic}.

      \begin{theorem}\label{theorem:probabilistic_optimization_derivatives_bounds_stochastic}
        \begin{subequations}\label{eqn:probabilistic_optimization_derivatives}
        Let $\design\sim\CondProb{\design}{\designsum}$, and let 
          \begin{align}
          \vec{g}
            &= \nabla_{\hyperparam} 
              \Expect{\design \sim\CondProb{\design}{\designsum}}{\obj(\design)}
          \,, \\
          \widehat{\vec{g}}
            &:= \frac{1}{\Nens} \sum_{k=1}^{\Nens}
              \obj(\design[k]) \nabla \log \CondProb{\design[k]}{\designsum} \,,
          \end{align}
        \end{subequations}
        where $\CondProb{\design}{\designsum}$ refers to any of the 
        models \eqref{eqn:CB_Model_PMF}, \eqref{eqn:CB_Model_Full_PMF}, 
        or \eqref{eqn:GCB_Model_PMF}, respectively.
        Then 
        \begin{subequations} \label{eqn:probabilistic_optimization_derivatives_bounds_stochastic}
        \begin{align}
           \label{eqn:probabilistic_optimization_derivatives_bounds_stochastic_1}
          \Expect{}{\widehat{\vec{g}}} 
          &= \nabla_{\hyperparam} 
            \Expect{\design \sim\CondProb{\design}{\designsum}}{\obj(\design)}  \,, \\
           \label{eqn:probabilistic_optimization_derivatives_bounds_stochastic_2}
          \Expect{}{\widehat{\vec{g}}\tran \,\widehat{\vec{g}} }
            &\leq K_1 +  \vec{g}\tran \vec{g} \,, \quad 0 < K_1 < \infty \,.
        \end{align}
        \end{subequations}
      \end{theorem}
      \begin{proof}
        See \Cref{app:probabilistic_optimization_proofs}.
      \end{proof}

      \Cref{theorem:probabilistic_optimization_derivatives_bounds_stochastic}
      shows that the stochastic gradient employed in 
      \Cref{alg:probabilistic_binary_optimization} is an unbiased estimator and 
      satisfies Assumption (d)
      of \cite[Assumptions 4.2]{BertsekasTsitsiklis96}, which guarantees 
      convergence of the stochastic optimization algorithm; see e.g., 
      \cite{attia2022stochastic}.
      Note that the convergence of
      \Cref{alg:probabilistic_binary_optimization}, where
      baseline versions $\widehat{\vec{g}}^{\rm b}$ of the stochastic gradient
      are employed, is guaranteed by
      \Cref{theorem:probabilistic_optimization_derivatives_bounds_stochastic},
      whose variance bound applies to the baselined estimator because
      $\brVar{\widehat{\vec{g}}^{\rm b}} \leq \brVar{\widehat{\vec{g}}}$,
      as shown in \Cref{subsec:probabilistic_optimization_equality}.

      \begin{remark}[Convexity, local optimality, and initialization]\label{remark:local_optimality}
        Theorems~\ref{theorem:probabilistic_optimization_derivatives_bounds_exact}
        and~\ref{theorem:probabilistic_optimization_derivatives_bounds_stochastic}
        guarantee convergence of \Cref{alg:probabilistic_binary_optimization} to a
        \emph{stationary point} (local optimum) of $\stochobj(\hyperparam)$, not
        necessarily the global one.
        This is expected: the original binary problem is NP-hard, and
        \Cref{theorem:Binary_to_Probabilistic} relates the global optima of the two
        formulations, but global optimality of the probabilistic objective does not
        follow from gradient convergence alone.
        Regarding convexity: $\stochobj(\hyperparam)$ is a weighted sum of the objective
        values $\obj(\design[k])$ over all feasible binary designs, and is in general
        nonconvex in $\hyperparam$ unless $\obj$ has special structure (e.g., linear or
        concave in the design).
        Multiple stationary points can therefore arise, and the algorithm may converge
        to different local optima depending on initialization and the random sample sequence.
        We initialize $\hyperparam$ uniformly ($\hyperparam_i = 0.5$ for all $i$),
        which corresponds to a maximally uncertain prior over sensor locations and is a
        natural starting point.
        The numerical comparisons in \Cref{sec:numerical_comparison_relaxation} show that
        the proposed approach consistently produces high-quality designs across problem
        sizes, suggesting practical robustness to initialization even without global
        optimality guarantees.
      \end{remark}

  \section{Numerical Experiments}
    \label{sec:numerical_experiments}
    All experiments in this work use the advection-diffusion optimal sensor placement problem
    described in \Cref{subsec:AD_Setup}.
    \Cref{subsec:AD_Results} uses the classical A-optimal design criterion to
    empirically verify the proposed approach, and \Cref{subsec:scalability} examines
    its scalability with increasing design dimensionality.
    \Cref{sec:numerical_baseline_variance} verifies the need for variance reduction
    and confirms the effectiveness of the proposed optimal baseline estimates.
    \Cref{sec:numerical_A_minimization} demonstrates that the approach applies
    out of the box to a structurally different OED criterion, the Bayesian A-optimal design.
    \Cref{sec:numerical_comparison_relaxation} compares the proposed approach against
    design-space relaxation, showing its advantage in high-dimensional settings.
    All experiments are carried out using \pyoed \cite{attia2024pyoed}. 

    \subsection{Experimental setup}
      \label{subsec:AD_Setup}
      We use an optimal sensor placement problem based on an
      advection-diffusion simulation model widely used for experimental verification
      in the OED literature \cite{PetraStadler11,attia2018goal,attia2022optimal}.
      The advection-diffusion model \eqref{eqn:advection_diffusion} 
      simulates the spatiotemporal evolution of a
      contaminant field $u=\xcont(\vec{x}, t)$ in a closed spatial domain $\domain$
      over time interval $[0, T]$,
      \begin{equation}\label{eqn:advection_diffusion}
        \begin{aligned}
          \xcont_t - \kappa \Delta \xcont + \vec{v} \cdot \nabla \xcont &= 0     
            \quad \text{in } \domain \times [0,T],   \\
          \xcont(x,\,0) &= \theta \quad \text{in } \domain,  \\
          \kappa \nabla \xcont \cdot \vec{n} &= 0  
            \quad \text{on } \partial \domain \times [0,T]\,,
        \end{aligned}
      \end{equation}
      where
      we set the diffusivity to $\kappa=10^{-3}$, and assume a constant velocity field $\vec{v}$ 
      obtained by solving a steady Navier--Stokes equation with the side walls driving 
      the flow \cite{PetraStadler11}. 
      The inference problem seeks to retrieve the true 
      $\iparam^{\rm true}$
      given 
      \eqref{eqn:advection_diffusion} and sparse 
      observations. 
      The domain and the boundary are 
      shown in \Cref{fig:AD_candidate_sensors} along with 
      grid discretization, and candidate sensor locations which are discussed next.

      \begin{figure}[htbp!]
        \centering
        \includegraphics[width=0.24\textwidth]{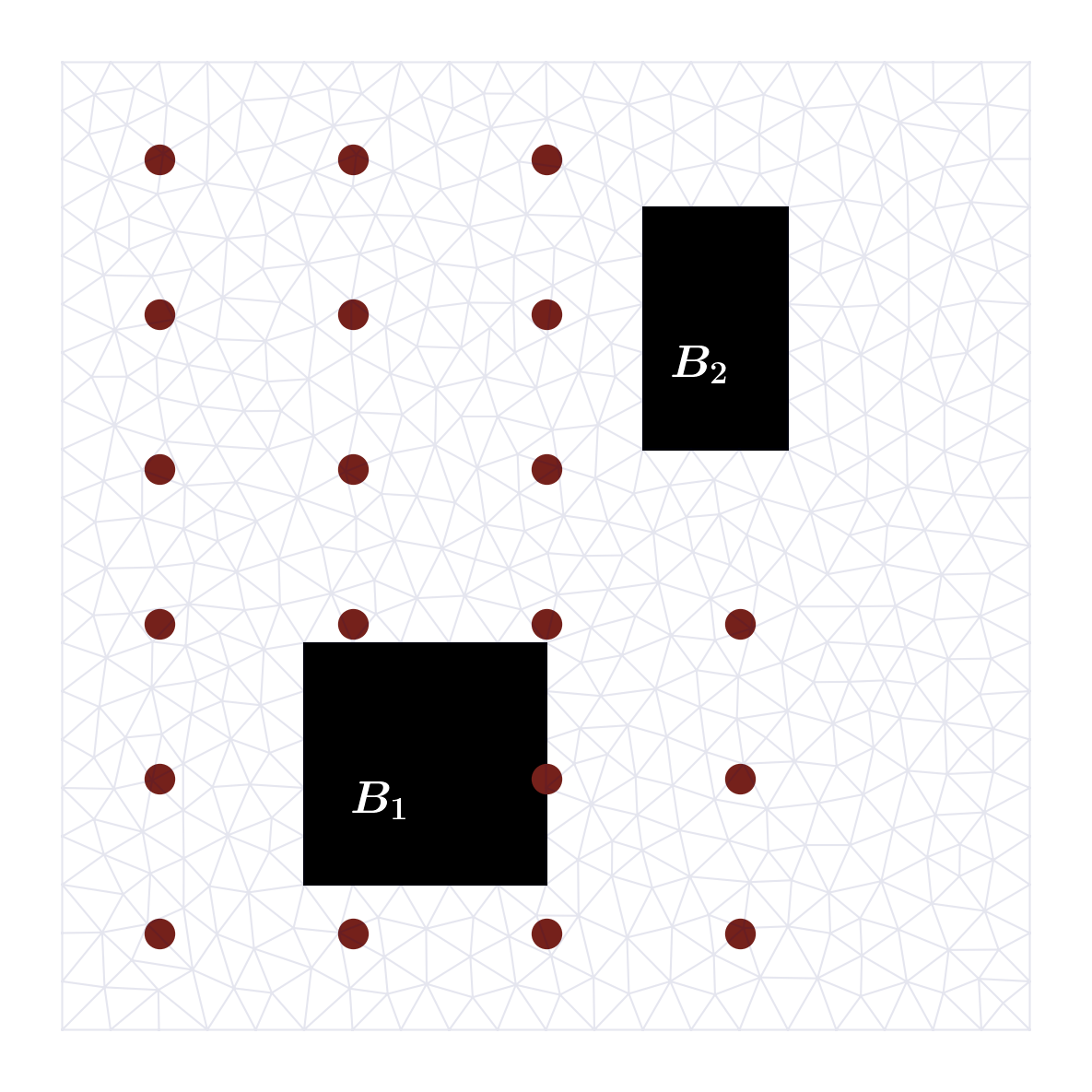}
        \hfill
        \includegraphics[width=0.24\textwidth]{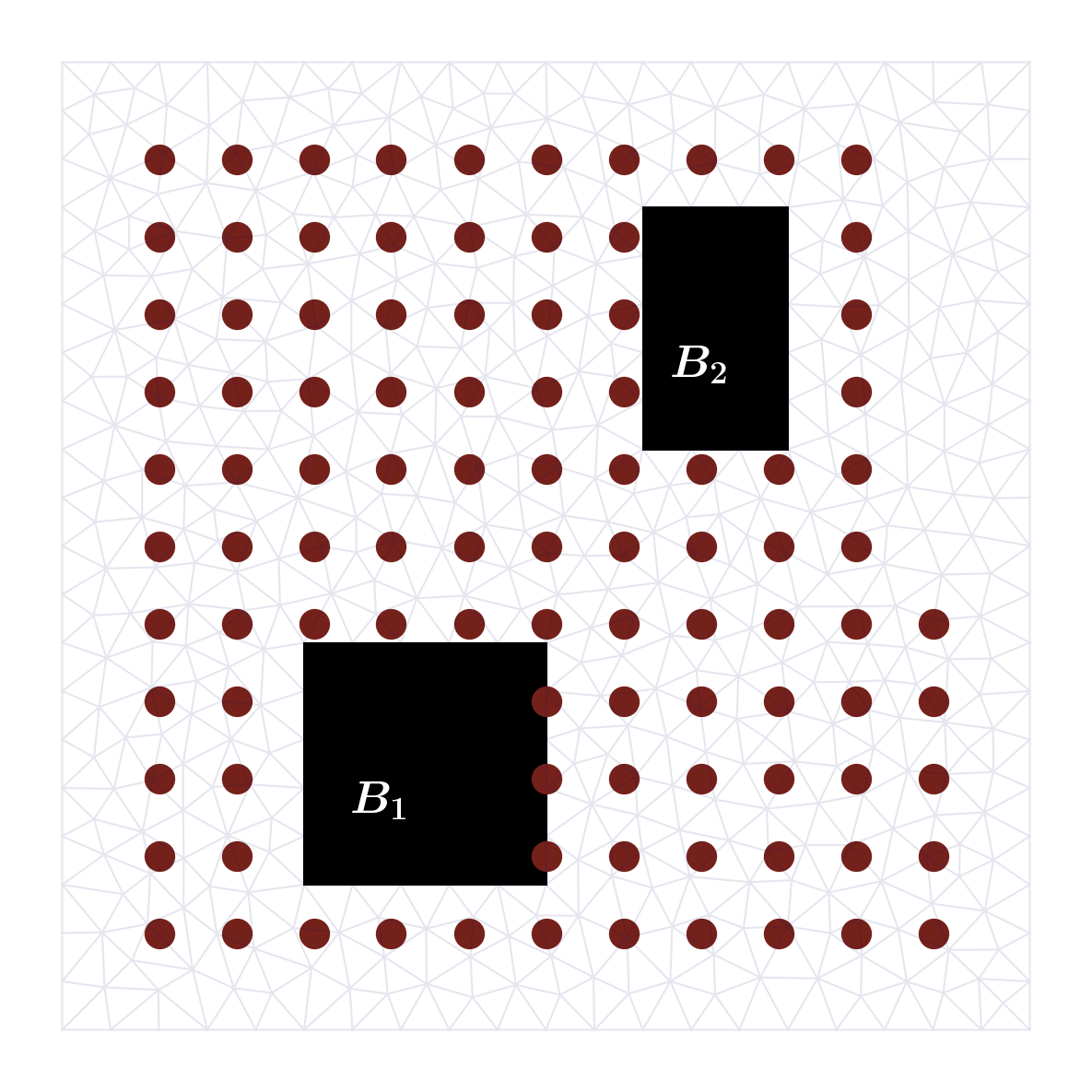}
        \hfill
        \includegraphics[width=0.24\textwidth]{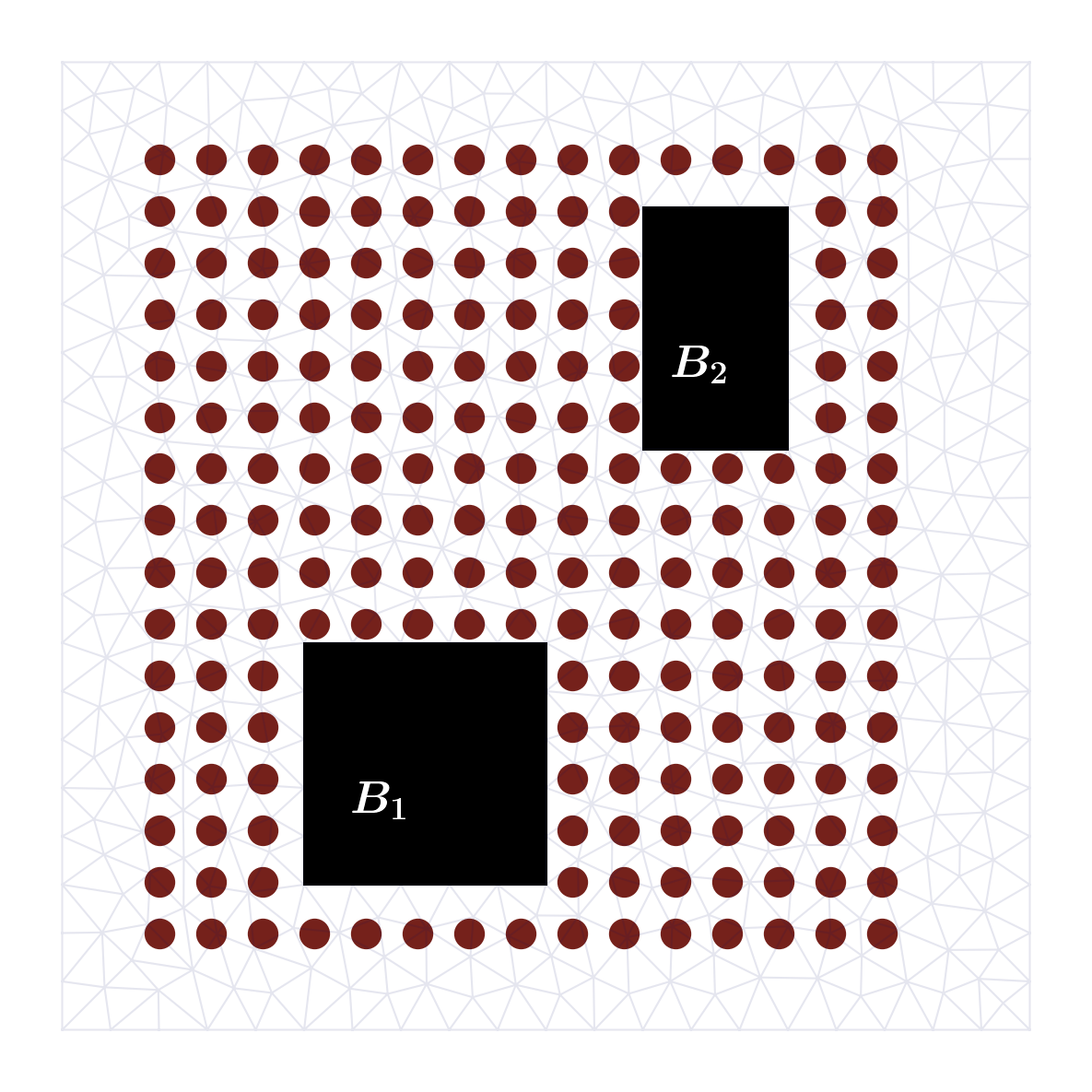}
        \hfill
        \includegraphics[width=0.24\textwidth]{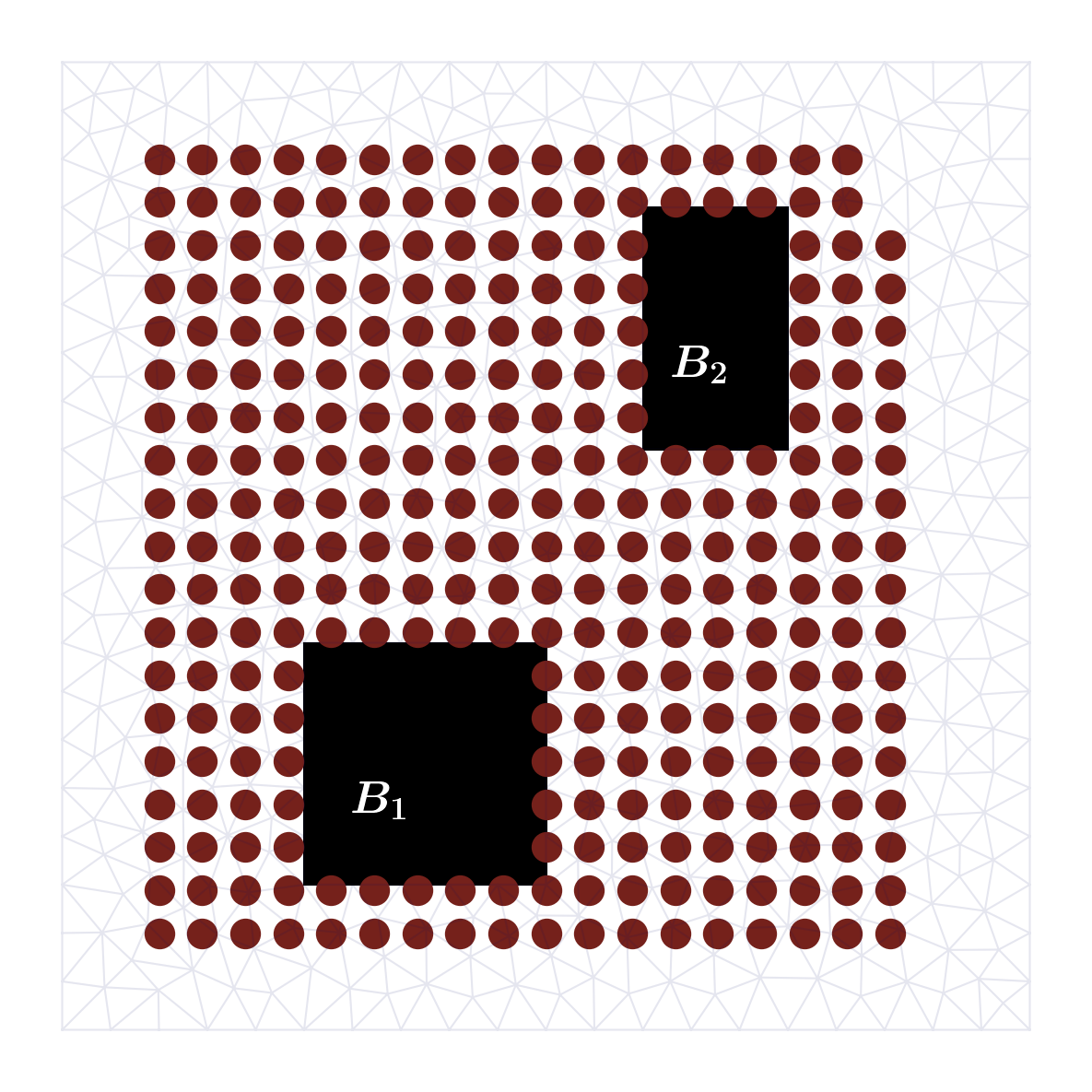}
        \caption{Candidate sensor locations $\Nsens=20, 100, 200, 300$, from left to right, respectively. 
          }
        \label{fig:AD_candidate_sensors} 
      \end{figure}
      %

      \paragraph{The prior.}
        When the Bayesian approach is adopted, a prior must be utilized to
        formulate the posterior covariance matrix. Following
        \cite{bui2013computational,attia2018goal,attia2022optimal}, the prior
        distribution of the inversion parameter is
        $\iparam \sim \GM{\iparb}{\Cparampriormat}$ with prior mean
        $\iparb \equiv 0.5$ (constant on $\domain$) and prior covariance
        $\Cparampriormat$ given by the discretization of the squared inverse
        elliptic operator $\mathcal{A}$: 
        \begin{equation}\label{eqn:prior_operator}
          \Cparampriormat \;=\; \mathcal{A}^{-2} \,, \qquad
          \mathcal{A} \;=\; \delta\,I \;-\; \gamma\,\Delta \,,
        \end{equation}
        with homogeneous Neumann boundary conditions on $\partial\domain$
        and coefficients $\gamma = 1$ and $\delta = 16$ used in all numerical
        experiments reported in this section.
        The covariance \eqref{eqn:prior_operator} is a discrete realization
        of a Mat\'ern--Whittle Gaussian field in the sense of
        \cite{lindgren2011explicit}, with smoothness parameter
        $\nu = 2 - d/2 = 1$ in the present two-dimensional setting.
        The two coefficients $(\gamma, \delta)$ admit the standard
        stochastic partial differential equation (SPDE) interpretation.
        The ratio $\gamma/\delta$ sets the Laplacian length scale
        $\ell := \sqrt{\gamma/\delta} = 0.25$ in the units of $\domain$.
        The corresponding Mat\'ern $0.1$-correlation distance is
        $\sqrt{8\nu}\,\ell \approx 0.71$.
        The product $\gamma\delta$ enters inversely in the pointwise
        prior variance, which evaluates to
        $1/(4\pi\,\gamma\,\delta) \approx 4.97 \times 10^{-3}$ for the
        present values via the Mat\'ern--Whittle marginal-variance
        formula \cite{lindgren2011explicit}.
        The finite-element realization of the prior follows the SPDE
        construction of \cite{lindgren2011explicit}.
        At the continuous level, the prior field $\iparam$ is the
        solution of the SPDE
        $\mathcal{A}\,\iparam = \mathcal{W}$, where $\mathcal{W}$ is
        spatial Gaussian white noise.
        The discretization uses the standard finite-element mass and
        stiffness matrices $\mat{M}$ and $\mat{K}$ of the parameter
        space.
        These are combined into the matrix
        $\mat{B} := \delta\,\mat{M} + \gamma\,\mat{K}$, which assembles
        the bilinear form of the operator $\mathcal{A}$.
        The resulting discrete prior covariance is then
        $\Cparampriormat = \mat{B}^{-1}\,\mat{M}\,\mat{B}^{-1}$.

      \paragraph{Forward and adjoint operators.}
        The forward operator $\F$ (mapping from the inference parameter to the observations)
        represents a simulation of \eqref{eqn:advection_diffusion} 
        followed by a restriction of the state $\xcont$ to the spatiotemporal domain.
        The spatial domain is discretized following a finite-element approach, and the 
        model adjoint is given by 
        $\F\adj \!:= \!\mat{M}\inv\mat{F}\tran$, where
        $\mat{M}$ is the finite-element mass matrix.

      \paragraph{Observational setup.}
        We consider 
        $\Nsens$ uniformly distributed candidate locations
        where the sensors are allowed to be placed.
        Here we consider $\Nsens=20, 100, 200, 300$ candidate 
        locations, respectively, as shown in \Cref{fig:AD_candidate_sensors}. 
        For simplicity, observational data 
        is collected at $\nobstimes = 3$ time instances
        $t_s = (1 + s)\,\Delta t$, $s = 0, 1, 2$,
        with model simulation timestep $\Delta t = 0.5$.
        
        The 
        observations are corrupted with Gaussian noise
        $\GM{\vec{0}}{\Cobsnoise}$, with 
        a covariance matrix 
        $\Cobsnoise:= \sigma^2 \mat{I}$. 
        Here $\mat{I}\in\Rnum^{\Nobs\times\Nobs}$ is the identity matrix with  
        $\Nobs=\Nsens \times \nobstimes$, 
        and we set the standard deviation of the observation error to $0.01$.

        \paragraph{Configurations of the probabilistic optimization algorithm.}
          We carry out numerical experiments to find the optimal placement of
          $10$ sensors (out of $\Nsens$ candidate locations) 
          in the domain $\domain$ such that the optimal design 
          primarily optimizes a scalar summary of the Fisher information matrix $\mathcal{I}(\design)$.
          The specific formulation of the optimization problem is explained in each experiment. 
          
          In all experiments we use the following setup for the probabilistic optimization 
          approach described by \Cref{alg:probabilistic_binary_optimization}.
          The stochastic gradient sample size is set to $\Nens=100$,
          and the learning rate is set to $\eta = 0.5$.
          All experiments use the optimal per-component baseline
          (\Cref{prop:optimal_percomponent_baseline}) for variance reduction of the
          stochastic gradient estimator, unless otherwise is stated explicitly.
          The optimal design is generally obtained by comparing the objective
          value of $100$ points sampled from the final policy.
          The algorithm terminates if the maximum number of iterations (set to $500$)
          is reached or if the magnitude of the projected
          gradient \eqref{eqn:scaling_projector} is below \textsc{pgtol} $=10^{-8}$. 
          %
    
          \Cref{alg:probabilistic_binary_optimization} employs random samples from the 
          probabilistic policy (Step \ref{algstep:random_sample}) at each
          iteration.
          In our experiments 
          we use the same random seed for reproducibility of results.

    \subsection{Classical A-optimal design}
      \label{subsec:AD_Results}
      
      \paragraph{The optimization problem.}
        %
        Here the optimal sensor placement problem is defined as: 
        \begin{equation}\label{eqn:OED_AD_optimization_equality_max}
          \argmax_{\design \in \{0, 1\}^{\Nsens} }{
            \obj(\design) :=  \Trace{\!
             \F\adj \Pseudoinv{\Diag{\design}\Cobsnoise\Diag{\design}} \F
            \!}
          }
          \quad \textrm{s.t.} \quad   
              \wnorm{\design}{0} = 10 
            \,,
        \end{equation}
        where $\Diag{\design}$ is a diagonal matrix with $\design$ on its main diagonal
        and $\dagger$ is the matrix pseudo-inverse.

        For the diagonal observation-error covariance $\Cobsnoise$
        used in the present experiments and for the binary design
        $\design \in \{0,1\}^{\Nsens}$, the pseudoinverse term in
        \eqref{eqn:OED_AD_optimization_equality_max} simplifies to a
        precision-weighted form,
        \begin{equation}\label{eqn:pseudoinv_equivalence}
          \Pseudoinv{\Diag{\design}\,\Cobsnoise\,\Diag{\design}}
            \;=\;
          \Cobsnoise^{-1/2}\,\Diag{\design}\,\Cobsnoise^{-1/2}
          \,,
        \end{equation}
        because $\Diag{\design}^2 = \Diag{\design}$ for binary
        $\design$ and $\Diag{\design}$ commutes with diagonal
        $\Cobsnoise$.
        We nevertheless retain the pseudoinverse formulation in
        \eqref{eqn:OED_AD_optimization_equality_max} and in
        \eqref{eqn:OED_AD_optimization_equality_min} below.
        That formulation remains valid for general $\Cobsnoise$,
        including the correlated-noise setting treated in
        \cite{attia2022optimal}.
        The reader interested in either setting can therefore extract
        the objective directly from
        \eqref{eqn:OED_AD_optimization_equality_max} or
        \eqref{eqn:OED_AD_optimization_equality_min}.
        Specifically, the simpler form
        \eqref{eqn:pseudoinv_equivalence} applies whenever
        $\Cobsnoise$ is diagonal and $\design$ is binary, while the
        pseudoinverse form applies equally to correlated and uncorrelated observation errors. 

        The objective of \eqref{eqn:OED_AD_optimization_equality_max} is to find 
        the optimal placement of $10$ sensors (out of $\Nsens$ candidate locations) 
        in the domain $\domain$ 
        such that the optimal design 
        maximizes the trace of the Fisher information matrix $\mathcal{I}(\design)$.
        Thus, the solution of \eqref{eqn:OED_AD_optimization_equality_max} is an 
        A-optimal design \cite{Pukelsheim93,attia2022optimal}.
        %
 
      \paragraph{Results with $\Nsens=20$ candidate locations.}
        \Cref{fig:OED_AD_SIZE_20_BUDGETS_10_1} (left) shows 
        the objective value at the optimal solution, 
        the best (highest) objective value explored at each iteration,
        and 
        the estimate of the stochastic objective $\Expect{}{\obj}$ 
        at consecutive iterations of the optimization procedure.
        The associated box plot shows the objective values $\obj$
        corresponding to a uniform random sample of $1000$ realizations of $\design$. 
        The random sample is generated by sampling the CB model with $\hyperparam=(0.5, \ldots, 0.5)\tran$.
        This shows that the optimization procedure achieves a superior result (better than
        the best random sample) in a small number of iterations.
        
        The value of the 
        parameter 
        $\hyperparam$ over consecutive iterations is shown in
        \Cref{fig:OED_AD_SIZE_20_BUDGETS_10_1} (right)
        showing that the optimization
        procedure 
        quickly identifies entries of $\design$ that should
        be associated with higher probabilities and those that correspond to lower probabilities.
        The fact that some success probabilities are not converging to a degenerate probability $\{0, 1\}$ 
        indicates that the global optimum is not unique, and the algorithm is converging to an 
        optimal policy that would ideally cover those global optima as part of its support. 

        \begin{figure}[htbp!]
          \centering
          \includegraphics[width=0.50\textwidth]{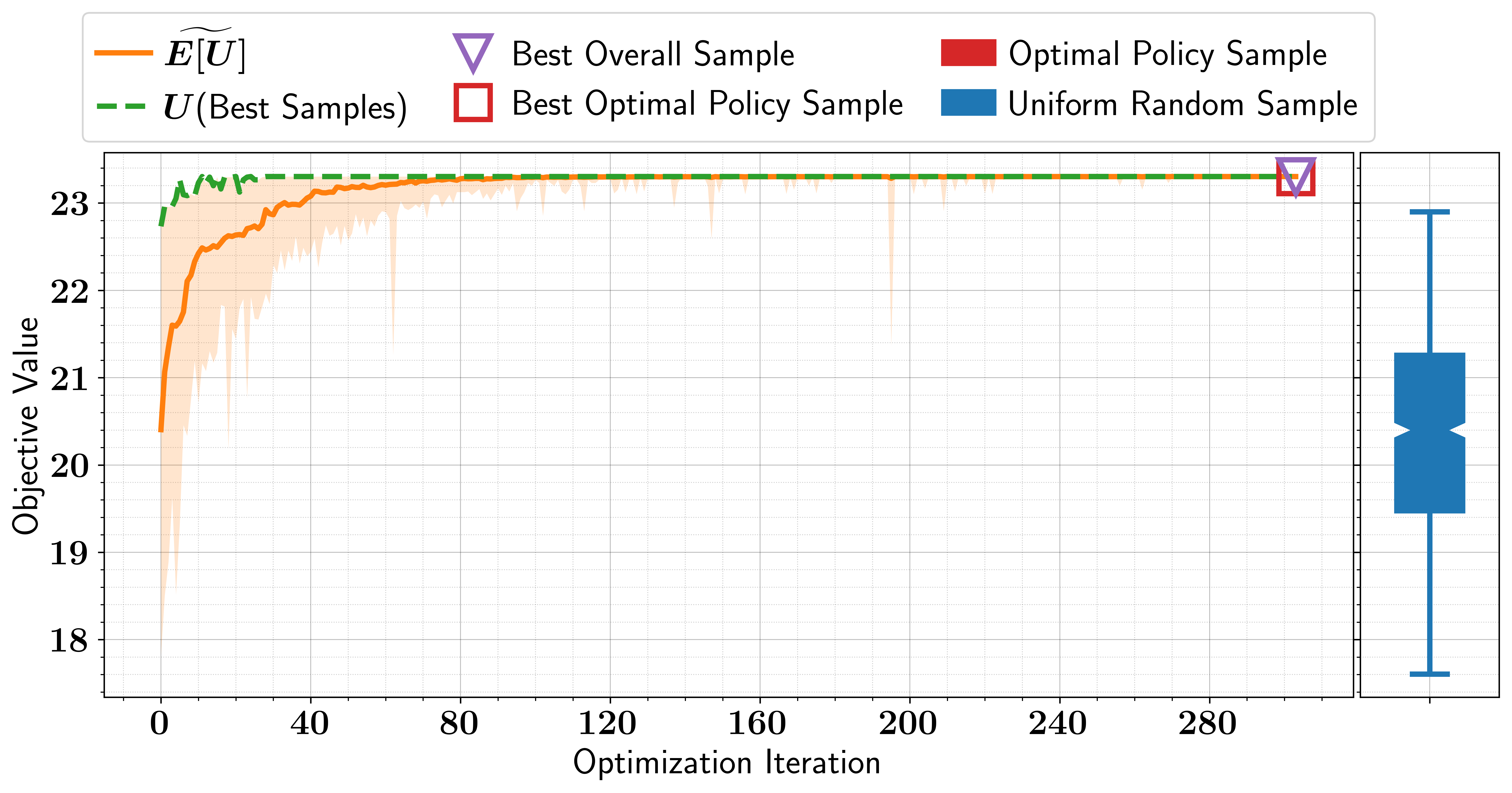}
          \hfill
          \includegraphics[width=0.47\textwidth]{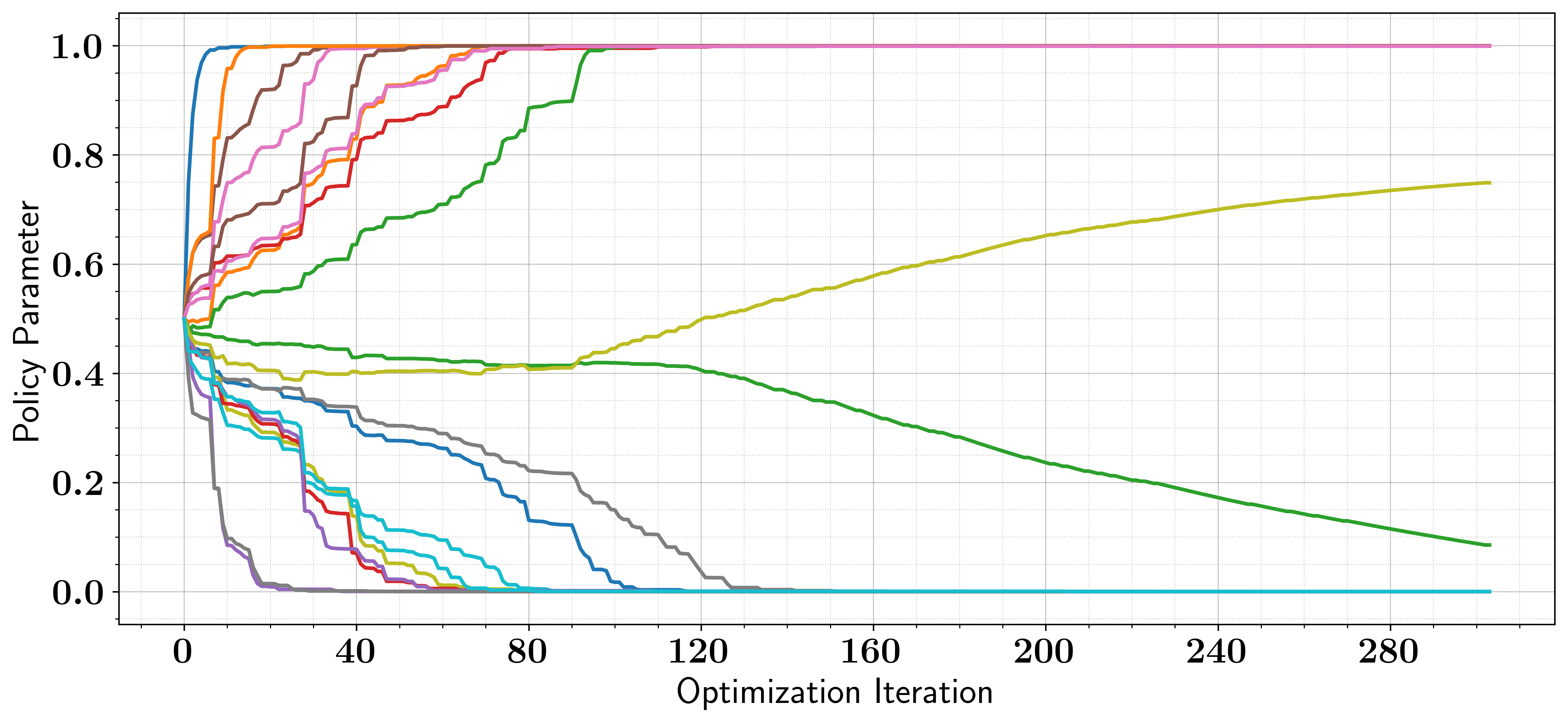}
          \caption{
            Behavior of \Cref{alg:probabilistic_binary_optimization} for
            solving \eqref{eqn:OED_AD_optimization_equality_max} over consecutive iterations,
            with $\Nsens=20$ as in \Cref{fig:AD_candidate_sensors}.
            Left: optimization progress as a function of iteration number (x-axis).
            The solid curve shows the per-iteration estimate
            $\widetilde{\Expect{}{\obj}}$ of the stochastic objective $\stochobj$,
            averaged over the gradient sample of size $100$.
            The dashed curve shows the best (largest) $\obj$ value found in that
            iteration's sample.
            The red square marks the best policy sample from the \emph{terminal}
            policy $\hyperparam\opt$; the purple triangle marks the best overall
            sample across \emph{all} iterations---when the two differ, the triangle
            provides a better fallback design in case of convergence to a local optimum.
            A box plot shows a uniform random sample of $1000$ feasible designs
            ($\hyperparam_i=0.5$, $i=1,\ldots,\Nsens$) for reference.
            Right: values of each entry of $\hyperparam$ over iterations.
            }
          \label{fig:OED_AD_SIZE_20_BUDGETS_10_1}
        \end{figure}

        \Cref{fig:OED_AD_SIZE_20_BUDGETS_10_1} (left) also shows that 
        as \Cref{alg:probabilistic_binary_optimization} iterates, 
        the variance of the stochastic gradient sample reduces, 
        indicating progression toward a local optimum.
        While \Cref{alg:probabilistic_binary_optimization} produces results 
        (e.g., an optimal solution) slightly
        better than the uniform random sample, 
        the algorithm beats the best random sample after only one iteration.
        Note that one iteration of the algorithm costs $\Nens=100$ evaluations of 
        the objective $\obj$ while the size of the uniform random sample here is $1000$. 
        The superiority of the proposed approach becomes clear as 
        the dimensionality of the problem increases as discussed below.
     
        The algorithm terminates after $307$ iterations; however,
        the major updates to the parameter happen at the first few iterations.
        This is demonstrated by \Cref{fig:OED_AD_SIZE_20_BUDGETS_10_1} (right),
        as well as \Cref{fig:OED_AD_SIZE_20_BUDGETS_10_2} (left) which shows the
        step update over consecutive iterations.
        Moreover, the algorithm keeps track of sampled realizations of $\design$
        along with the corresponding objective values $\obj(\design)$.
        As the algorithm proceeds, the probabilities are updated and are generally pushed
        toward the bounds $\{0, 1\}$, thus promoting realizations of $\design$ it has
        previously explored.
        This is demonstrated by \Cref{fig:OED_AD_SIZE_20_BUDGETS_10_2} (right), which
        shows the number of new evaluations of the objective $\obj$ at each iteration.

        \begin{figure}[htbp!]
          \centering
          \includegraphics[width=0.48\textwidth]{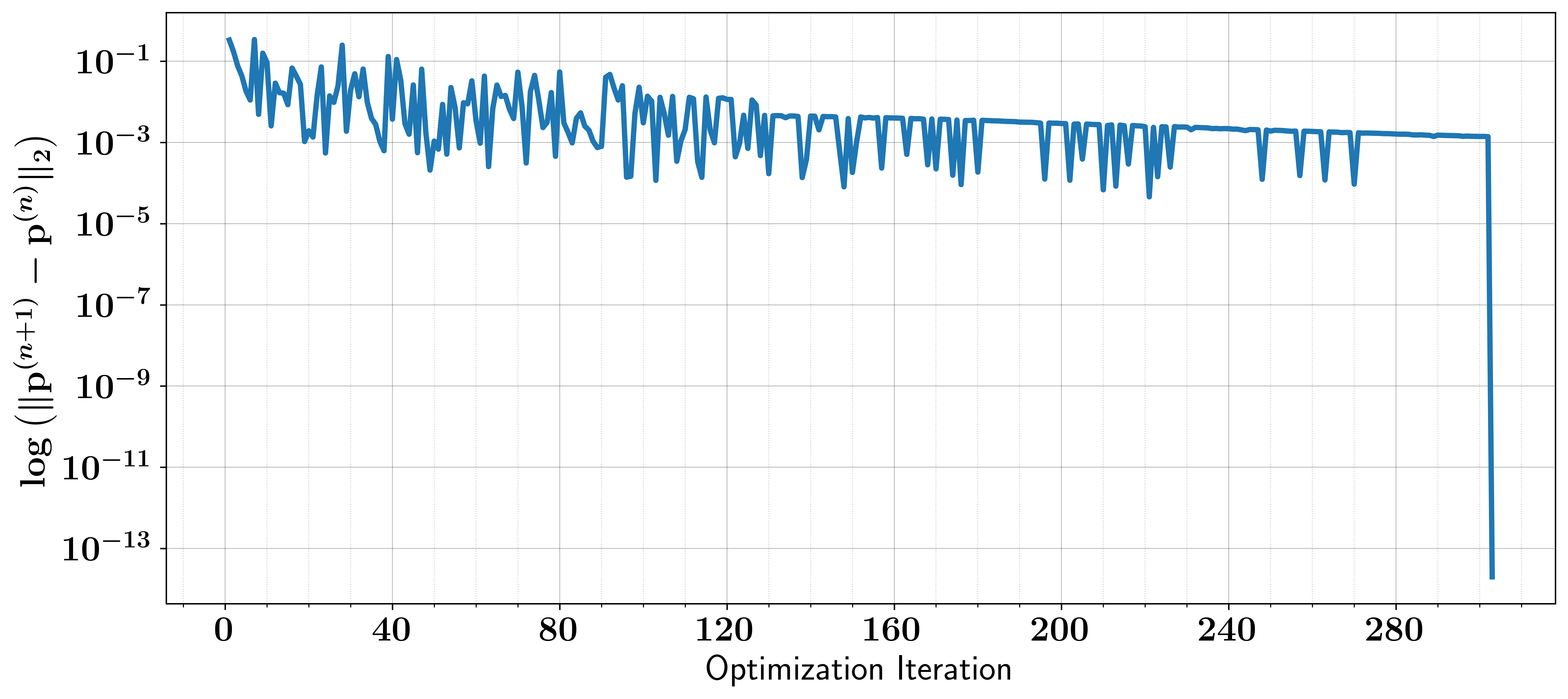}
          \hfill
          \includegraphics[width=0.49\textwidth]{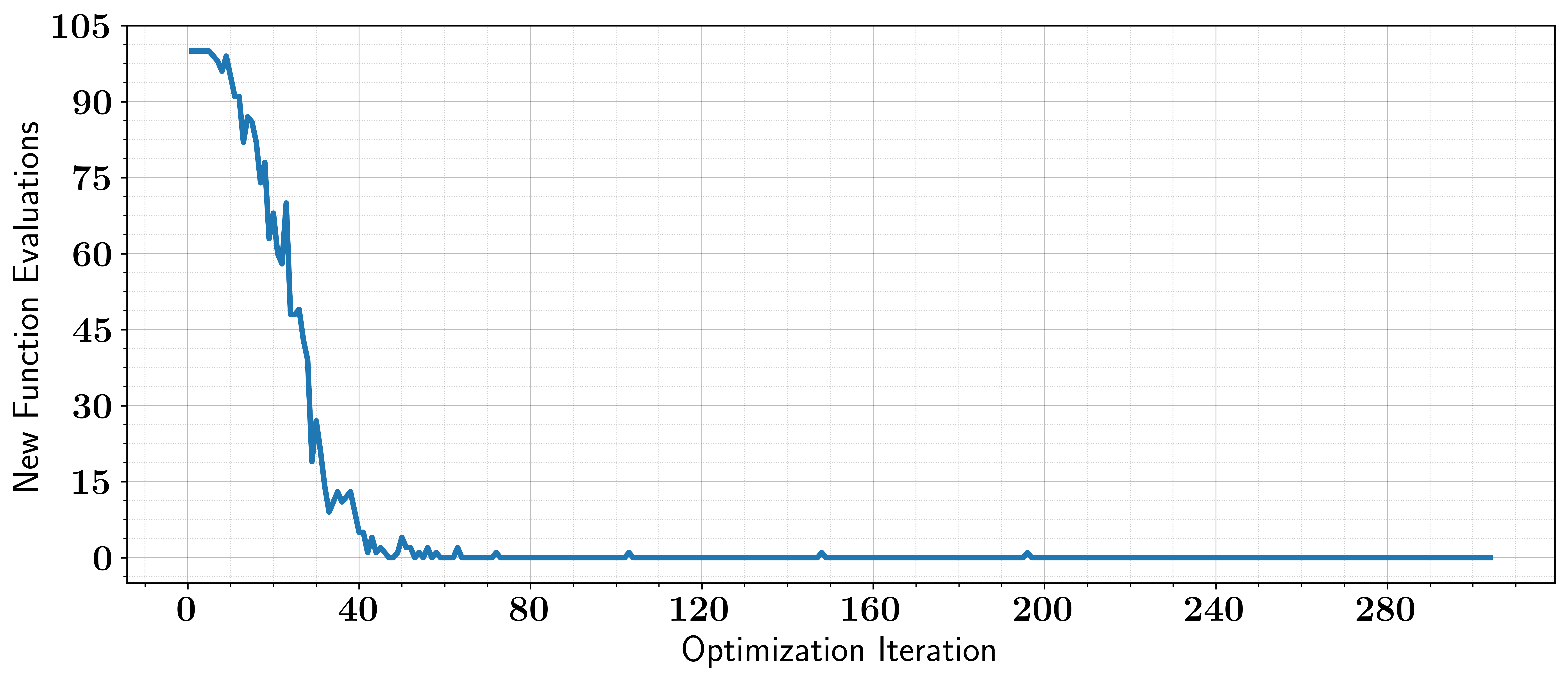}
          \caption{Left: 
            Norm of the update 
            over consecutive  iterations.
            Right: Number of new function evaluations, that is, 
            the number of evaluations of $\obj$ for realizations of $\design$ that have 
            not been previously explored by the optimizer. 
          }
          \label{fig:OED_AD_SIZE_20_BUDGETS_10_2} 
        \end{figure}

        Here the number of feasible designs 
        is $\binom{20}{10}=184,756$, which enables conducting a brute-force 
        search to benchmark our results as shown in 
        \Cref{fig:OED_AD_SIZE_20_BUDGETS_10_4}.
        \Cref{fig:OED_AD_SIZE_20_BUDGETS_10_4} (left) shows that the sample drawn
        from the optimal policy $\hyperparam\opt$ concentrates near the global optimum
        value (dashed line), and the best policy sample (red square) attains 
        the global optimum.
        In fact, the sample (of size $100$) generated from the optimal policy contain only 
        two unique designs with objective values almost identical to the global optimal value.
        These two experimental designs 
        correspond to the sensor placements shown in \Cref{fig:OED_AD_SIZE_20_BUDGETS_10_4} 
        (middle and right). The sensor placement is overlaid on the interpolated 
        policy parameter $\hyperparam\opt$ highlighting that the optimal 
        policy is almost degenerate.
        
        \begin{figure}[htbp!]
          \centering
          \includegraphics[width=0.48\textwidth]{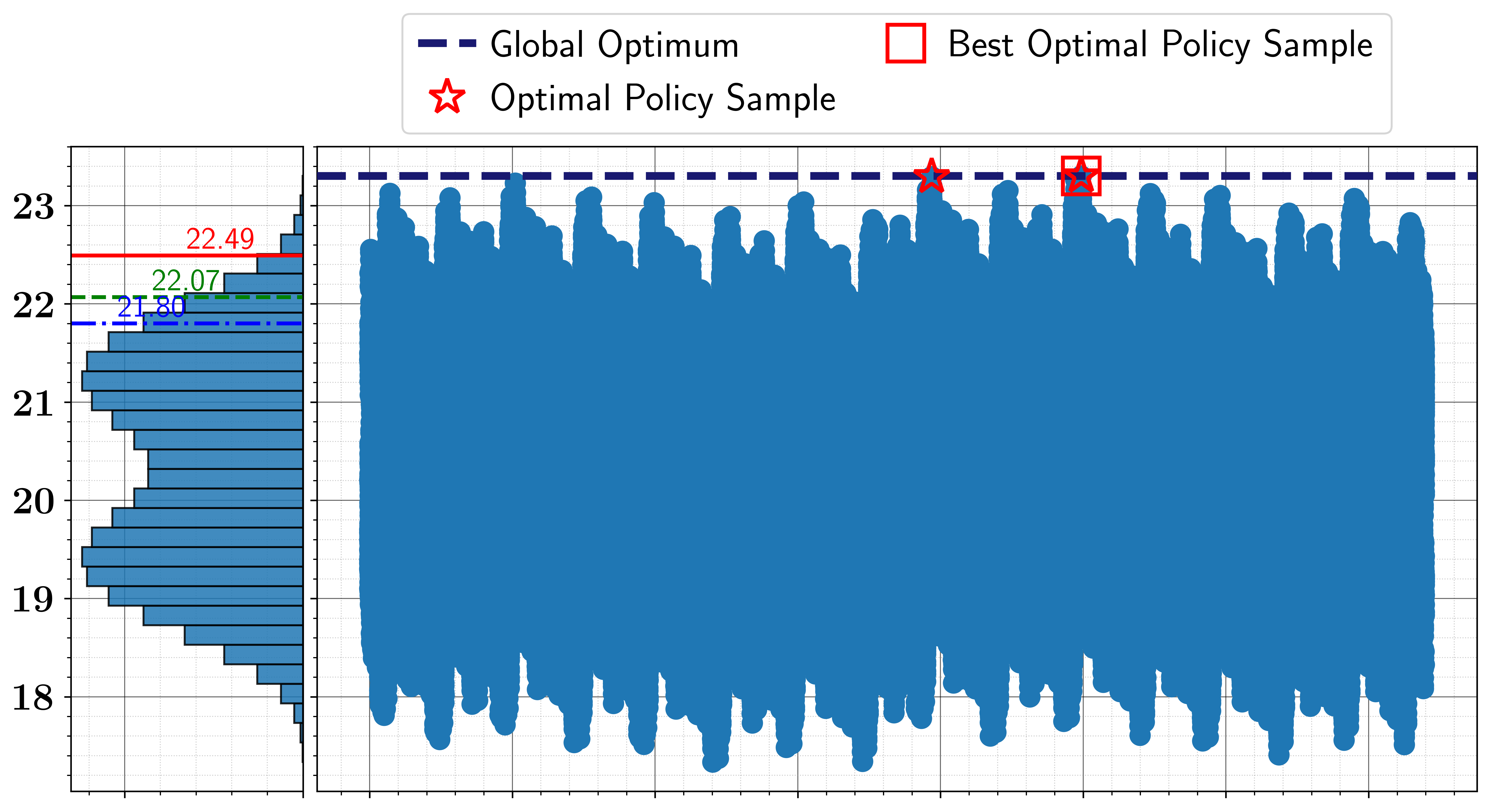}
          \hfill
          \includegraphics[width=0.25\textwidth]{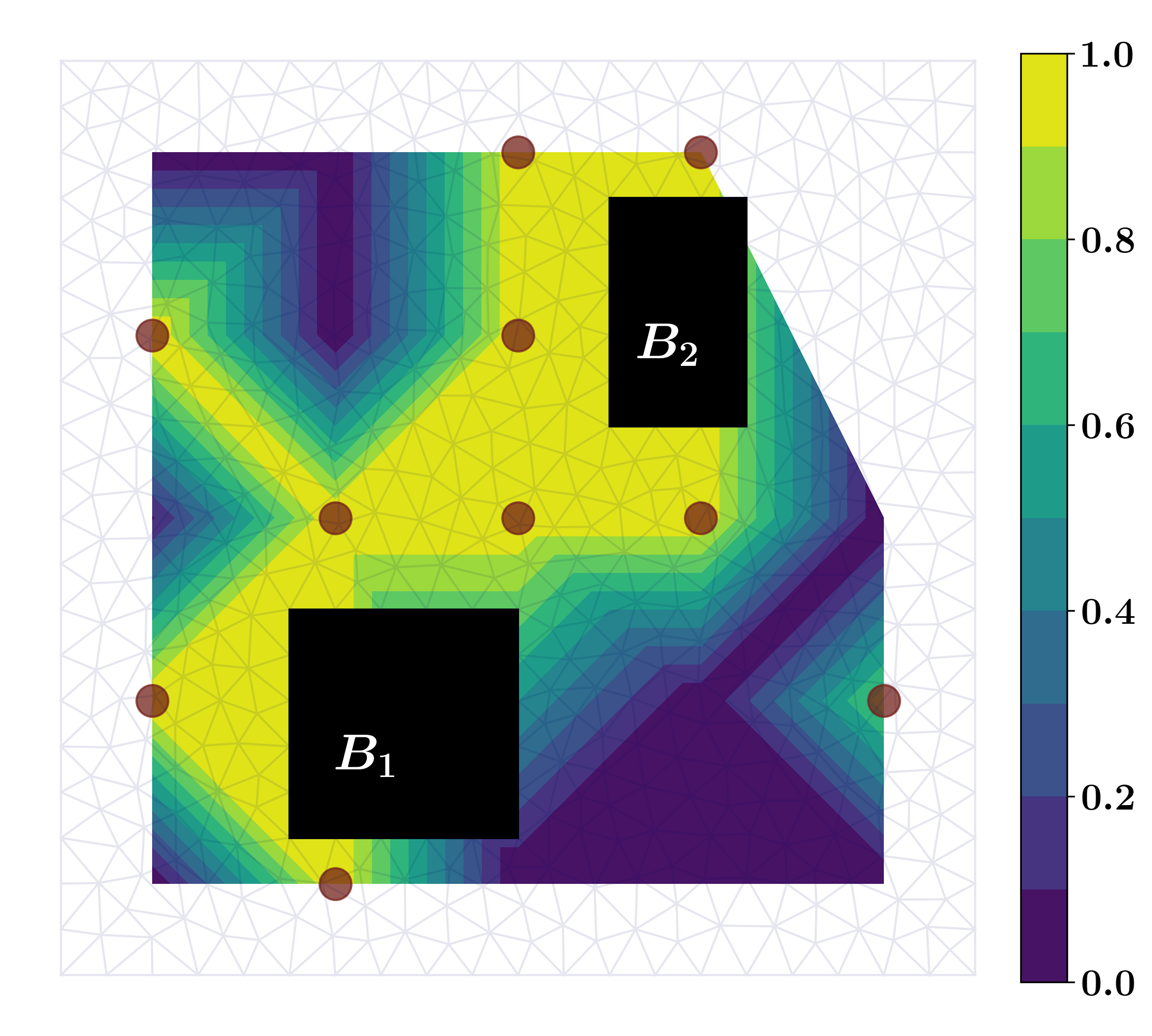}
          \hfill
          \includegraphics[width=0.25\textwidth]{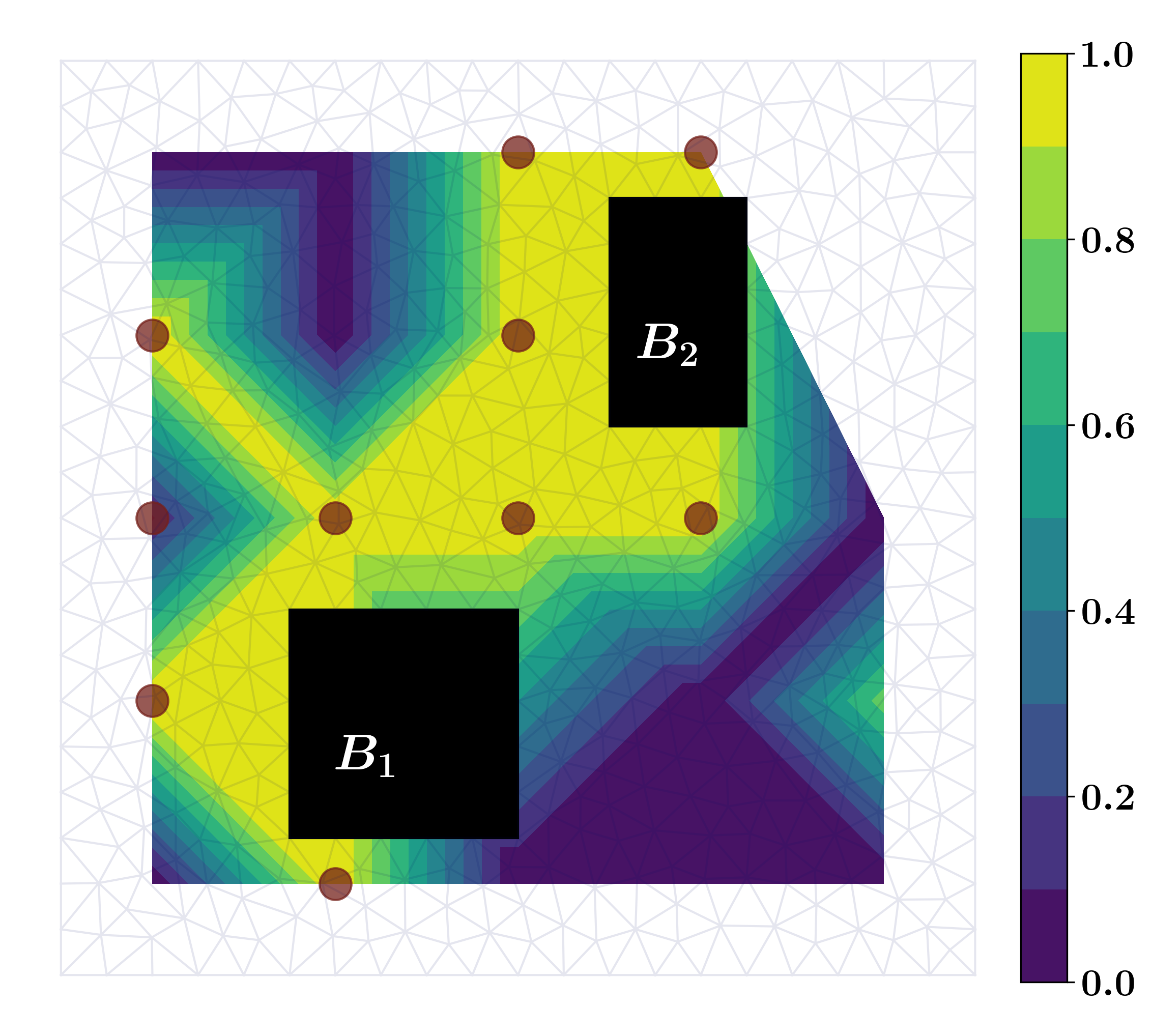}
          \caption{Results of applying \Cref{alg:probabilistic_binary_optimization} to
            solve \eqref{eqn:OED_AD_optimization_equality_max} compared with
            the brute-force search of all $\binom{20}{10}=184{,}756$ feasible designs.
            Left: objective value of every feasible $\design$ (blue dots);
            a sample of size $100$ drawn from the optimal policy $\hyperparam\opt$
            is shown as red stars, and the best realization among that sample
            (best policy sample) is marked by a red square.
            The dashed line indicates the global optimum value.
            Middle: global optimum solution which is also discovered by sampling the optimal policy.
            Right: the second unique sample generated by sampling the optimal policy.
            Both samples yield almost identical objective value with optimality gaps $0$, 
            and $3.8\times10^{-9}$,
            respectively.
            The two designs are plotted as sensor placement overlaid on the interpolated optimal policy
            parameter $\hyperparam\opt$.
            }
          \label{fig:OED_AD_SIZE_20_BUDGETS_10_4}
        \end{figure}
        %

    \subsection{Performance and scalability.}
    \label{subsec:scalability}
      Here we study the scalability of the proposed approach in terms of both 
      the performance and the computational cost.
      First we allow increasing design space $\Nsens$ while keeping the budget
      size $\designsumval$ fixed to analyze the performance of the proposed 
      approach with increasing design space dimensionality. 
      Additionally, we fix the design space dimensionality and increase the
      budget size in a computational cost study.
        
      \subsubsection{Increasing design dimensionality.}
        We run \Cref{alg:probabilistic_binary_optimization}
        for increasing cardinality $\Nsens$.
        Results obtained for $\Nsens$ set to $100$, $200$, and $300$ are shown in
        \Cref{fig:OED_AD_SIZE_100_BUDGETS_10},
        \Cref{fig:OED_AD_SIZE_200_BUDGETS_10}, and
        \Cref{fig:OED_AD_SIZE_300_BUDGETS_10}, respectively.
        These results show consistent behavior of the optimization procedure as the
        size of the problem (number of candidate sensor locations) increases.
        Moreover, the results show that the most likely place (highly probable) to place sensors
        is near the center of the domain concentrated at the sides of the two buildings 
        which is consistent with the results obtained with the coarse experiments; as shown
        in \Cref{fig:OED_AD_SIZE_20_BUDGETS_10_4}.
        The values of the CB model parameters (success probabilities) 
        $\hyperparam\opt$ returned by the optimizer and interpolated over the 
        domain show that the resulting policy is non-degenerate,
        which indicates that many candidate designs achieve an objective value 
        near the global optimum value.

        \begin{figure}[htbp!]
          \centering
          \includegraphics[width=0.50\textwidth]{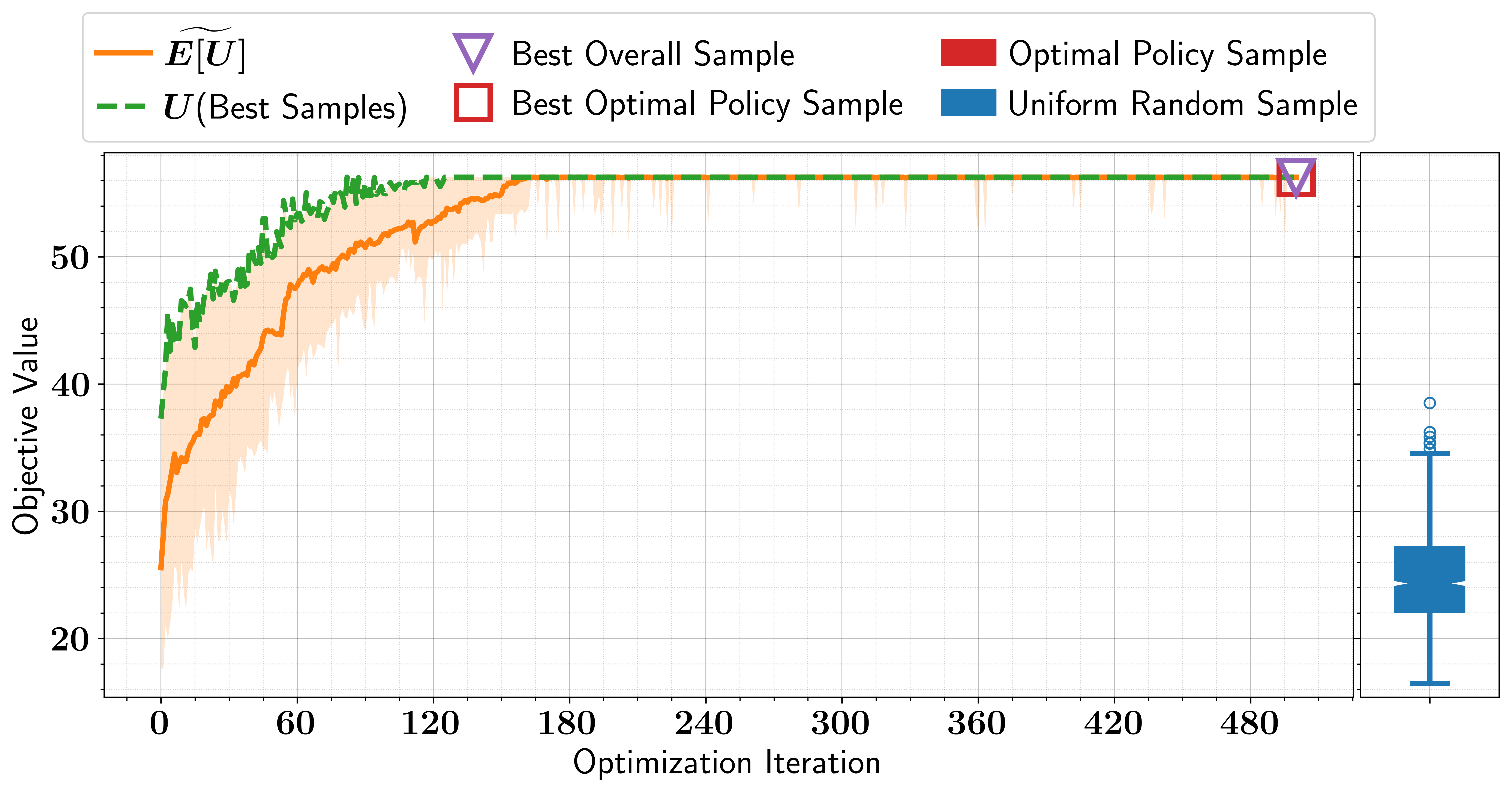}
          \qquad
          \includegraphics[width=0.30\textwidth]{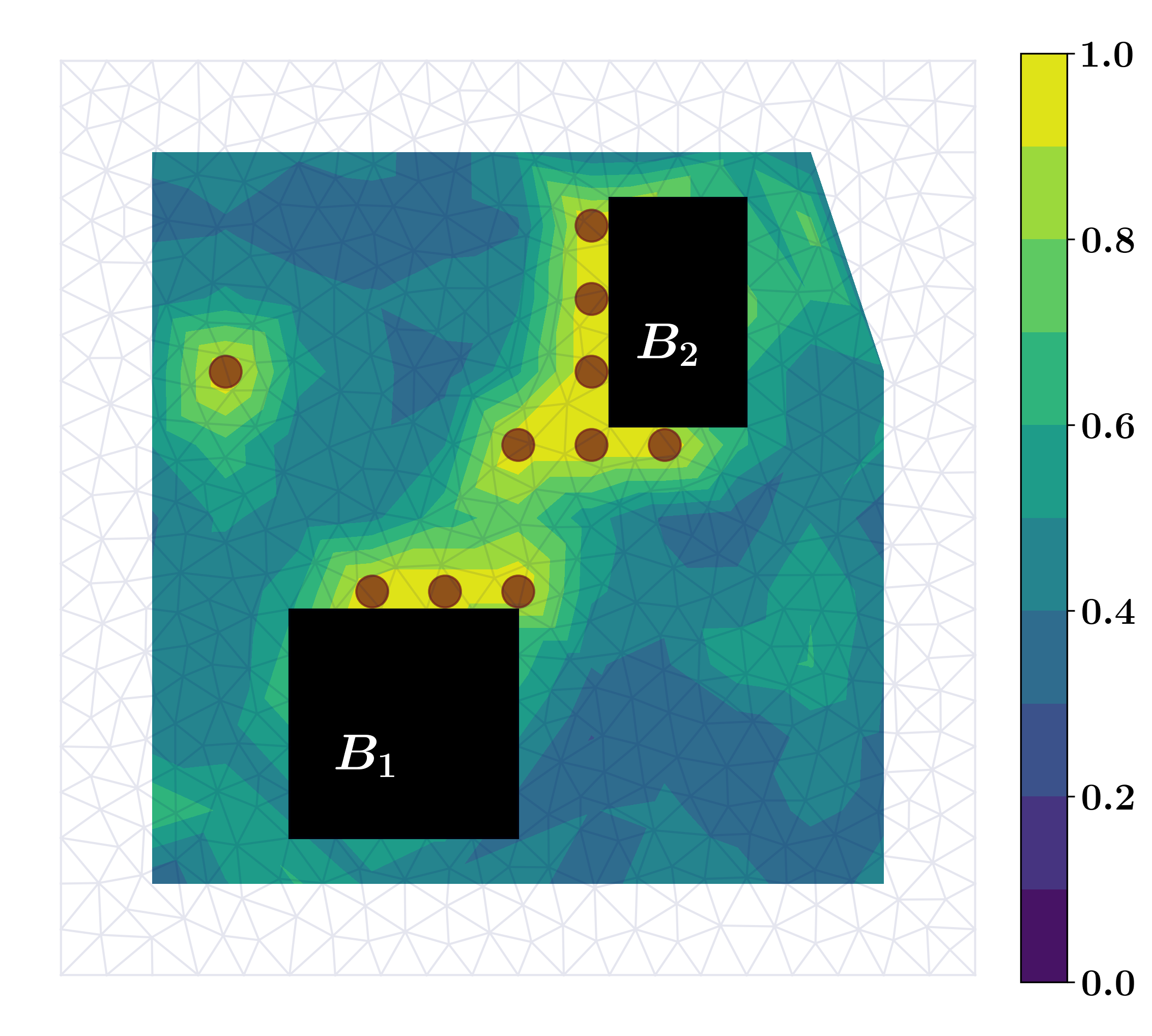}
          \caption{Behavior of \Cref{alg:probabilistic_binary_optimization} for
            solving \eqref{eqn:OED_AD_optimization_equality_max} with $\Nsens=100$.
            Left: optimization progress (see \Cref{fig:OED_AD_SIZE_20_BUDGETS_10_1} for details).
            Right: the best policy sample overlaid on the interpolated policy
            parameter $\hyperparam\opt$.
            }
          \label{fig:OED_AD_SIZE_100_BUDGETS_10} 
        \end{figure}
        \begin{figure}[htbp!]
          \centering
          \includegraphics[width=0.50\textwidth]{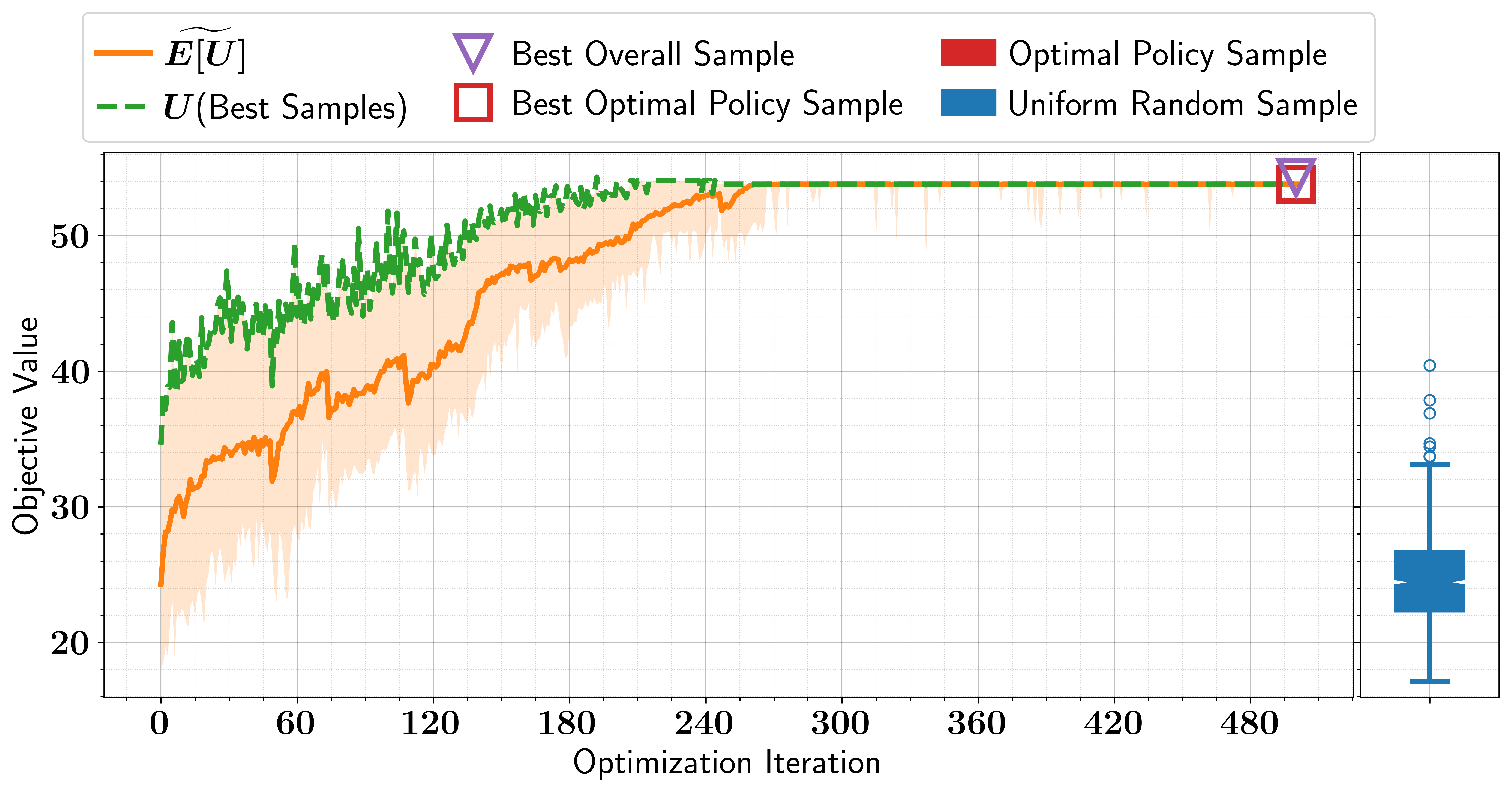}
          \quad
          \includegraphics[width=0.30\textwidth]{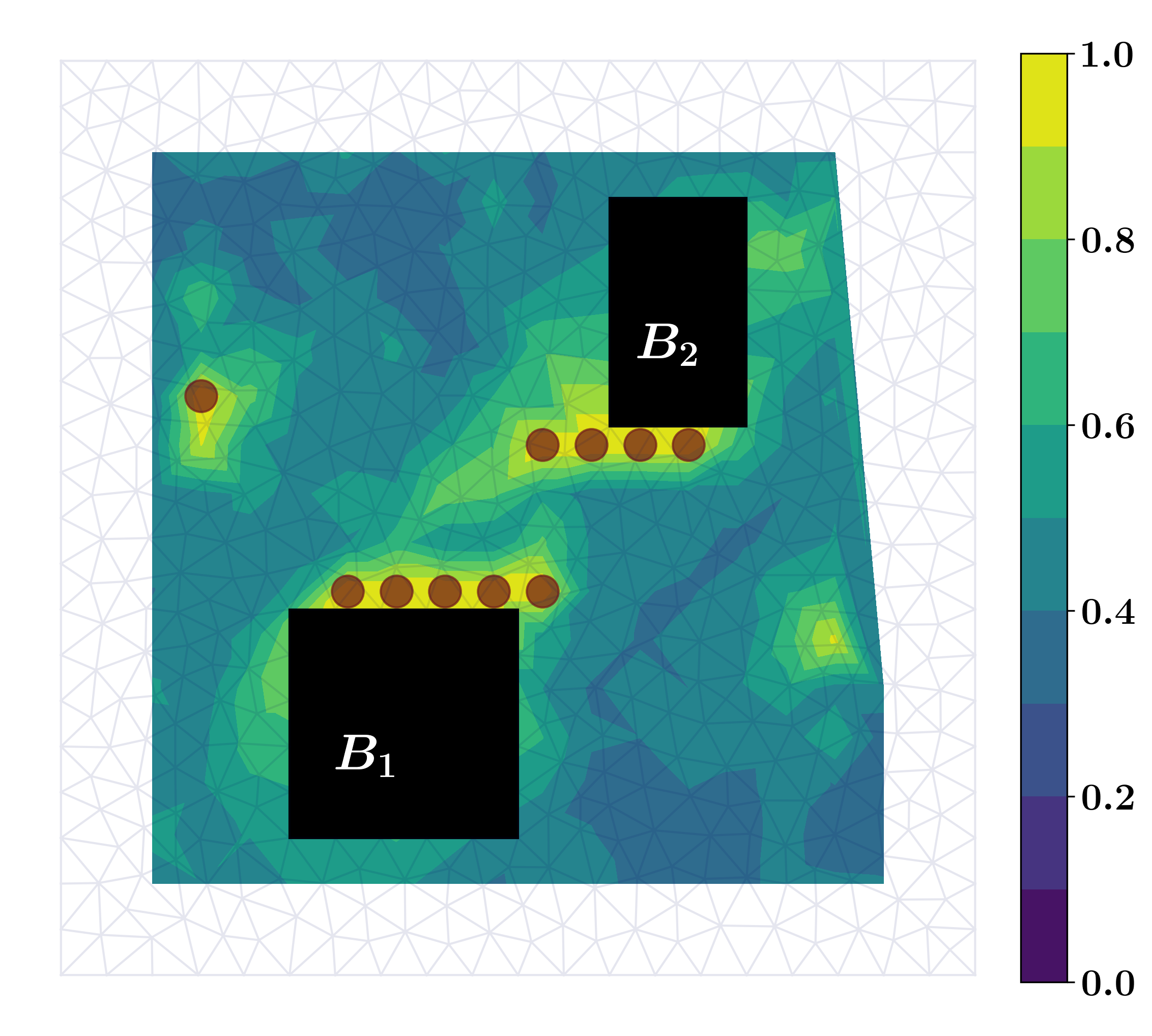}
          \caption{Same as \Cref{fig:OED_AD_SIZE_100_BUDGETS_10} with $\Nsens=200$.}
          \label{fig:OED_AD_SIZE_200_BUDGETS_10} 
        \end{figure}
        \begin{figure}[htbp!]
          \centering
          \includegraphics[width=0.50\textwidth]{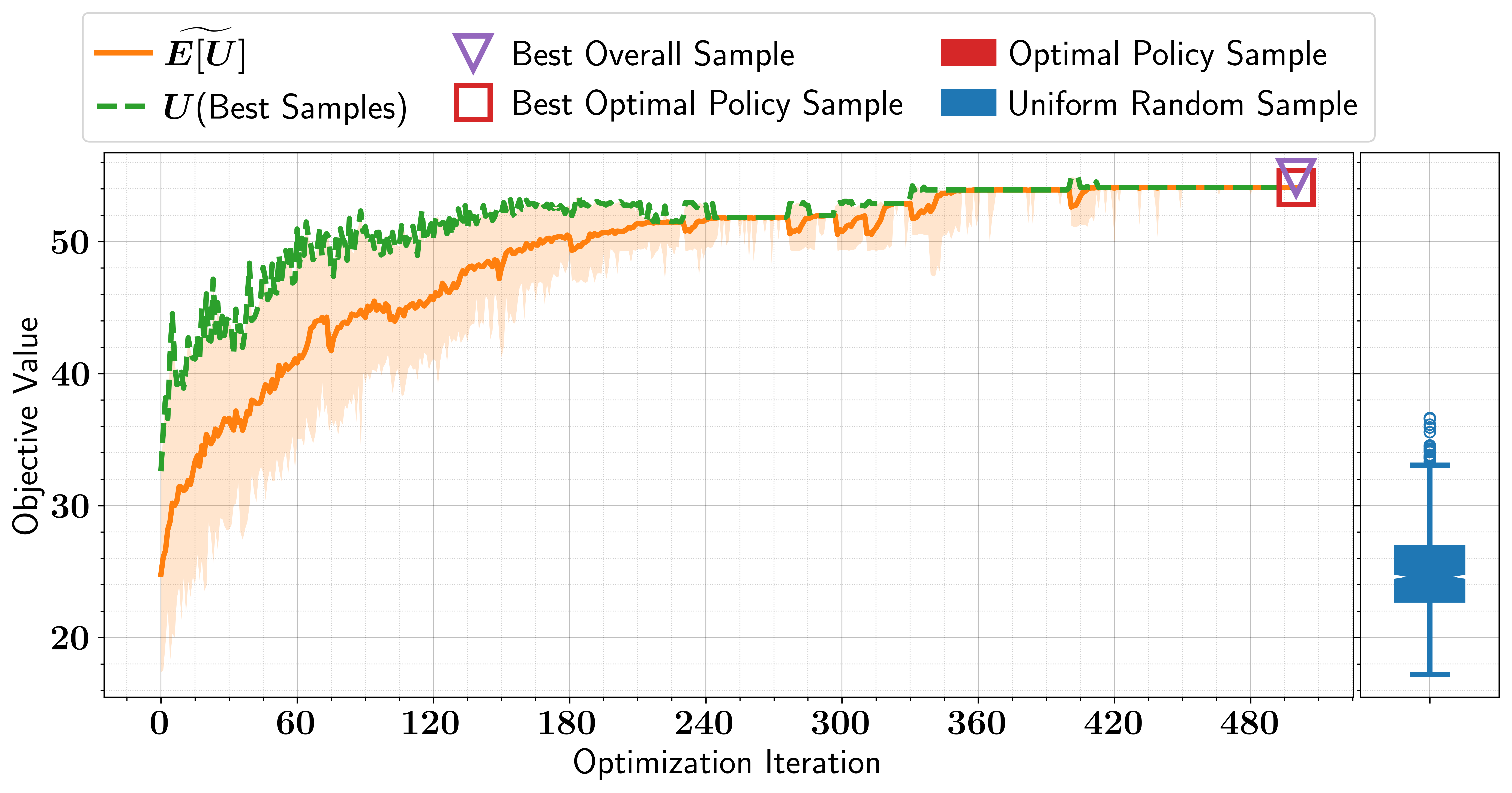}
          \quad
          \includegraphics[width=0.30\textwidth]{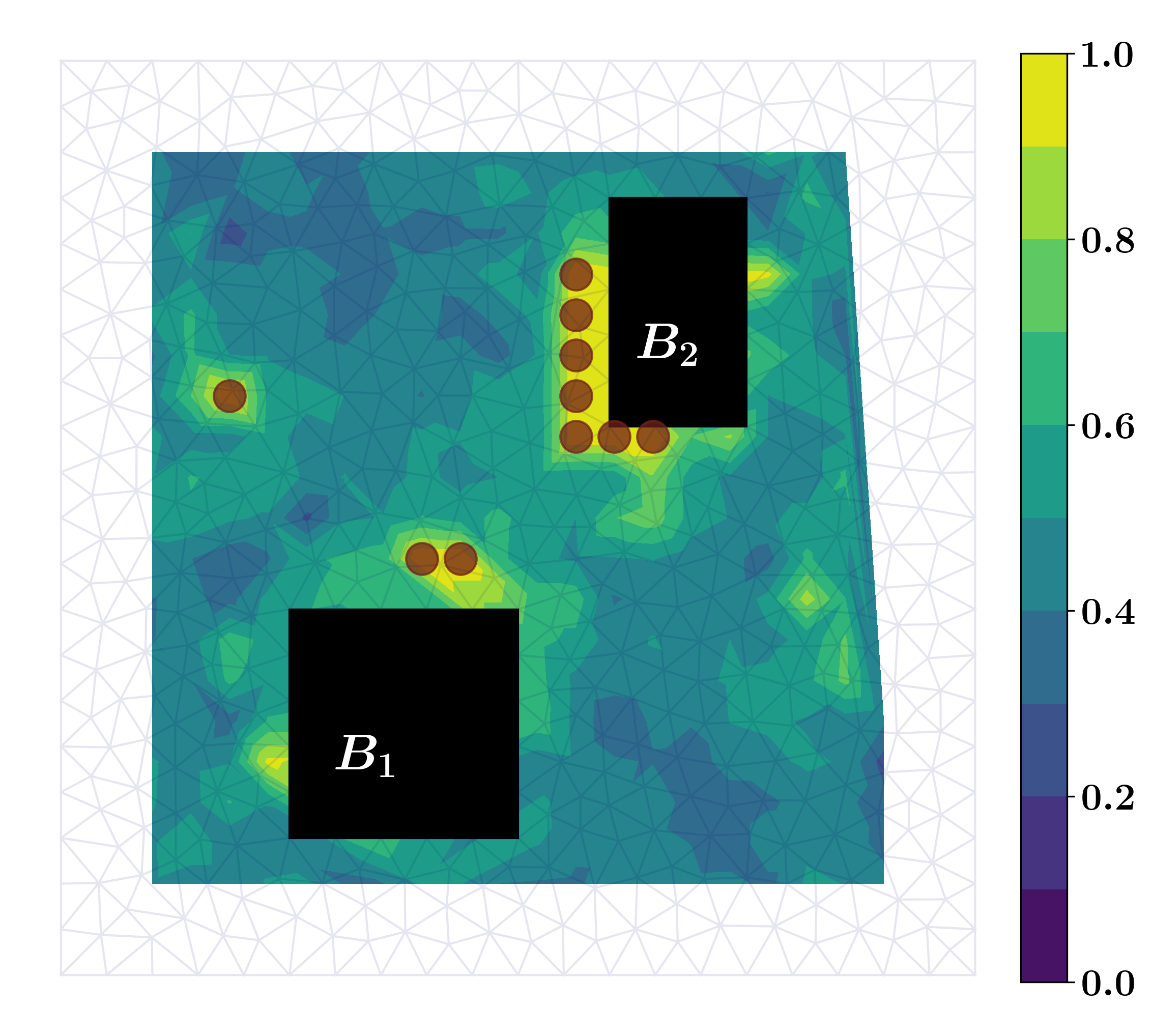}
          \caption{Same as \Cref{fig:OED_AD_SIZE_100_BUDGETS_10} with $\Nsens=300$.}
          \label{fig:OED_AD_SIZE_300_BUDGETS_10} 
        \end{figure}

        \Cref{fig:OED_AD_SIZE_300_BUDGETS_10} shows that keeping track the best design explored at 
        each iteration can be beneficial as a fall back when the algorithm converges to a local optimum.
        While the optimal designs found by
        \Cref{alg:probabilistic_binary_optimization} are not guaranteed to coincide with 
        the global optima, the best policy samples consistently outperform uniform 
        random sampling by a large margin.
        Moreover, in spite of setting the maximum number of iterations to $500$,
        the optimization algorithm converges quickly to a local optimum for 
        small as well as large dimensionality. 

      \subsubsection{Computational cost.}
      \label{subsubsec:computational_cost}
        The computational cost of \Cref{alg:probabilistic_binary_optimization}
        is predetermined by the size of the sample used in estimating the 
        gradient (the stochastic gradient)
        at each iteration and by the number of iterations 
        of the optimization procedure.
        Here we focus on the computational cost per iteration.

        The asymptotic per-iteration cost is derived in
        \Cref{subsec:computational_considerations}; here we report an
        empirical confirmation on the present advection--diffusion
        setup.

        \begin{figure}[htbp!]
          \centering
          \includegraphics[width=0.95\textwidth]{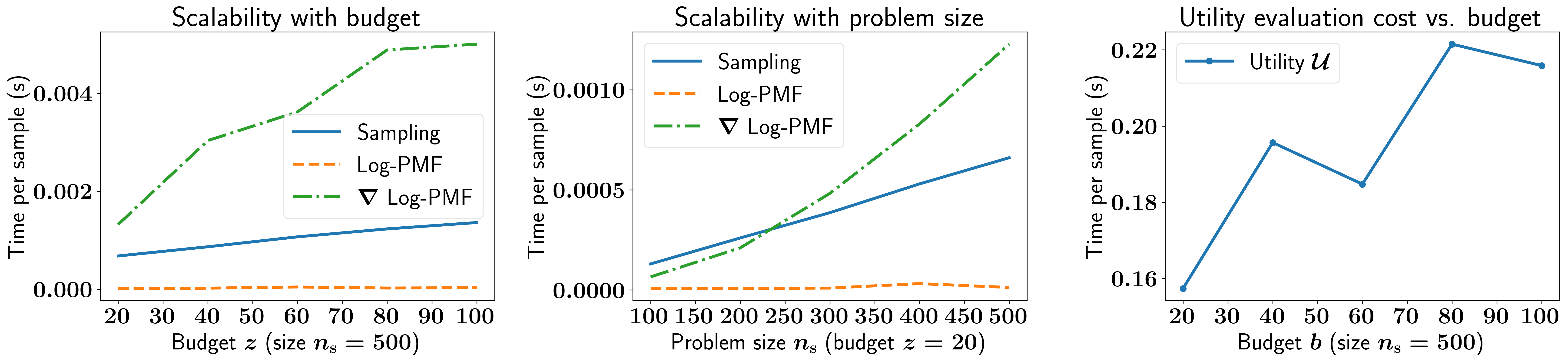}
          \caption{
            Scalability of \Cref{alg:probabilistic_binary_optimization}.
            Left: wall-clock time for the three algorithmic
              components ---CB model sampling, 
              log-PMF evaluation, and score $\nabla_{\hyperparam}\log\CondProb{\design[k]}{\designsum}$
              computation---as a function of $\designsumval$ 
              where the cardinality is set to $\Nsens=500$.
            Middle: wall-clock times for increasing cardinality $\Nsens$ while
              fixing the budget size to $\designsumval=20$.
            Right: wall-clock times evaluation of the objective $\obj$ 
              as a function of $\designsumval$ 
              where the cardinality is set to $\Nsens=500$.
          }
          \label{fig:budget_scalability}
        \end{figure}

        \Cref{fig:budget_scalability} reports the wall-clock time for the three algorithmic
        components — CB model sampling, log-PMF evaluation, and score
        $\nabla_{\hyperparam}\log\CondProb{\design[k]}{\designsum}$ computation — as a function of $\designsumval$
        where the cardinality is set to $\Nsens=500$.
        The middle panel shows the time for increasing dimensionality $\Nsens$ while 
        keeping the budget fixed to $\designsumval=20$.
        In both cases, sampling and log-PMF evaluation are essentially flat while 
        the score computation grows only very mildly remaining negligible compared
        to the cost of evaluating the objective $\obj$ which is shown on the right panel.
        The right panel shows the wall-clock time for a single evaluation of
        $\obj$ which grows modestly with $\designsumval$ because larger budgets
        activate more sensors and increase the cost of the forward solve.
        Crucially, this growth reflects the cost of evaluating the OED criterion
        itself — not any overhead introduced by the budget constraint in the
        proposed approach.
        Note that we do not show the cost of evaluating the objective $\obj$
        for increasing cardinality $\Nsens$ (middle panel), because
        the cost of the forward/adjoint solve depends on the number of active sensors
        (i.e., the budget $\designsumval$), not on the total number of candidates $\Nsens$.

        The results shown here and the fact that the stochastic gradient is data parallel
        show that the proposed approach is computationally efficient and scalable 
        and thus is
        suitable for challenging applications such as sensor placement in large-scale
        data assimilation and inverse problems.

    \subsection{On the necessity of variance reduction}
    \label{sec:numerical_baseline_variance}
      The theoretical analysis in \Cref{subsec:variance_reduction} establishes that
      variance reduction via the optimal baseline
      \eqref{eqn:optimal_scalar_baseline_theory} and its per-component
      generalization \eqref{eqn:optimal_percomponent_baseline_theory}
      can significantly reduce the variance of the stochastic gradient
      estimator without any additional evaluations of the objective $\obj$.
      Here we verify this empirically and examine
      how the variance depends on the sample size $\Nens$
      and the state of $\hyperparam$ during optimization.

      \Cref{fig:baseline_variance} illustrates the effect of the baseline on the
      empirical total gradient variance
      $\Trace{\widehat{\mathrm{Cov}}[\widehat{\vec{g}}]}$
      as a function of sample size $\Nens$ (with $N_b=1$ mini-batch),
      for two representative initializations of $\hyperparam$:
      a \emph{uniform} initialization ($p_i = 0.5$ for all $i$,
      corresponding to an uninformative starting point) and a \emph{U-shaped}
      distribution (parameters concentrated near $0$ or $1$,
      representative of a near-converged solution).
      Similar results were obtained for larger values of $N_b$ and are omitted for brevity.

      \begin{figure}[htbp!]
        \centering
        \includegraphics[trim={0 0 16cm 0}, clip, width=0.48\textwidth]{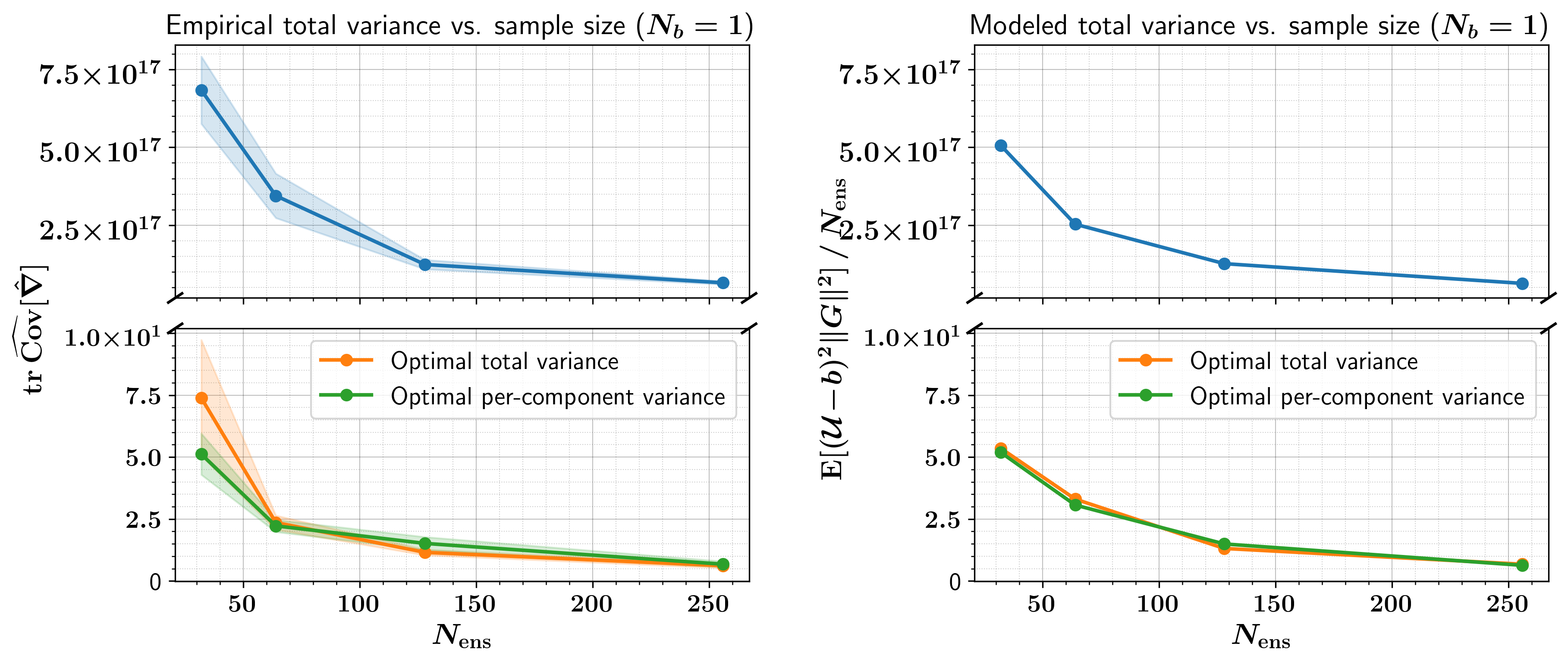}
        \hfill
        \includegraphics[trim={0 0 16cm 0}, clip, width=0.48\textwidth]{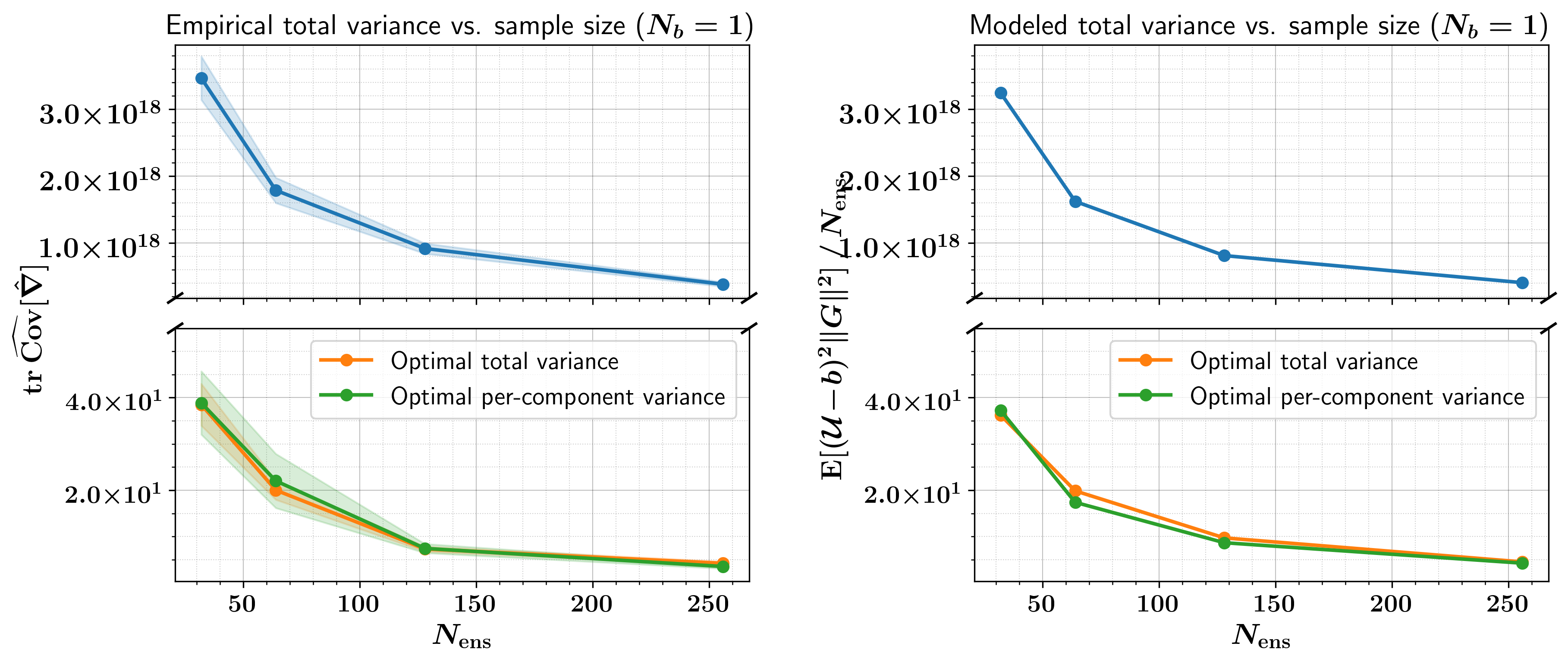}
        \caption{
          Empirical total gradient variance
          $\Trace{\widehat{\mathrm{Cov}}[\widehat{\vec{g}}]}$
          as a function of sample size $\Nens$ (with $N_b = 1$ mini-batch).
          Each panel is a $2\times 1$ grid with a broken $y$-axis:
          the top sub-panel shows the variance without any baseline;
          the bottom sub-panel compares the optimal total-variance baseline
          $\baseline\opt$~\eqref{eqn:optimal_scalar_baseline_theory}
          and the optimal per-component baseline
          $\baseline\opt_j$~\eqref{eqn:optimal_percomponent_baseline_theory}.
          \emph{Left}: uniform initialization ($p_i = 0.5$).
          \emph{Right}: U-shaped initialization (parameters near $0$ or $1$),
          representative of a near-converged solution.
          The variance reduction is more pronounced in the U-shaped case because
          the score components $\nabla_{p_j}\log\CondProb{\design[k]}{\designsum}$
          vary widely across coordinates near convergence.
        }
        \label{fig:baseline_variance}
      \end{figure}

      These results confirm that by using the same sample already drawn for the
      stochastic gradient to estimate the optimal baseline — at no additional
      computational cost — the gradient variance is reduced dramatically,
      with the per-component baseline consistently outperforming the scalar baseline
      especially for small sample sizes $\Nens$.
      
    \subsection{Bayesian A-optimal design}
    \label{sec:numerical_A_minimization}
      %
      The previous experiments all adopt the classical A-optimal design defined by 
      the maximization problem 
      \eqref{eqn:OED_AD_optimization_equality_max}.  
      Here we consider the complementary formulation of A-optimal 
      designs typically employed in Bayesian OED \cite{AlexanderianPetraStadlerEtAl16}. 
      Specifically, the Bayesian A-optimal design is defined as the \emph{minimizer} 
      of the posterior covariance matrix trace. 
      In this setup all algorithmic parameters are identical to those in 
      \Cref{subsec:AD_Setup}; only the optimization direction changes (minimization instead of maximization),
      which requires no modification to \Cref{alg:probabilistic_binary_optimization}.
      The optimization objective, here the inverse of the posterior covariance matrix
      remains a black box to the probabilistic optimization algorithm.

      This setup will be used directly in the comparison with design-space relaxation
      (\Cref{sec:numerical_comparison_relaxation}), where we observed that
      minimization provides a more discriminating test of the two approaches
      at increasing problem dimensions.
    
      The optimal sensor placement problem is thus defined as: 
      \begin{equation}\label{eqn:OED_AD_optimization_equality_min}
        \argmin_{\design \in \{0, 1\}^{\Nsens} }{
          \obj(\design) :=  \Trace{\!
           \F\adj \Pseudoinv{\Diag{\design}\Cobsnoise\Diag{\design}} \F
           +\Cparampriormat^{-1}
          \!}^{-1}
        }
        \quad \textrm{s.t.} \quad   
            \wnorm{\design}{0} = 10 
          \,.
      \end{equation}
      \begin{figure}[htbp!]
        \centering
        \includegraphics[width=0.45\textwidth]{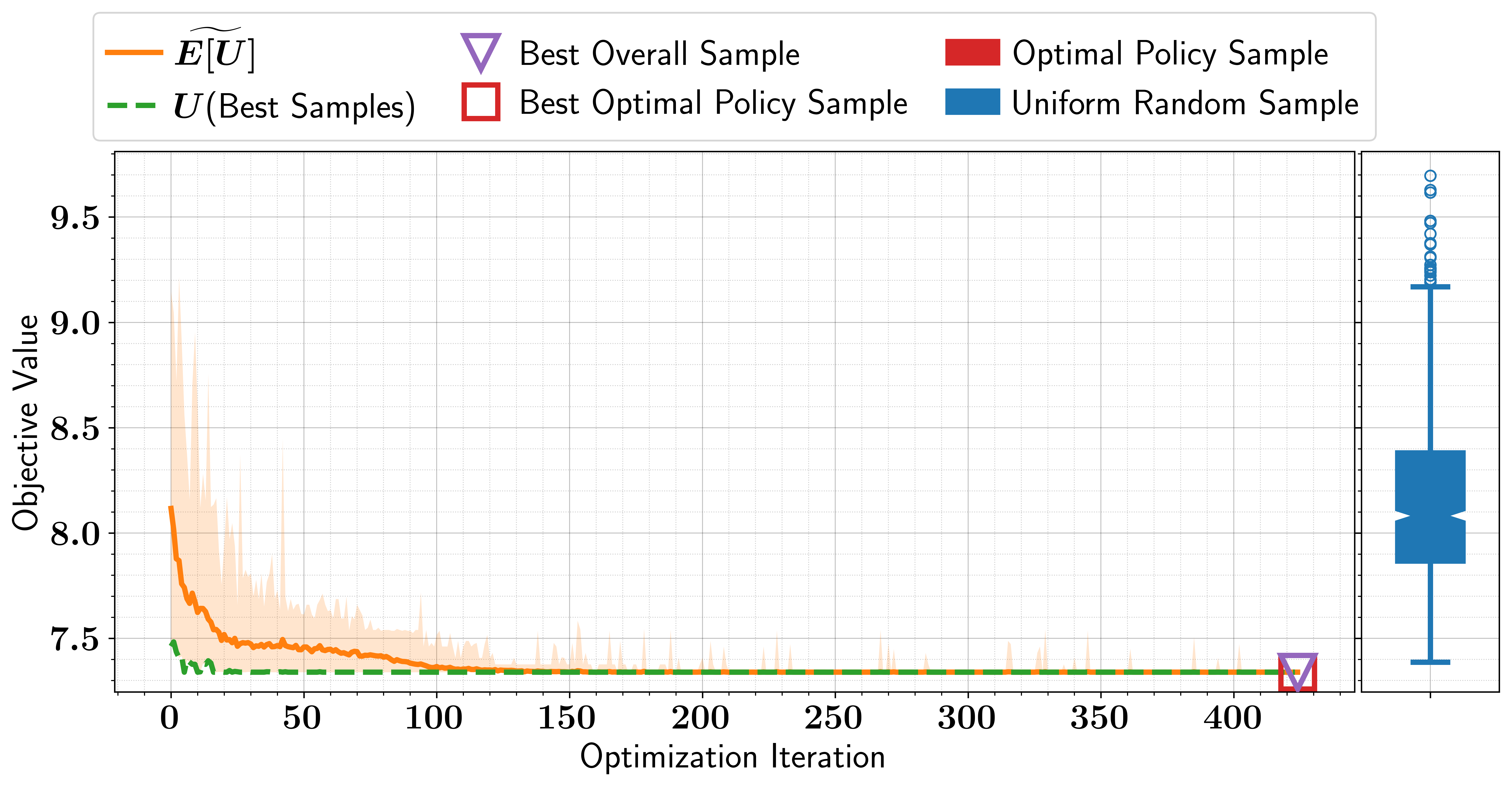}
        \hfill
        \includegraphics[width=0.45\textwidth]{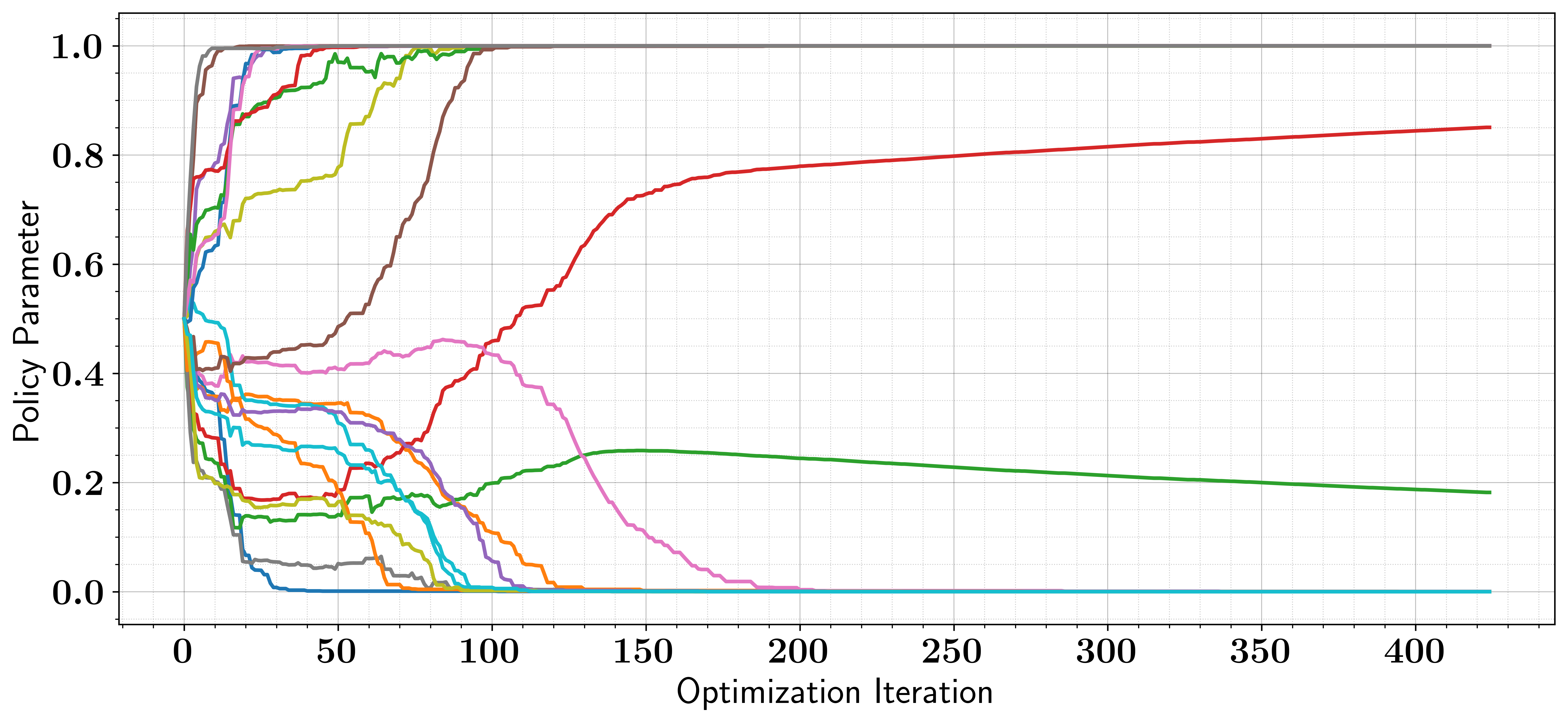}
        \caption{
          A-optimality minimization with $\Nsens=20$ candidate sensors and
          budget $\designsumval=10$.
          Left: optimization progress over iterations; the box plot shows a uniform
          random sample of $1000$ feasible designs for reference.
          Right: evolution of the CB model parameter $\hyperparam$ over iterations.
        }
        \label{fig:OED_AD_A_Min_iter}
      \end{figure}
      \begin{figure}[htbp!]
        \centering
        \includegraphics[width=0.48\textwidth]{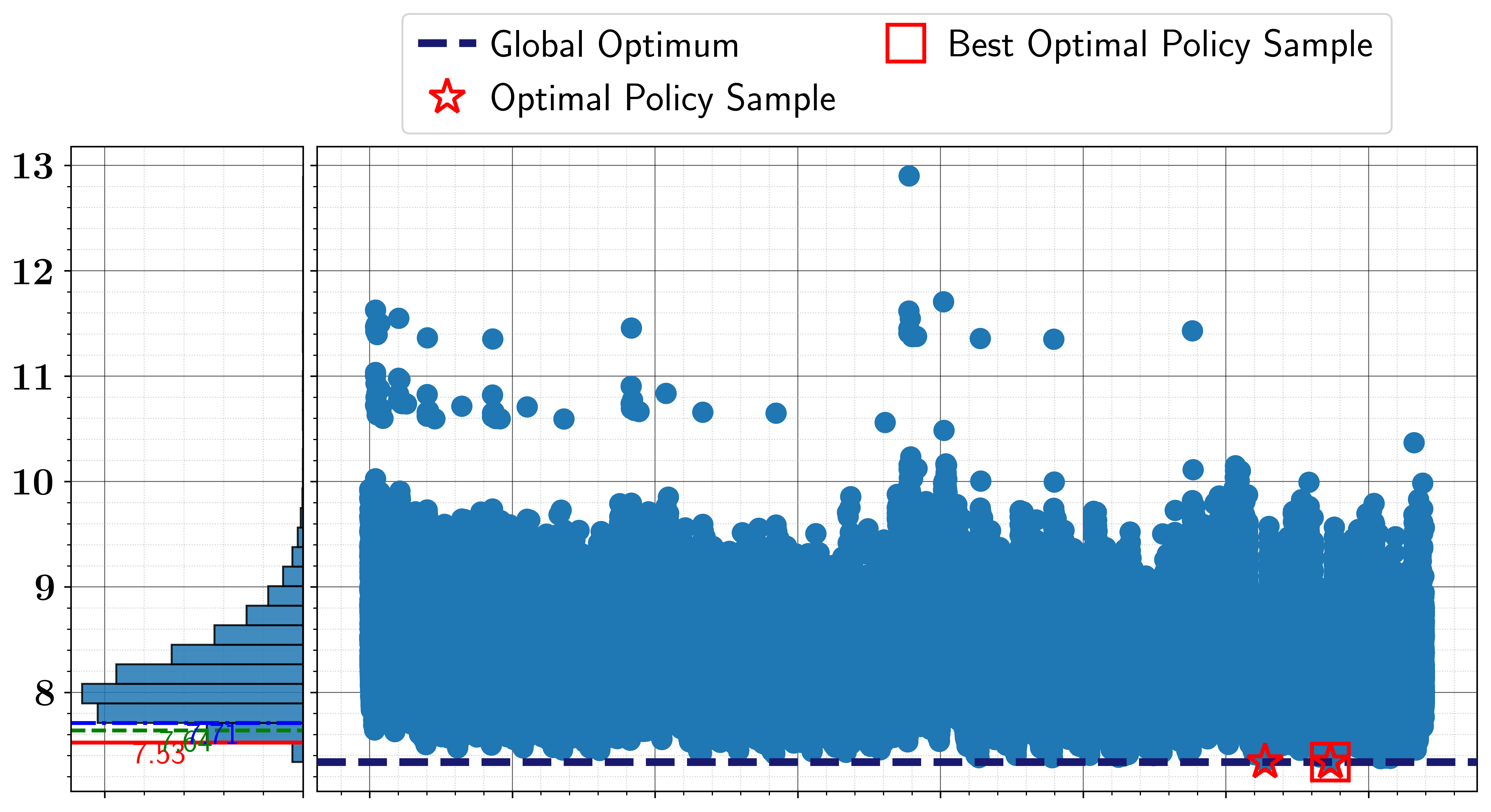}
        \hfill
        \includegraphics[width=0.25\textwidth]{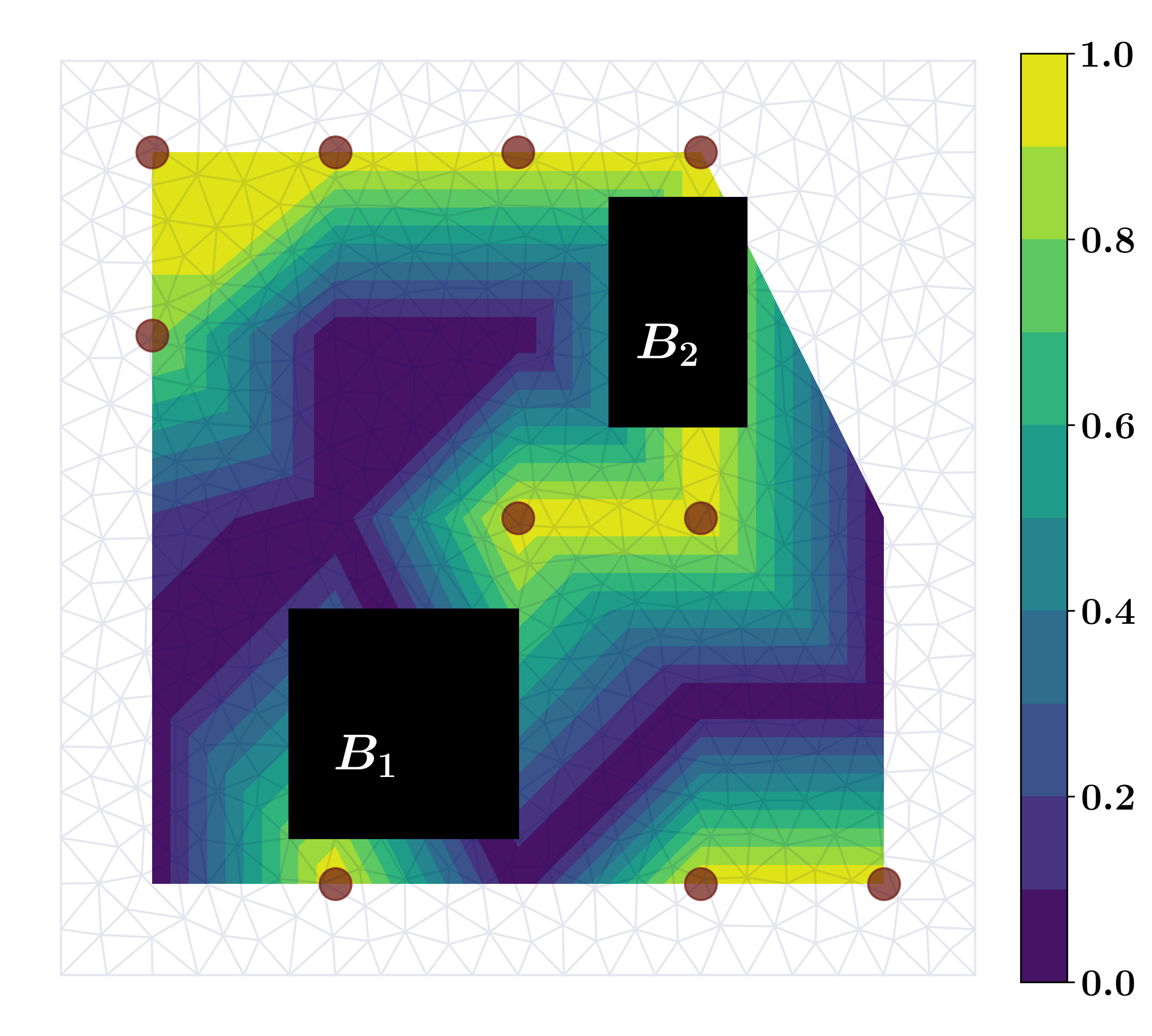}
        \hfill
        \includegraphics[width=0.25\textwidth]{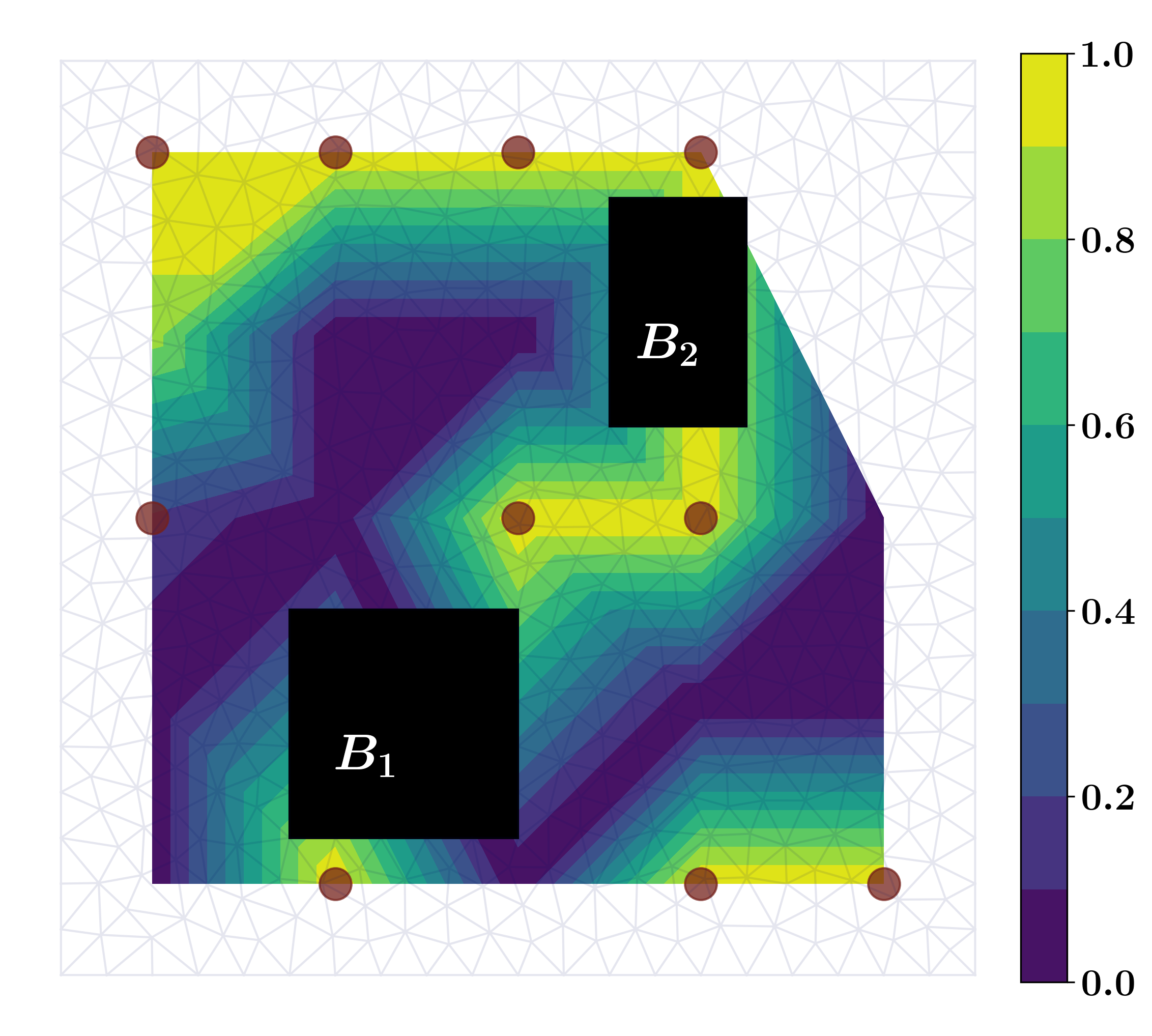}
        \caption{
          A-optimality minimization with $\Nsens=20$ and $\designsumval=10$,
          compared against brute-force enumeration.
          Left: objective value of all feasible designs (blue bars); a sample
          drawn from the optimal policy is shown as red stars, and the best
          policy sample (lowest $\obj$) is marked by a red square.
          The dashed line indicates the global minimum value.
          The following two plots (middle and right) are the two unique samples 
          generated from the optimal policy with objective values
          almost identical to each other and to the global optimum.
          These two plots are overlaid on the interpolated optimal policy
          parameter $\hyperparam\opt$.
        }
        \label{fig:OED_AD_A_Min_designs}
      \end{figure}

      \Cref{fig:OED_AD_A_Min_iter} shows the optimization progress for
      $\Nsens=20$, $\designsumval=10$, and \Cref{fig:OED_AD_A_Min_designs}
      compares the result against the brute-force enumeration of all feasible designs.
      The algorithm converges to a design that attains in this case the global minimum, 
      confirming that the probabilistic framework handles minimization without modification.
      Note that the optimal design here is different from that defined by 
      the maximization problem which is expected due to both insertion of the prior 
      covariance and inversion of the covariance matrix.
     
      We finally note that the formulation of the OED optimality criterion
      is a user choice, and the proposed probabilistic
      optimization approach enables exploring these choices out of the box.
      The two experiments in this section — maximizing the Fisher information trace
      (\Cref{subsec:AD_Results}) and minimizing the posterior covariance trace
      (\Cref{sec:numerical_A_minimization}) — use structurally different objectives
      (one involves only the forward operator; the other incorporates the prior
      covariance and a matrix inversion) and optimize in opposite directions,
      yet \Cref{alg:probabilistic_binary_optimization} requires no modification
      in either case.
      This confirms that the approach is genuinely criterion-agnostic: it treats
      $\obj$ as a black box and the CB model enforces the budget constraint
      independently of the criterion's functional form.

    \subsection{Comparison with design-space relaxation and sum-up rounding}
    \label{sec:numerical_comparison_relaxation}
      A natural alternative to the probabilistic approach is to
      \emph{relax} the binary constraint $\design \in \{0,1\}^{\Nsens}$
      to the continuous box $[0,1]^{\Nsens}$, optimize the
      relaxed objective, and then recover a feasible binary design
      by \emph{sum-up rounding} (SUR) \cite{yu2021multidimensional}.
      This \emph{design-space relaxation} approach requires no stochastic
      gradient estimation and typically converges in fewer iterations;
      however, it requires the gradient of $\obj$ with respect to $\design$
      — information that is unavailable in a black-box setting — and the
      rounding step may degrade the objective value, especially under
      tight budget constraints.
      The proposed probabilistic approach, by contrast, operates entirely
      in the original binary feasible region, requires only black-box
      evaluations of $\obj$, and imposes the budget constraint exactly
      through the CB/GCB model.
      Here we compare the performance of the proposed approach against design relaxation 
      with SUR for binary design retrieval.
      
      \paragraph{Relaxation + SUR setup.}
      The relaxed objective is minimized over $[0,1]^{\Nsens}$ using the
      L-BFGS-B algorithm \cite{byrd1995limited} (via SciPy \cite{2020SciPy-NMeth})
      with a maximum of $500$ iterations, convergence tolerances
      $f_{\rm tol} = 10^{-15}$ and $g_{\rm tol} = 10^{-8}$, and
      box constraints $\design \in [0,1]^{\Nsens}$.
      The budget constraint $\|\design\|_0 = \designsumval$ is incorporated
      as an $\ell_2$ penalty term $\lambda\,\|\|\design\|_0 - \designsumval\|_2^2$
      appended to the objective.
      The penalty weight $\lambda$ is set adaptively for each problem instance
      by matching the gradient norm of the penalty term to that of the
      unpenalized criterion at the initial (uniform) design point,
      ensuring that neither term overwhelms the other regardless of the
      criterion scale or problem dimension.
      The gradient of $\obj$ with respect to $\design$ is computed analytically
      via the adjoint of the forward model, making the gradient evaluation
      cost comparable to a single forward solve \cite{attia2022optimal}.
      Once the continuous solution $\design^* \in [0,1]^{\Nsens}$ is
      obtained, it is rounded to a binary design $\design^* \in \{0,1\}^{\Nsens}$
      with $\|\design^*\|_0 = \designsumval$ using SUR \cite{yu2021multidimensional}.
      When rounding produces ties among candidate sensors, the final design
      is selected by evaluating the objective at each tied candidate and
      choosing the one achieving the best value.

      \Cref{fig:relaxation_vs_probabilistic} compares the two approaches on the
      A-optimality \emph{minimization} problem (\Cref{sec:numerical_A_minimization})
      for $\Nsens \in \{25, 50, 100, 200, 300, 400, 500, 599, 700\}$ candidate sensors
      with a fixed budget $\designsumval = 10$.
      Each configuration is repeated over $25$ independent runs with different
      random seeds. 

      \begin{figure}[htbp!]
        \centering
        \includegraphics[width=0.95\textwidth]{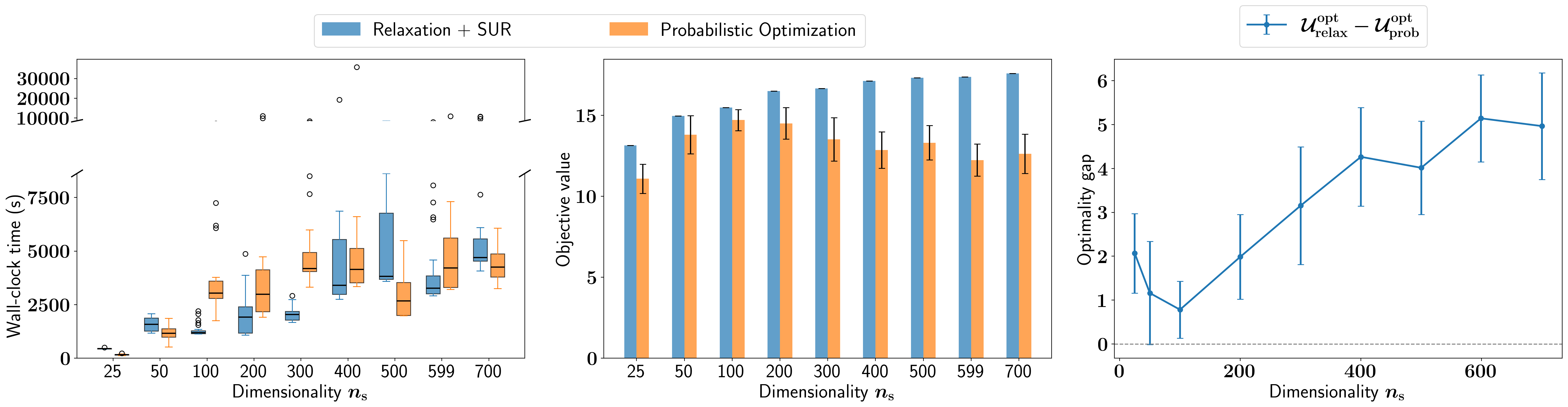}
        \caption{
          Comparison of the probabilistic optimization approach
          (\Cref{alg:probabilistic_binary_optimization}) and the
          design-space relaxation with sum-up rounding~\cite{yu2021multidimensional}
          using the same setup in \Cref{sec:numerical_A_minimization} 
          for increasing design-space
          dimension $\Nsens$ and fixed budget $\designsumval = 10$.
          Results are accumulated over $25$ independent runs per configuration.
          Left: wall-clock time per method.
          Middle: mean $\pm$ one standard deviation of the final objective value
          (lower is better).
          Right: optimality gap $\mathcal{U}^{\rm opt}_{\rm relax} - \mathcal{U}^{\rm opt}_{\rm prob}$;
          positive values indicate that the probabilistic approach achieves a
          lower (better) objective value than relaxation + SUR.
        }
        \label{fig:relaxation_vs_probabilistic}
      \end{figure}

      The left panel of \Cref{fig:relaxation_vs_probabilistic} shows that the two
      methods have comparable wall-clock times across all tested dimensions with the 
      probabilistic approach often experiencing lower times at higher dimensions. 
      The middle panel (lower is better for minimization) shows that the probabilistic
      approach achieves lower mean objective values relative to
      relaxation + SUR, with the advantage becoming more pronounced at higher
      dimensions. 
      This is in contrast to what one might expect given that the relaxation method
      has access to the analytic gradient via the adjoint: the SUR rounding step
      degrades the continuous solution, and this degradation worsens as the
      problem dimension grows and the rounding becomes less accurate.
      The right panel confirms this trend: the gap
      $\mathcal{U}^{\rm opt}_{\rm relax} - \mathcal{U}^{\rm opt}_{\rm prob}$
      is positive and grows with $\Nsens$, indicating that the probabilistic approach
      increasingly outperforms relaxation + SUR at larger problem sizes.
      %

  \section{Discussion and Concluding Remarks}
    \label{sec:conclusions}
    In this work we presented a fully probabilistic approach for binary optimization 
    with a black-box objective function and with budget constraints.
    The approach views the optimization variable as a random variable and employs
    a parametric conditional distribution with support equal to the
    feasible region defined by the constraints.
    The objective function is cast as a stochastic objective over the parameters
    of the conditional distribution, which is then optimized via stochastic gradient ascent.
    The resulting optimal parameter (policy) 
    enables
    sampling the region near the global optimum solution(s).
    This is similar to 
    policy optimization in reinforcement learning where the 
    horizon is finite with only one time step.
    
    The proposed approach does not employ a penalty term to enforce the budget constraint.
    In contrast, the original probabilistic framework \cite{attia2022stochastic} enforces
    budget constraints via soft penalty terms appended to the objective, which requires
    expensive tuning and is best suited to unconstrained binary optimization.
    By using a conditional distribution with support restricted to 
    the feasible region, 
    the proposed approach only explores 
    feasible binary points
    leading to massive computational cost reduction.
    Note that the cost of the proposed approach, in terms of number of function evaluations, 
    is predetermined by the settings of the 
    optimization procedure, namely, by the choice of the size of the sample used to 
    numerically evaluate the stochastic gradient, and by the number of optimization iterations.
    Thus, the computational cost of the proposed approach is dominated by 
    the cost of evaluating the objective 
    at sampled feasible realizations of the binary variable. 
    The proposed approach, however, is data parallel and is thus ideal for large-scale optimization
    problems rather than small problems where a near-optimal solution can be found, for example,
    by greedy methods, brute force, or even random search.
    Moreover, efficient inexpensive approximations of the objective function such as randomization 
    methods and machine learning surrogates of expensive simulations can be employed to speed up the 
    optimization process when applicable. 

    Here
    we focused primarily on budget 
    constraints that are popular in applications 
    such as optimal sensor placement. 
    Modeling hard constraints using conditional probability distributions in general is a nontrivial
    task. General approaches such as combining Fourier transform with the characteristic function of
    basic distributions can help in developing conditional distributions of more complex constraints.

    \noindent\textbf{Optimizing the expectation versus other objectives.}
    A key design choice in the proposed framework, inherited from \cite{attia2022stochastic},
    is to optimize the \emph{expected} utility
    $\stochobj(\hyperparam) = \Expect{\design \sim \CondProb{\design}{\hyperparam}}{\obj(\design)}$
    over the Bernoulli parameter $\hyperparam$.
    This is a natural choice: maximizing the mean encourages the distribution to concentrate
    probability mass on high-utility designs.
    However, it is not the only option.
    One could instead target a risk-aware objective, such as the conditional value-at-risk
    (CVaR) or a variance-penalized utility $\stochobj(\hyperparam) - \lambda\,\mathrm{Var}[\obj(\design)]$,
    which would reward consistently producing good designs rather than occasionally exceptional ones.
    Conversely, a variance-\emph{increasing} term could promote exploration early in optimization
    to escape local optima before gradually shifting to exploitation.
    Since the present work focuses on the hard-constraint extension and foundational guarantees,
    we retain the expectation objective.
    We note that the gradient estimator framework developed here is agnostic to the
    specific functional of the distribution being optimized, so such extensions
    can be accommodated within the same algorithmic structure.

    \noindent\textbf{Non-binary convergence of the Bernoulli parameters.}
    In the numerical experiments carried out in this work, the Bernoulli 
    parameter $\hyperparam$ converges in some cases to values that are not binary 
    (i.e., not in $\{0,1\}^{\Nsens}$).
    This is not a deficiency of the method but rather a natural and informative outcome.
    Specifically, when the feasible set includes multiple global (or near-global) optima, 
    the policy optimization algorithm yields policy parameters that try to capture 
    more than one optimal solution.

    \noindent\textbf{Step-size selection.}
    The only tunable parameter in this work is the optimization step size (learning rate).
    This is generally an open problem for stochastic optimization, with standard guidelines
    including diminishing schedules of the form $\eta^{(n)} = \eta_0 / (1 + \alpha n)$
    (which satisfy the Robbins--Monro conditions \cite{BertsekasTsitsiklis96} and
    guarantee convergence), or adaptive methods such as Adam \cite{kingma2014adam}
    that adjust the step size per component based on gradient history.
    In this work, the domain of the optimization variable $\hyperparam$ is the
    unit hypercube $[0,1]^{\Nsens}$.
    The scaling projection operator \eqref{eqn:scaling_projector} absorbs the effect
    of the step size by ensuring feasibility for any $\eta \in (0,1]$,
    which enabled us to use a fixed learning rate in all experiments
    (\Cref{sec:numerical_experiments}) without loss of convergence.
    A fixed step size is convenient but not always optimal; for problems with
    ill-conditioned gradient landscapes, a diminishing or adaptive schedule
    would be preferable.

    \noindent\textbf{Relation to other penalty-free binary optimization methods.}
    A characteristic feature of the proposed approach is that it is
    \emph{penalty-free}: the budget constraint is encoded directly in the
    constrained probability model rather than appended to the objective as a
    soft penalty term.
    Several other methods share this property, but differ in their assumptions
    and mechanisms.
    \emph{Sum-up rounding} \cite{yu2021multidimensional} produces integer-feasible
    solutions by rounding a continuous relaxation in a way that preserves the
    integral of the control; it therefore operates as a post-processing step
    applied to a relaxed solution and requires gradient access to the objective
    $\obj$ with respect to the design $\design$ in order to solve the relaxed
    problem.
    In contrast, the proposed approach optimizes the binary objective directly
    via stochastic sampling and requires only black-box evaluations of $\obj$.
    The numerical comparison in \Cref{sec:numerical_comparison_relaxation}
    confirms that the proposed probabilistic approach outperforms traditional 
    relaxation, especially in higher dimensional design spaces, 
    while retaining its advantage in the black-box setting.
    \emph{Redundant-dominant designs} \cite{aarset2025global} certify global
    optimality via subgradient-based optimality conditions that exploit dominance
    relations in the objective; this requires problem-specific structure that is
    unavailable in the black-box setting and is therefore complementary to
    the present work rather than directly comparable.
    \emph{Binary neural network training} \cite{meng2020training} also optimizes
    a continuous parameter governing a discrete (binary) distribution, using the
    Bayesian learning rule to justify and improve upon straight-through gradient estimators.
    The mechanics differ: \cite{meng2020training} relaxes the binary weights
    themselves and applies deterministic binarization at inference time, whereas
    the present work maintains a proper probability distribution over binary
    designs and uses unbiased stochastic gradient estimators derived
    from the CB/GCB models.
    The budget constraint of the present work has no direct analogue in the
    binary neural network context.

    \noindent\textbf{Future directions.}
    Several natural extensions of the present framework are left for future work.
    First, while this work focuses on equality and inclusion budget constraints,
    modeling more general hard constraints via conditional probability distributions
    remains an open and practically important problem; general tools such as
    characteristic-function-based constructions offer a promising avenue.
    Second, although a fixed step size proved sufficient in all experiments reported here
    (owing to the scaling projection operator), systematic investigation of
    diminishing or adaptive step-size schedules (e.g., Adam \cite{kingma2014adam} or
    Robbins--Monro-type rules \cite{BertsekasTsitsiklis96}) in this framework
    is warranted for problems with ill-conditioned gradient landscapes.
    Third, while the present work treats the objective as a black-box scalar function,
    replacing it with a stochastic simulator or a machine-learning surrogate
    can further reduce cost in expensive-evaluation settings.
    More broadly, the probabilistic viewpoint developed here---casting a discrete
    feasibility-constrained problem as stochastic optimization over a parametric
    distribution---has since proven to generalize naturally beyond binary designs,
    for example, in robust and  path-based experimental design settings,
    pointing to a broader research program in probabilistic optimal experimental design.

\ack{
    This material is based upon work supported by 
    the U.S. Department of Energy, 
    Office of Science, 
    Office of Advanced Scientific Computing Research (ASCR),
    ASCR Applied Mathematics Base Program,
    and Scientific Discovery through Advanced Computing (SciDAC) Program 
    through the FASTMath Institute under contract 
    number DE-AC02-06CH11357 at Argonne National Laboratory.
}

\data{
  The data that support the findings of this study are openly available through the 
  \pyoed package \cite{attia2024pyoed,attia2023pyoedRepo}. 
  All numerical results are reproducible through 
  the following URL/DOI:
  \url{https://gitlab.com/ahmedattia/pyoed/-/tree/main/pyoed/examples/new_developments/trajectory_oed}
}

\appendix
      
  \section{
    Proofs of Theorems and Lemmas in \Cref{sec:probability_models}
  }\label{app:probability_models_proofs}
    \begin{proof}[Proof of \Cref{lemma:PB_Model_Degenerate}]
      The identity \eqref{eqn:PB_Model_Degenerate_1} evaluates the PMF 
      of the PB model when one of the 
      success probabilities is equal to $0$.
      Because we can reorder the entries of $\design$ and $\hyperparam$, it is sufficient
      to prove \eqref{eqn:PB_Model_Degenerate_1} for $i=1$.
      By applying the law of total probability, 
      \begin{equation}\label{eqn:PB_Model_Degenerate_1_proof_1}
        \begin{aligned}
          \CondProb{\designsum\!=\!\designsumval}{\hyperparam_1\!=\!0}
            &=\CondProb{\designsum\!=\!\designsumval, \design_1\!=\!0}{\hyperparam_1\!=\!0}
              + \CondProb{\designsum\!=\!\designsumval, \design_1\!=\!1}{\hyperparam_1\!=\!0} \\
            &= \CondProb{\designsum\!=\!\designsumval, \design_1\!=\!0}{\hyperparam_1\!=\!0} + 0 
            = \CondProb{(0\!+\!\design_2 \!+\! \ldots
              \!+\! \design_{\Nsens})\!=\!\designsumval}{\hyperparam_1\!=\!0}   \\
             &= \CondProb{(\design_2\!+\!\ldots\!+\!\design_{\Nsens})\!=\!\designsumval}{
               \hyperparam_2, \ldots, \hyperparam_{\Nsens}
             } 
             \stackrel{\eqref{eqn:Poisson_Binomial_PMF}}{=}
             \frac{
                R(\designsumval, S\setminus\{1\}) 
              }{
                \prod_{\substack{j=2}}^{\Nsens} (1+w_j)
              }
             \,,
        \end{aligned}
      \end{equation}
      which is the PMF of a PB model
      resulting from removing $\hyperparam_1$ from \eqref{eqn:Poisson_Binomial_PMF}. 
      Similarly,  
      \begin{equation}\label{eqn:PB_Model_Degenerate_1_proof_2}
        \begin{aligned}
        \CondProb{\designsum\!=\!\designsumval}{\hyperparam_1\!=\!1}
          &= \CondProb{(1\!+\!\design_2\!+\!\ldots\!+\!\design_{\Nsens})\!=\!\designsumval}{\hyperparam_1\!=\!1}  
          \\
          &= \CondProb{(\design_2\!+\!\ldots\!+\!\design_{\Nsens})\!=\!\designsumval-1}{
               \hyperparam_2, \ldots, \hyperparam_{\Nsens}
          } 
          \stackrel{\eqref{eqn:Poisson_Binomial_PMF}}{=}
            \frac{
                R(\designsumval-1, S\setminus\{1\}) 
              }{
                \prod_{\substack{j=2}}^{\Nsens} (1+w_j)
              }
          \,, 
        \end{aligned}
      \end{equation}
      which proves \eqref{eqn:PB_Model_Degenerate_2}.
      To prove \eqref{eqn:PB_Model_Degenerate_3} and \eqref{eqn:PB_Model_Degenerate_4}, 
      we note that 
      \begin{equation}\label{eqn:PB_Model_Degenerate_3_proof}
        \begin{aligned}
          \Prob{(\designsum\!=\!\designsumval)}
            &\stackrel{(\ref{eqn:Poisson_Binomial_PMF}, \ref{eqn:R_function_and_Bernoulli_weights_recurrence_relation_2_formula})}{=} 
               (1\!-\!\hyperparam_i)
               \Bigl(
                 R(\designsumval, S\!\setminus\!\{i\})  
                 + w_i R(\designsumval\!-\!1, S\!\setminus\!\{i\})  
               \Bigr)
              \prod_{\substack{j=1,\,j\neq i}}^{\Nsens}\! (1\!-\!\hyperparam_j) \\ 
            &\stackrel{(\ref{eqn:R_function_and_Bernoulli_weights})}{=}  
               \Bigl(
                 (1\!-\!\hyperparam_i) R(\designsumval, S\!\setminus\!\{i\})  
                 + \hyperparam_i R(\designsumval\!-\!1, S\!\setminus\!\{i\})  
               \Bigr)
              \prod_{\substack{j=1,\,j\neq i}}^{\Nsens}\! (1\!-\!\hyperparam_j) 
              \,. 
        \end{aligned}
      \end{equation}
      Thus,  
      \eqref{eqn:PB_Model_Degenerate_3} and \eqref{eqn:PB_Model_Degenerate_4}
      are obtained by setting $\hyperparam_i$ in the right-hand side
      to $0$ and $1$, respectively.
    \end{proof}
    \begin{proof}[Proof of \Cref{theorem:PB_Model_Full}]
      The proof follows immediately from 
      \eqref{eqn:Poisson_Binomial_PMF} and
      recursive application of \eqref{eqn:PB_Model_Degenerate}
      over all entries of $\design$ corresponding to degenerate probabilities 
      $\hyperparam_i\in\{0, 1\},i=1, \ldots, \Nsens$.
    \end{proof}

    \begin{proof}[Proof of \Cref{lemma:CB_Model_Degenerate}]
      First, we note that \eqref{eqn:CB_Model_PMF} 
      takes the following equivalent form,
      \begin{equation}\label{eqn:CB_Model_Degenerate_conditional}
        \begin{aligned}
        \CondProb{
            \design 
        }{
          \designsum=\designsumval
        } 
          &=
            \CondProb{
              \design_{-i} 
            }{
              \design_i;\, 
              \sum_{j=1}^{\Nsens} \design_{j} =\designsumval
            }
            \CondProb{\design_{i}}{\designsum=\designsumval}
          \\
          &=
            \CondProb{
              \design_{-i} 
            }{
              \sum_{\substack{ j=1 \\ j\neq i} }^{\Nsens} \design_{j}
                =\designsumval\!-\!\design_{i}
            } 
            \CondProb{\design_{i}}{\designsum=\designsumval}
            \,, 
        \end{aligned}
      \end{equation}
      where 
      $
        \design_{-i}\equiv\design_1, \ldots, \design_{i-1}, \design_{i+1}, \ldots, \design_{\Nsens}
      $ is obtained by removing the $i$th entry 
      from $\design$.
      Thus, \eqref{eqn:CB_Model_Degenerate_conditional} is the product of  the PMF of 
      a CB model obtained by 
      discarding the $i$th component of $\hyperparam$ with the first-order 
      inclusion probability 
      $\CondProb{\design_i}{\designsum=\designsumval}=\pi_i$. 
      Assuming $\hyperparam_i\in\{0, 1\}$,  
      the value of the PMF \eqref{eqn:CB_Model_PMF} is equal to $0$ when 
      $\design_i\neq \hyperparam_i$. 
      If we let $\hyperparam_i=0$, then 
      the first term in \eqref{eqn:CB_Model_Degenerate_conditional} is given by 
      \begin{equation}\label{eqn:CB_Model_Degenerate_1_proof_1}
          \CondProb{
            \design_{-i}
          }{
            \sum_{\substack{ j=1 \\  j\neq i} }^{\Nsens} \design_{j}
            =\designsumval\!-\!\design_{i}, \, 
            \hyperparam_i=0 
          } 
          =
          \CondProb{
            \design_{-i}
          }{
            \sum_{\substack{ j=1 \\  j\neq i} }^{\Nsens} \design_{j}=\designsumval,\, 
            \hyperparam_{-i}
          }
          =
            \frac{
              \prod\limits_{\substack{j=1 \\ j\neq i}}^{\Nsens} w_j^{\design_j} 
            }{
              R(\designsumval, S\!\setminus\!\{i\}) 
            } 
          \,,
      \end{equation}
      where $\hyperparam_{-i}$ is obtained by removing $i$th entry from $\hyperparam$. 
      This proves \eqref{eqn:CB_Model_Degenerate_1}.
      Similarly, for $\hyperparam_i=1$, 
      \begin{equation}\label{eqn:CB_Model_Degenerate_1_proof_2}
          \CondProb{
            \design_{-i}
          }{
            \sum\limits_{\substack{ j=1 \\  j\neq i} }^{\Nsens} \design_{j}
              =\designsumval\!-\!\design_{i}, 
              \hyperparam_i=1 
          } 
          =
          \CondProb{
            \design_{-i}
          }{
            \sum\limits_{\substack{ j=1 ,\,  j\neq i} }^{\Nsens} \design_{j}\!=\!\designsumval\!-\!1; 
            \hyperparam_{-i}
          }
          =
            \frac{
              \prod\limits_{\substack{j=1,\, j\neq i}}^{\Nsens} w_j^{\design_j} 
            }{
              R(\designsumval\!-\!1, S\!\setminus\!\{i\}) 
            } 
          \,,
      \end{equation}
      which 
      proves \eqref{eqn:CB_Model_Degenerate_2}.
      To prove \eqref{eqn:CB_Model_Degenerate_3} and \eqref{eqn:CB_Model_Degenerate_4}, 
      we note that 
      \eqref{eqn:CB_Model_PMF} 
      takes the following equivalent form 
      \begin{equation}
        \begin{aligned}
          \CondProb{\design}{\designsum\!=\!\designsumval}
          &
          \stackrel{(\ref{eqn:CB_Model_PMF})}{=} \frac{
              \prod\limits_{i=1}^{\Nsens}{{w_i}^{\design_i}}
            }{
              R(\designsumval, S)
            } 
          \\ 
          &
          \stackrel{(
            \ref{eqn:R_function_and_Bernoulli_weights_recurrence_relation_2_formula},
            \ref{eqn:R_function_and_Bernoulli_weights}
          )}{=} 
            \frac{
              {\hyperparam_{i}}^{\design_i}
                \prod\limits_{\substack{i=j,\, j\neq i}}^{\Nsens}{{w_j}^{\design_j}} 
            }{
              {\left({1-\hyperparam_{i}}\right)}^{\design_i}
              R(\designsumval, S\setminus\{i\}) 
              + \hyperparam_{i} {\left({1-\hyperparam_{i}}\right)}^{\design_i-1}
                 R(\designsumval-1, S\setminus\{i\}) 
            } 
          \,.
        \end{aligned}
      \end{equation}

      By setting the value of 
      $\design_i$ to $0/1$ and taking the limit 
      as $\hyperparam_i\to 0/1$, respectively we prove \eqref{eqn:CB_Model_Degenerate_3} 
      and \eqref{eqn:CB_Model_Degenerate_4} as follows.
      \begin{align}
        \left.
          \lim_{ \hyperparam_i\rlim 0} \CondProb{\design}{\designsum\!=\!\designsumval}
        \right|_{\design_i=0} 
        &=
        \lim_{ \hyperparam_i\rlim 0} 
        \frac{
          {\prod\limits_{\substack{j=1 ,\, j\neq i}}^{\Nsens}{{w_j}^{\design_j}} }
        }{
          R(\designsumval, S\!\setminus\!\{i\}) 
            + \hyperparam_{i} {\left({1\!-\!\hyperparam_{i}}\right)}^{-1}
            R(\designsumval\!-\!1, S\!\setminus\!\{i\}) 
        } 
        = 
        \frac{
          {\prod\limits_{\substack{j=1 ,\, j\neq i}}^{\Nsens}{{w_j}^{\design_j}} }
        }{
            R(\designsumval, S\!\setminus\!\{i\}) 
          } \,, \\
        \left. 
          \lim_{ \hyperparam_i\llim 1} \CondProb{\design}{\designsum\!=\!\designsumval} 
        \right|_{\design_i=0} 
          &=
          \lim_{ \hyperparam_i\llim 1} 
          \frac{
            {\prod\limits_{\substack{j=1 ,\, j\neq i}}^{\Nsens}{{w_j}^{\design_j}} }
          }{
            R(\designsumval, S\setminus\{i\}) 
              + \hyperparam_{i} {\left({1-\hyperparam_{i}}\right)}^{-1}
              R(\designsumval-1, S\setminus\{i\}) 
          }  
          = 0 \,,  \\
        \left.
          \lim_{ \hyperparam_i\rlim 0} \CondProb{\design}{\designsum\!=\!\designsumval} 
        \right|_{\design_i=1} 
          &= 
          \lim_{ \hyperparam_i\rlim 0} 
            \frac{
              \hyperparam_i
              {\prod\limits_{\substack{j=1 ,\, j\neq i}}^{\Nsens}{{w_j}^{\design_j}} }
            }{
              (1\!-\!\hyperparam_i)
              R(\designsumval, S\!\setminus\!\{i\}) 
                \!+\! \hyperparam_{i} R(\designsumval\!-\!1, S\!\setminus\!\{i\}) 
          }
          = 0  \,, \\
        \left.
          \lim_{ \hyperparam_i\llim 1} \CondProb{\design}{\designsum\!=\!\designsumval} 
        \right|_{\design_i=1} 
          &= 
          \lim_{ \hyperparam_i\llim 1} 
            \frac{
              \hyperparam_i
              {\prod\limits_{\substack{j=1 ,\, j\neq i}}^{\Nsens}{{w_j}^{\design_j}} }
            }{
              (1\!-\!\hyperparam_i)
              R(\designsumval, S\!\setminus\!\{i\}) 
                \!+\! \hyperparam_{i} R(\designsumval\!-\!1, S\!\setminus\!\{i\}) 
          }
          = 
            \frac{
              \prod\limits_{\substack{j=1 ,\, j\neq i}}^{\Nsens}{{w_j}^{\design_j}} 
            }{
              R(\designsumval\!-\!1, S\!\setminus\!\{i\}) 
          }
          \,. 
      \end{align}
    \end{proof}
    \begin{proof}[Proof of \Cref{theorem:CB_Model}]
      The proof 
      follows from the definition of the CB model's
      PMF for non-degenerate probabilities
      \eqref{eqn:CB_Model_PMF}
      and 
      by recursive application of
      \Cref{lemma:CB_Model_Degenerate}
      to entries of $\design_i$ with degenerate probabilities 
      $\hyperparam_i\in\{0, 1\}, i=1, \ldots, \Nsens$.
    \end{proof}
    \begin{proof}[Proof of \Cref{proposition:CB_Model_moments}]
      While in \cite{chen2000general} it was noted that 
      $\Expect{}{\design_{i}} = \pi_i$, 
      the proof was not accessible.
      Thus, we provide the proof here in detail for completeness.
      \begin{equation}
        \begin{aligned}
          \Expect{}{\design_i}
            &\stackrel{(\ref{eqn:CB_Model_PMF})}{=} 
            \frac{ 
              \sum_{\substack{\design \in \{0, 1\}^{\Nsens} \\ \wnorm{\design}{0}=\designsumval}}
                \design_i \, 
                  \prod\limits_{j=1}^{\Nsens}{{w_j}^{\design_j}}
            }{ R(\designsumval, S) } 
            = \frac{ 
                0 w_{i}^{0} \, 
                \sum\limits_{\substack{ B\subseteq S\setminus \{i\} \\ |B|=\designsumval} } 
                  \prod\limits_{j \in B} w_j 
                +
                1 w_{i}^{1} \, 
                \sum\limits_{\substack{ B\subseteq S\setminus \{i\} \\ |B|=\designsumval-1 } } 
                  \prod\limits_{j \in B} w_j 
              }{ R(\designsumval, S) } 
            \\
            &= \frac{ 
              w_{i} \, 
              \sum\limits_{\substack{ B\subseteq S\setminus \{i\} \\ |B|=\designsumval-1}} 
                \prod\limits_{j \in B} w_j 
            }{ R(\designsumval, S) } 
            \stackrel{(\ref{eqn:R_function_and_Bernoulli_weights})}{=} 
            \frac{ w_{i} \, R(\designsumval-1, S\setminus \{i\}) }{ R(\designsumval, S) } 
            = \pi_{i} 
            \,.
        \end{aligned}
      \end{equation}
      
      \begin{equation}
        \begin{aligned}
        \Expect{}{\design_i \design_j}
          &\stackrel{(\ref{eqn:CB_Model_PMF})}{=} 
            \frac{
              \sum_{\substack{\design \in \{0, 1\}^{\Nsens} \\ \wnorm{\design}{0}=\designsumval}}
                \design_i \design_j \,  
              \prod\limits_{j=1}^{\Nsens}{{w_j}^{\design_j}}
            }{
              R(\designsumval, S)
            }  
          = \frac{ 
              w_{i} w_{j} 
              \sum_{\substack{ \substack{B\subseteq S\setminus \{i, j\} \\ |B|=\designsumval-2} } } 
                \prod_{k \in B} w_k 
            }{ R(\designsumval, S) } \\
            %
          &\stackrel{
            (\ref{eqn:R_function_and_Bernoulli_weights}, 
             \ref{eqn:first_second_inclusion}
            )}{=}
          \pi_{i, j}
          \,;\, \, i\neq j\,,
        \end{aligned}
      \end{equation}
      where the second step is obtained by dropping all terms with 
      $\design_i=0$ or $\design_j=0$ 
      from the expectation expansion. 
      In the other case, namely $i=j$, since $\design_i\in\{0, 1\}$,  
      $
      \design_i \design_j = \design_i^2 = \design_i,
      $ we have
      $
        \Expect{}{\design_i \design_j} 
        = \Expect{}{\design_i }
        = \pi_i 
      $ and  thus for  $i, j\in\{1, \ldots, \Nsens\}$ we have 
      $
        \Expect{}{\design_i \design_j}
          = \delta_{ij} \pi_{i} + (1-\delta_{ij}) \pi_{ij} 
      $. 
      The covariance is given by 
      $
        \brCov{\design_{i}, \design_{j}}
          = \Expect{}{\design_i \design_j}
            - \Expect{}{\design_i }
              \Expect{}{\design_j}
          = \delta_{ij} (\pi_{i} - \pi_{i}^2)  
            + (1-\delta_{ij}) (\pi_{i, j} - \pi_{i}\pi_{j})   
      $.
    \end{proof}
    \begin{proof}[Proof of \Cref{theorem:CB-Multiple-Sums}]
        By using Bayes' rule, 
        \begin{equation}\label{eqn:Bayes_designsum}
          \CondProb{\design}{\designsum\!\in\! \designsumset}
            = \CondProb{\design}{\designsum\in 
            \{\designsumval_1, \ldots, \designsumval_{m}\}} 
            =  
            \frac{
            \Prob{(\design,\, \designsum\in \designsumset )}
            }{
            \Prob{(\designsum\in \designsumset) }
            } 
            =  
            \frac{
              \CondProb{\designsum\in \designsumset }{\design} 
                \, \Prob{(\design)}
            }{
              \Prob{(\designsum\in \designsumset)}
            } \,, 
        \end{equation}
        and 
        because realizations of the sum $\designsum$ are mutually exclusive, 
        that is, $\designsum$ cannot take different values at the same time, 
        then  $\forall i\neq j;\, i, j\in S$, 
        $
          \CondProb{\designsum=\designsumval_{i} 
            \cap \designsum=\designsumval_{j}}{\hyperparam;\, \design}=0 
        $,
        hence, 
        $
          \CondProb{\designsum=\designsumval_{i} \cup \designsum=\designsumval_{j}}{\design}
          = \CondProb{\designsum=\designsumval_1}{\design}
            + \CondProb{\designsum=\designsumval_2}{\design}
        $ for any two realizations $\designsumval_1,\, \designsumval_2$ of $\designsum$. 
        Thus,

        \begin{equation}
          \begin{aligned}
            \CondProb{\design}{\designsum\in \designsumset}
              &=  
              \frac{
                \CondProb{\designsum\in \designsumset }{\design} 
                  \, \Prob{(\design)}
              }{
                \Prob{(\designsum\in \designsumset)}
              } 
              =  
              \frac{
                \CondProb{\designsum=\designsumval_1 \cup \ldots 
                  \cup \designsum=\designsumval_m }{\design} 
                  \, \Prob{(\design)}
              }{
                \Prob{(\designsum=\designsumval_1)} + \ldots 
                  + \Prob{(\designsum=\designsumval_m)}
              }  \\ 
              &=  
              \frac{
                \Bigl(
                  \CondProb{\designsum=\designsumval_1}{\design} + \ldots + 
                  \CondProb{\designsum=\designsumval_m }{\design} 
                \Bigr)
                  \, \Prob{(\design)}
              }{
                \Prob{(\designsum=\designsumval_1)} + \ldots 
                  + \Prob{(\designsum=\designsumval_m)}
              }
              \\
              &= \frac{
              \sum\limits_{\designsumval\in \designsumset} 
                \CondProb{\design}{\designsum=\designsumval} 
                  \Prob{(\designsum=\designsumval)}
              }{
              \sum\limits_{\designsumval\in \designsumset} 
                \Prob{(\designsum=\designsumval)}
              }
              \,.
          \end{aligned}
        \end{equation}
        
        \begin{equation}
          \begin{aligned}
            \CondExpect{f(\design)}{\designsum\in\designsumset} 
              &= \sum_{\substack{ \design\in\{0, 1\}^{\Nsens} \\ \wnorm{\design}{0}\in\designsumset }}
                f(\design) \CondProb{\design}{\designsum\in\designsumset} \\ 
              &\stackrel{\ref{eqn:GCB_Model_PMF}}{=} \sum_{\substack{ \design\in\{0, 1\}^{\Nsens} \\ \wnorm{\design}{0}\in\designsumset }}
                f(\design) 
                \frac{
                  \sum\limits_{\designsumval\in \designsumset} 
                    \CondProb{\design}{\designsum=\designsumval} 
                      \Prob{(\designsum=\designsumval)}
                }{
                  \sum\limits_{\designsumval\in \designsumset} 
                    \Prob{(\designsum=\designsumval)}
                }
                \\
              &= 
              \frac{1}{
                \sum\limits_{\designsumval\in \designsumset} 
                  \Prob{(\designsum=\designsumval)}
              }
              \sum_{\substack{ \design\in\{0, 1\}^{\Nsens} \\ \wnorm{\design}{0}\in\designsumset }}
                f(\design) 
                \sum\limits_{\designsumval\in \designsumset} 
                  \CondProb{\design}{\designsum=\designsumval} 
                    \Prob{(\designsum=\designsumval)}
              \\
              &= 
              \frac{
                \sum\limits_{\designsumval\in \designsumset} 
                \sum\limits_{\substack{ \design\in\{0, 1\}^{\Nsens} \\ \wnorm{\design}{0}\in\designsumset }}
                  f(\design) 
                    \CondProb{\design}{\designsum=\designsumval} 
                      \Prob{(\designsum=\designsumval)}
              }{
                \sum\limits_{\designsumval\in \designsumset} 
                  \Prob{(\designsum=\designsumval)}
              }  \\
              &= 
              \frac{
                \sum\limits_{\designsumval\in \designsumset} 
                  \CondExpect{f(\design)}{\designsum=\designsumval}
                  \Prob{(\designsum=\designsumval)}
              }{
                \sum\limits_{\designsumval\in \designsumset} 
                  \Prob{(\designsum=\designsumval)}
              }
              \,.
          \end{aligned}
        \end{equation}
        
        \begin{equation}
          \begin{aligned}
            &\CondVar{f(\design)}{\designsum\in\designsumset} 
            = \CondExpect{f(\design)^2}{\designsum\in\designsumset} 
              - 
              \left(
                \CondExpect{f(\design)}{\designsum\in\designsumset}
              \right)^2 
            \\
            &\quad\stackrel{\eqref{eqn:GCB_Model-Function-Expect}}{=} 
              \frac{
                \sum\limits_{\designsumval\in \designsumset} 
                  \CondExpect{f(\design)}{\designsum=\designsumval}
                  \Prob{(\designsum=\designsumval)}
              }{
                \sum\limits_{\designsumval\in \designsumset} 
                  \Prob{(\designsum=\designsumval)}
              }
              -  
              \frac{
                \left(
                  \sum\limits_{\designsumval\in \designsumset} 
                    \CondExpect{f(\design)}{\designsum=\designsumval}
                    \Prob{(\designsum=\designsumval)}
                \right)^2 
                }{
                  \left(
                    \sum\limits_{\designsumval\in \designsumset} 
                      \Prob{(\designsum=\designsumval)}
                \right)^2 
                }\,.
          \end{aligned}
        \end{equation}
    \end{proof}
    \begin{proof}[Proof of \Cref{proposition:probability_models_derivatives}]
      The derivative of the R-function \eqref{eqn:R_function_and_Bernoulli_weights} 
      is given by
      \begin{align}
        \del{R(k, A)}{w_i} 
        &\stackrel{(\ref{eqn:R_function_and_Bernoulli_weights})}{=}
          \del{
            \sum\limits_{\substack{B\subseteq A \\ |B|=k}}\! \prod\limits_{j \in B} w_j
          }{w_i}  
          =\! \sum\limits_{\substack{B\subseteq A \\ |B|=k}}\! \del{ \prod\limits_{j \in B} w_j }{w_i} 
          =\! \sum\limits_{\substack{B\subseteq A\!\setminus\! \{i\} \\ |B|=k}} \prod\limits_{j \in B} w_j 
          = R(k\!-\!1, A\!\setminus\! \{i\}) 
          \,, \label{eqn:R_function_derivative_w}
      \\
        \del{R(k, A)}{\hyperparam_i} 
        &= \del{R(k, A)}{w_i} \del{w_i}{\hyperparam_i} 
        \stackrel{(\ref{eqn:R_function_derivative_w}, \ref{eqn:R_function_and_Bernoulli_weights})}{=}
          \frac{R(k\!-\!1, A\!\setminus\! \{i\})} 
            {\left( 1 - \hyperparam_i \right)^2} 
        \equiv R(k\!-\!1, A\!\setminus\! \{i\}) 
          \left( 1 + w_i \right)^2 
          \,. \label{eqn:R_function_derivative_param}
      \end{align}
      The derivative \eqref{eqn:PB_Model_Full_Gradient} of the PMF of the PB model 
      \eqref{eqn:Poisson_Binomial_PMF}
      is then given by
      \begin{equation}\label{eqn:PB_PMF_derivative}
        \begin{aligned}
          \del{\Prob{(\designsum\!=\!\designsumval)}}{\hyperparam_i}
          &\stackrel{
            (\ref{eqn:Poisson_Binomial_PMF}, 
            \ref{eqn:R_function_derivative_param}
            )}{=} 
          \frac{
            R(\designsumval\!-\!1, S\!\setminus\! \{i\})
          }{
            \left( 1 \!-\! \hyperparam_i \right)^2
            \prod\limits_{j=1}^{\Nsens} (1\!+\!w_j)
          }   
            +  
            \frac{R(\designsumval, S)
            }{
              \prod\limits_{\substack{j=1\\j\neq i}}^{\Nsens} 
                (1\!+\!w_j)
            }
          \\
          &\stackrel{(\ref{eqn:R_function_and_Bernoulli_weights_recurrence_relation_2_formula})}{=}
          \frac{
              R(\designsumval\!-\!1, S\!\setminus\! \{i\})
                -  R(\designsumval, S\!\setminus\! \{i\})
            }{
              \prod\limits_{\substack{j=1\\j\neq i}}^{\Nsens} 
                (1\!+\!w_j)
            }
             \,,
        \end{aligned}
      \end{equation}
      where 
      $1+w_i=\frac{1}{1-\hyperparam_i}$. 
      From \eqref{eqn:PB_PMF_derivative} the derivative  
      $\del{\Prob{(\designsum=\designsumval)}}{\hyperparam_i}$ is 
      independent from the value of $\hyperparam_i$.
      It also shows that the 
      derivative with respect 
      to the $i$th entry of $\hyperparam$ is obtained by calculating
      the R-function value at both $\designsumval$ 
      and $\designsumval-1$ for the set of Bernoulli trials excluding 
      $\design_i$.
      This can be generalized easily 
      to the degenerate case 
      \eqref{eqn:PB_Model_Full} by noting that 
      in the case $\hyperparam_i=1$, the set $I$ already includes 
      the $i$th index 
      which yields the derivative \eqref{eqn:PB_Model_Full_Gradient}.
      One can also show that the derivatives
      \eqref{eqn:PB_Model_Full_Gradient} are continuous 
      over the interval $[0, 1]^{\Nsens}$ which proves smoothness 
      of the PMF of the PB model \eqref{eqn:PB_Model_Full}.

      \noindent{}The derivative 
      \eqref{eqn:CB_Model_Full_Gradient} 
      over $\hyperparam\!\in\!(0, 1)^{\Nsens}$
      is 
      $
      \del{\CondProb{\design}{\designsum\!=\!\designsumval}}{\hyperparam_i}
      = \CondProb{\design}{\designsum\!=\!\designsumval}
        \del{\log\CondProb{\design}{\designsum\!=\!\designsumval}}{\hyperparam_i}
      $, with

      \begin{equation}\label{eqn:CB_PMF_derivative_interior}
        \begin{aligned}
          \del{\log\CondProb{\design}{\designsum\!=\!\designsumval}}{\hyperparam_i}
          &\stackrel{(\ref{eqn:CB_Model_PMF})}{=}
              \left(
                \sum\limits_{j=1}^{\Nsens}{ \design_i \del{\log{w_j}}{\hyperparam_j} }
                -
                \del{\log{R(\designsumval, S)}}{\hyperparam_i} 
            \right)
          \\
          &\stackrel{(\ref{eqn:CB_Model_PMF}, \ref{eqn:R_function_derivative_param}, \ref{eqn:first_second_inclusion})}{=}
            \frac{
              (1\!+\!w_i)^2
            }{
              w_i
            }
            \left(
                \design_i
                -
                \pi_i
            \right)
          \,,
        \end{aligned}
      \end{equation}
      which is a compact version of \eqref{eqn:CB_Model_Full_Gradient} 
      over $\hyperparam\!\in\!(0, 1)^{\Nsens}$.
      The right-hand derivative is 
      \begin{equation}
        \left.\del{
          \CondProb{\design}{\designsum=\designsumval} 
        }{
          \hyperparam_i
        }\right|_{\hyperparam_i=0}
        = \lim_{ \epsilon \to 0 } 
        \frac{
          \CondProb{\design}{\designsum=\designsumval; \hyperparam_i=\epsilon} 
          - \CondProb{\design}{\designsum=\designsumval; \hyperparam_i=0} 
        }{\epsilon} 
        \,,
      \end{equation}
      which, by letting $\design_i=0$, can be equivalently written (by replacing $\epsilon$ 
      with $\hyperparam_i$) as 
      \begin{equation}
        \begin{aligned}
          &\left.\del{
            \CondProb{\design}{\designsum=\designsumval} 
          }{
            \hyperparam_i
          }\right|_{\hyperparam_i=0} 
          \\
          &\,\,
          \stackrel{(\ref{eqn:CB_Model_PMF}, \ref{eqn:CB_Model_Degenerate_1})}{=} 
          \lim_{ \hyperparam_i \rlim 0 } 
          \frac{
            \frac{
              \prod\limits_{j=1}^{\Nsens}{{w_j}^{\design_j}}
            }{
              R(\designsumval, S)
            }
            - 
            \frac{
              \prod\limits_{\substack{j=1,\, j\neq i}}^{\Nsens} w_j^{\design_j} 
            }{
              R(\designsumval, S\!\setminus\!\{i\}) 
            } 
          }{
            \hyperparam_i
          }  
          =
            \left(\prod\limits_{\substack{j=1,\, j\neq i}}^{\Nsens} w_j^{\design_j} \right)
            \lim_{ \hyperparam_i \rlim 0 } 
            \frac{
              \frac{
                  R(\designsumval, S\!\setminus\!\{i\}) 
                  -
                R(\designsumval, S) 
              }{
                R(\designsumval, S) 
                R(\designsumval, S\!\setminus\!\{i\}) 
              }
            }{
              \hyperparam_i
            }  \\ 
          &\,\, \stackrel{(\ref{eqn:R_function_and_Bernoulli_weights_recurrence_relation_2_formula})}{=} 
            \left(\prod\limits_{\substack{j=1,\,j\neq i}}^{\Nsens} w_j^{\design_j} \right)
            \lim_{ \hyperparam_i \rlim 0 } 
            \frac{
              \frac{
                - \hyperparam_i  
                  R(\designsumval\!-\!1, S\!\setminus\!\{i\}) 
              }{
                R(\designsumval, S) 
                R(\designsumval, S\!\setminus\!\{i\}) 
              }
            }{
              \hyperparam_i
            } 
          =
            \frac{
              R(\designsumval\!-\!1, S\!\setminus\!\{i\}) 
              \prod\limits_{\substack{j=1,\,j\neq i}}^{\Nsens} w_j^{\design_j}
            }{
              R(\designsumval, S\!\setminus\!\{i\}) 
            }
            \lim_{ \hyperparam_i \rlim 0 } 
              \frac{
                - 1 
              }{
                R(\designsumval, S) 
              } \\ 
          & \,\, \stackrel{(\ref{eqn:R_function_and_Bernoulli_weights_recurrence_relation_2_formula})}{=} 
            \frac{
              R(\designsumval\!-\!1, S\!\setminus\!\{i\}) 
              \prod\limits_{\substack{j=1,\,j\neq i}}^{\Nsens} w_j^{\design_j} 
            }{
              R(\designsumval, S\!\setminus\!\{i\}) 
            }
            \lim_{ \hyperparam_i \rlim 0 } 
              \frac{
                - 1 
              }{
                R(\designsumval, S\!\setminus\!\{i\}) + 
                  \frac{\hyperparam_i}{1\!-\!\hyperparam_i} 
                R(\designsumval\!-\!1, S\!\setminus\!\{i\}) 
              } 
            \\
          &\quad
            = \frac{
              - R(\designsumval\!-\!1, S\!\setminus\!\{i\}) 
            }{
              \left(R(\designsumval, S\!\setminus\!\{i\}) \right)^2
            }
            \prod\limits_{\substack{j=1\\j\neq i}}^{\Nsens} w_j^{\design_j} 
          \,.
        \end{aligned}
      \end{equation}
      \begin{equation}
        \begin{aligned}
          &\left.\del{
            \CondProb{\design}{\designsum=\designsumval} 
          }{
            \hyperparam_i
          }\right|_{\hyperparam_i=0} 
          \stackrel{(\ref{eqn:CB_Model_PMF}, \ref{eqn:CB_Model_Degenerate_1})}{=} 
          \lim_{ \hyperparam_i \rlim 0 } 
          \frac{
            \frac{
              \prod\limits_{j=1}^{\Nsens}{{w_j}^{\design_j}}
            }{
              R(\designsumval, S)
            }
            - 
           0 
          }{
            \hyperparam_i
          }  
          =
          \prod_{\substack{j=1,\, j\neq i}}^{\Nsens} w_j^{\design_j} 
          \lim_{ \hyperparam_i \rlim 0 } 
          \frac{
            \frac{
              \hyperparam_i
            }{
              (1\!-\!\hyperparam_i) R(\designsumval, S)
            }
          }{
            \hyperparam_i
          } 
          \\
          &\quad \stackrel{(\ref{eqn:R_function_and_Bernoulli_weights_recurrence_relation_2_formula})}{=}     
            \prod_{\substack{j=1,\, j\neq i}}^{\Nsens} w_j^{\design_j} 
            \lim_{ \hyperparam_i \rlim 0 } 
            \frac{
              1
            }{
              (1\!-\!\hyperparam_i) R(\designsumval, S\!\setminus\!\{i\})
              + \hyperparam_i R(\designsumval\!-\!1, S\!\setminus\!\{i\})
            } 
            =
            \frac{
            \prod_{\substack{j=1,\, j\neq i}}^{\Nsens} w_j^{\design_j} 
            }{
              R(\designsumval, S\!\setminus\!\{i\})
            }
          \,.
        \end{aligned}
      \end{equation}
      The derivative 
      $
        \left.\del{
          \CondProb{\design}{\designsum=\designsumval} 
        }{
          \hyperparam_i
        }\right|_{\hyperparam_i=1} 
      $
      is obtained by following the same procedure above.
      The derivative \eqref{eqn:CB_Model_Full_Gradient} is then obtained by
      employing \eqref{eqn:CB_Model_Full_PMF} where the set $I$ excludes the index $i$.
      One can also show that the derivative \eqref{eqn:CB_Model_Full_Gradient}
      are continuous at the bounds.
      Finally, the 
      gradient \eqref{eqn:GCB_Model-LogProb-Grad}
      is obtained by employing the derivative of the logarithm, and,
      \begin{equation}
        \begin{aligned}
          &\nabla_{\hyperparam} \log \CondProb{\design}{\designsum\in \designsumset}
            \stackrel{(\ref{eqn:GCB_Model_PMF})}{=} \nabla_{\hyperparam} \log
                \sum\limits_{\designsumval\!\in\! \designsumset}
                  \CondProb{\design}{\designsum\!=\!\designsumval}
                    \Prob({\designsum\!=\!\designsumval})
              - \nabla_{\hyperparam} \log
                \sum\limits_{\designsumval\!\in\! \designsumset}
                  \Prob({\designsum\!=\!\designsumval})
            \\
            &\quad= \frac{
                \sum\limits_{\designsumval\!\in\! \designsumset}
                  \Prob({\designsum\!=\!\designsumval})
                  \nabla_{\hyperparam}
                    \CondProb{\design}{\designsum\!=\!\designsumval}
                  +
                \sum\limits_{\designsumval\!\in\! \designsumset}
                  \CondProb{\design}{\designsum\!=\!\designsumval}
                  \nabla_{\hyperparam}
                    \Prob({\designsum\!=\!\designsumval})
              }{
                \sum\limits_{\designsumval\!\in\!\designsumset}
                    \CondProb{\design}{\designsum\!=\!\designsumval}
                      \Prob({\designsum\!=\!\designsumval})
              }
              -
              \frac{
                \sum\limits_{\designsumval\!\in\!\designsumset}
                  \nabla_{\hyperparam}
                  \Prob({\designsum\!=\!\designsumval})
              }{
                \sum\limits_{\designsumval\!\in\!\designsumset}
                  \Prob({\designsum\!=\!\designsumval})
              } \,.
        \end{aligned}
      \end{equation}

    \end{proof}

  \section{
    Proofs of Theorems and Lemmas in \Cref{sec:probabilistic_optimization}
  }\label{app:probabilistic_optimization_proofs}

    First we introduce \Cref{proposition:CB_Model_gradient_moments} which will be used 
    in the proofs in this section.
    \begin{proposition}\label{proposition:CB_Model_gradient_moments}
      Let $\design$ be distributed according to the CB model 
      \eqref{eqn:CB_Model_PMF}, then
      \begin{equation}\label{eqn:CB_Model_gradient_moments}
        \Expect{}{ 
            \nabla_{\hyperparam} \log \CondProb{\design}{\designsum\!=\!\designsumval} 
          }
            =\vec{0} \,; \quad 
          \brVar{ 
            \nabla_{\hyperparam} \log \CondProb{\design}{\designsum\!=\!\designsumval} 
          }
            = \sum_{i=1}^{\Nsens} \frac{(1\!+\!w_i)^4}{w_i^2} (\pi_{i}\!-\!\pi_{i}^2) \,, 
      \end{equation}
      where the total variance $\brVar{\vec{x}}$ is
      the sum of elementwise variances.
    \end{proposition}
    \begin{proof}[Proof of \Cref{proposition:CB_Model_gradient_moments}]
      By using the linearity property of the expectation, 
      \begin{equation}\label{eqn:CB_Model_gradient_mean}
          \Expect{}{
            \nabla_{\hyperparam} \log \CondProb{\design}{\designsum=\designsumval} 
          }
          \stackrel{(\ref{eqn:first_second_inclusion},\ref{eqn:CB_PMF_derivative_interior})}{=}
            \sum_{i=1}^{\Nsens} 
            \frac{(1\!+\!w_i)^2}{w_i}  \left(\CondExpect{ \design_i} -\pi_i\right) 
          \vec{e}_i  
          \stackrel{(\ref{eqn:CB_Model_moments_1})}{=}
          \vec{0} \,,
      \end{equation}
      where $\vec{e}_i$  is the $i$th unit vector in $\Rnum^{\Nsens}$. 
      $
        \brVar{ 
          \nabla_{\hyperparam} \log \CondProb{\design}{\designsum=\designsumval} 
        }
      $ is obtained as follows: 
      \begin{equation}
        \begin{aligned}
          &\brVar{ 
            \nabla_{\hyperparam} \log \CondProb{\design}{\designsum\!=\!\designsumval} 
          }
          =  
          \sum_{i=1}^{\Nsens}
            \CondVar{ 
              \nabla_{\hyperparam} \log \CondProb{\design}{\designsum=\designsumval} 
            }{ \designsum=\designsumval} 
            \vec{e}_i  \\
          &\quad\qquad= 
          \sum_{i=1}^{\Nsens}
            \brVar{ 
              \del{\log \CondProb{\design}{\designsum\!=\!\designsumval } }{\hyperparam_i} 
            }
          \stackrel{(\ref{eqn:CB_Model_gradient_mean}, \ref{eqn:CB_PMF_derivative_interior})}{=}
          \sum_{i=1}^{\Nsens}
            \Expect{}{ 
                \frac{(1+w_i)^4}{w_i^2}  \left(\design_i\!-\!\pi_i\right)^2 
            } \\
          &\quad\qquad=
          \sum_{i=1}^{\Nsens} 
            \frac{(1+w_i)^4}{w_i^2} 
            \CondExpect{
              \design_i^2 - 2 \design_{i} \pi_{i} + \pi_{i} ^2
            }{ \designsum\!=\!\designsumval} 
          =
          \sum_{i=1}^{\Nsens} 
            \frac{(1+w_i)^4}{w_i^2} \pi_{i} - \pi_{i}^2 
          \,, 
        \end{aligned}
      \end{equation}
      where we used the fact that 
      $
        \design_i^2=\design_i,\, 
        \CondExpect{\design_i}{\designsum=\designsumval}\!=\!\pi_i
      $ 
      in the last step. 
    \end{proof}

    \begin{proof}[Proof of \Cref{lemma:CB_Model_derivatives_bounds}]
      The first bound \eqref{eqn:CB_Model_derivatives_bounds_1} is given by
      \begin{equation}
        \begin{aligned}
          \label{eqn:CB_Model_derivatives_bounds_1_proof} 
          &\sqnorm{\nabla_{\hyperparam} \CondProb{\design}{\designsum\!=\!\designsumval}}
            = \sum\limits_{i=1}^{\Nsens} 
              \left( 
                \del{\CondProb{\design}{\designsum\!=\!\designsumval}}{\hyperparam_i} 
              \right)^2
          \\
          &\qquad\stackrel{(\ref{eqn:CB_Model_PMF}, \ref{eqn:CB_PMF_derivative_interior})}{=}
            {
            \CondProb{\design}{\designsum\!=\!\designsumval }
            }^2
            \sum\limits_{i=1}^{\Nsens} 
                \frac{(1\!+\!w_i)^4}{w_i^2}
                \left(\design_i-\pi_i\right)^2
          \leq
            \sum\limits_{i=1}^{\Nsens} 
                \frac{(1\!+\!w_i)^4}{w_i^2}
                \left(\design_i-\pi_i\right)^2
                \\
          &\,\,\qquad\stackrel{(\ref{eqn:CB_PMF_derivative_interior})}{=} 
            \sqnorm{\nabla_{\hyperparam} \log \CondProb{\design}{\designsum\!=\!\designsumval}}
          \stackrel{(\ref{eqn:R_function_and_Bernoulli_weights_bounds})}{\leq} 
            \sum\limits_{i=1}^{\Nsens} 
              C^2
              \left(\design_i-\pi_i\right)^2 
          <
          \Nsens\, C^2
          \,.
        \end{aligned}
      \end{equation}
      where $C= \max\limits_{i=1,\ldots,\Nsens} \frac{(1+w_i)^2}{w_i}$;
      we used the fact that $\CondProb{\design}{\designsum=\designsumval}$ is constant 
      for all terms in the sum in the second step; 
      $0\leq\CondProb{\design}{\designsum=\designsumval}\leq 1$ is used in the third step;
      and 
      because $\design_i\in\{0, 1\},\, \pi_i\in(0, 1)$, it follows that 
      $0 < \left(\design_i\!-\!\pi_i\right)^2<1$, and
      $
        \sum_{i=1}^{\Nsens}(\design_i\!-\!\pi_i)^2< 
      \Nsens$, 
      which is used 
      in the final step which proves \eqref{eqn:CB_Model_derivatives_bounds_1}.
      The second bound \eqref{eqn:CB_Model_derivatives_bounds_2} is obtained
      as follows:
      \begin{equation}\label{eqn:CB_Model_derivatives_bounds_2_proof}
        \begin{aligned}
          \Expect{}{
            \sqnorm{
              \nabla_{\hyperparam} \log \CondProb{\design}{\designsum\!=\!\designsumval}  
            }
          }
          &=
          \Trace{
            \Expect{}{
              \left(
                \nabla_{\hyperparam} \log 
                  \CondProb{\design}{\designsum=\designsumval}  
              \right)\tran
              \left(
                \nabla_{\hyperparam} \log 
                  \CondProb{\design}{\designsum=\designsumval}  
              \right)
            }
          }
          \\
          &\stackrel{(\ref{eqn:CB_Model_gradient_moments})}{=} 
          \brVar{
              \nabla_{\hyperparam} \log 
                \CondProb{\design}{\designsum\!=\!\designsumval}  
            }
          \stackrel{(\ref{eqn:CB_Model_gradient_moments})}{=} 
            \sum_{i=1}^{\Nsens} \frac{(1+w_i)^4}{w_i^2} (\pi_{i}-\pi_{i}^2) 
          \\
          &\leq
            \frac{1}{4}\, \sum_{i=1}^{\Nsens} \frac{(1+w_i)^4}{w_i^2} 
            \stackrel{(\ref{eqn:R_function_and_Bernoulli_weights_bounds})}{\leq}
              \frac{\Nsens }{4}\, C^2
              \,, 
        \end{aligned}
      \end{equation}
      where we used the circular property of the trace, 
      and linearity of the trace and the expectation
      in the first step. 
      Moreover, since for $\hyperparam_i\in(0, 1)$ it holds that 
      $\pi_i\in(0, 1)$  hence $(\pi_i-\pi_i^2)$ 
      is a quadratic with maximum value $\frac{1}{4}$ which is used in the last step.
      \eqref{eqn:CB_Model_derivatives_bounds_3} is achieved as follows: 
      \begin{equation}
        \begin{aligned}
          \del{\pi_i}{\hyperparam_i} 
            &\stackrel{\eqref{eqn:first_second_inclusion}}{=}
            \frac{
              R(z, S) w_i \del{R(z\!-\!1, S\!\setminus\!\{i\})}{\hyperparam_i} 
                - w_i R(z\!-\!1, S\!\setminus\!\{i\}) \del{R(z, S)}{\hyperparam_i}
            }{
              R(z, S)^2
            }
            \stackrel{\eqref{eqn:R_function_derivative_param}}{=}
              \frac{- \pi_i^2 (1\!+\!w_i)^2}{w_i}  \,,
          \label{eqn:pi_i_derivative_i}
        \end{aligned}
      \end{equation}
      
      \begin{equation}
        \begin{aligned}
          \del{\pi_i}{\hyperparam_j} 
            &\stackrel{\eqref{eqn:first_second_inclusion}}{=}
            \frac{
              R(z, S) w_i \del{R(z\!-\!1, S\!\setminus\!\{i\})}{\hyperparam_j} 
                - w_i R(z\!-\!1, S\!\setminus\!\{i\}) \del{R(z, S)}{\hyperparam_j}
            }{
              R(z, S)^2
            }
            \\
            &\stackrel{\eqref{eqn:R_function_derivative_param}}{=}
              \frac{(1\!+\!w_j)^2 (\pi_{i, j}\!-\!\pi_i \pi_j)}{w_j}  \,,
          \label{eqn:pi_i_derivative_j}
        \end{aligned}
      \end{equation}
      
      \begin{equation}
        \begin{aligned}
          &\delll{ 
              \CondProb{\design}{\designsum\!=\!\designsumval } 
            }{\hyperparam_i}{\hyperparam_j}
            \stackrel{\eqref{eqn:CB_PMF_derivative_interior}}{=}
              \del{}{\hyperparam_j} \left(
                \CondProb{\design}{\designsum\!=\!\designsumval} 
                \frac{(1\!+w_i)^2(\design_i\!-\!\pi_i)}{w_i}
              \right)
            \\
            &\qquad=
              \frac{(1\!+w_i)^2}{w_i}
              \del{}{\hyperparam_j} \left(
                - \pi_i
                \CondProb{\design}{\designsum\!=\!\designsumval} 
              \right)
            \\
            &\quad\stackrel{
              ( \ref{eqn:CB_PMF_derivative_interior}, 
                \ref{eqn:pi_i_derivative_i}, 
                \ref{eqn:pi_i_derivative_j}
                )}{=}
                \frac{(1\!+\!w_i)^2}{w_i}
                \frac{(1\!+\!w_j)^2}{w_j}
                \CondProb{\design}{\designsum\!=\!\designsumval} 
                \left(\begin{aligned}
                  &\pi_{i}^2 \delta_{ij} 
                  \!+\! (\pi_{i}\pi_{j}\!-\!\pi_{i,j}) (1\!-\!\delta_{ij})
                  \\
                  &\!+\! (\design_{i}\!-\!\pi_{i}) (\design_{j}\!-\!\pi_{j})
                \end{aligned}
                \right) \,,
          \label{eqn:CB_Hessian}
        \end{aligned}
      \end{equation}
      
      \begin{equation}
        \begin{aligned}
          \abs{
            \delll{ 
              \CondProb{\design}{\designsum\!=\!\designsumval } 
            }{
              \hyperparam_i}{\hyperparam_j}
            }
          &\stackrel{(\ref{eqn:R_function_and_Bernoulli_weights_bounds}, \ref{eqn:CB_Hessian})}{\leq}
            C^2\, \left(
                \delta_{ij} 
                + (1-\delta_{ij}) 
                + 1 
            \right)
            = 2  C^2 
            \,, 
        \end{aligned}
      \end{equation}
      where we used the fact that 
      $\pi_{i}, \pi_{i, j}$ \eqref{eqn:first_second_inclusion}
      satisfy the following relations (see, e.g., \cite{sunter1986solutions} ):
      $  \sum_{i=1}^{N} \pi_i = n  \,; \, 
        \pi_{i} \pi_{j} >  \pi_{i, j}  \,,\, 
            i,j=1, \ldots, \Nsens\,, i\neq j 
      $. 
      Note that this also proves 
      \eqref{eqn:CB_Model_Full_derivatives_bounds} 
      for \eqref{eqn:CB_Model_Full_Gradient} when $0<\hyperparam_i<1$.
      The degenerate cases in \eqref{eqn:CB_Model_Full_Gradient} rely only on non-degenerate 
      components of $\hyperparam$.
      Thus, from \eqref{eqn:CB_Model_derivatives_bounds} it follows that 
      there is always a finite number $\widehat{C}$ (defined based on the maximum/minimum success
      probability $0<\hyperparam_i<1$) that bounds first- and second-order derivatives.
      which proves \eqref{eqn:CB_Model_derivatives_bounds}.
    \end{proof}
    \begin{proof}[Proof of
      \Cref{theorem:probabilistic_optimization_derivatives_bounds_exact}
      ]
      By using \eqref{eqn:probabilistic_optimization_budget_equality_constraint} and 
      the triangle inequality of the norm,
      \begin{equation}
        \norm{\nabla_{\hyperparam} \stochobj(\hyperparam)} 
        \stackrel{(\ref{eqn:probabilistic_optimization_budget_equality_constraint})}{\leq}
          \sum_{\substack{\design \in \{0, 1\}^{\Nsens} \\ \wnorm{\design}{0}=\designsumval}}
          \norm{
            \nabla_{\hyperparam}\, \CondProb{\design}{\designsum=\designsumval}
          } 
          \stackrel{(\ref{eqn:CB_Model_derivatives_bounds_1})}{\leq} 
          M\, \binom{\Nsens}{\designsumval}  C \sqrt{\Nsens} \,,
      \end{equation}
      which proves \eqref{eqn:exact_gradient_norm_bound}.
      For two realizations of the CB model parameters
      $\hyperparam[1], \hyperparam[2]$, 
      \begin{equation}
        \norm{ 
          \nabla_{\hyperparam}\,\stochobj(\hyperparam[1]) - 
          \nabla_{\hyperparam}\,\stochobj(\hyperparam[2]) 
        } 
        \leq 
        \norm{ 
          \nabla_{\hyperparam}\,\stochobj(\hyperparam[1])
        } + 
        \nabla_{\hyperparam}\,\stochobj(\hyperparam[2]) 
        \stackrel{(\ref{eqn:exact_gradient_norm_bound})}{\leq}
        2 \, M\, \sqrt{
          \binom{\Nsens}{\designsumval} \Nsens C
        } \,, 
      \end{equation}
      which proves \eqref{eqn:exact_gradient_Lipschitz}.
      The bound on the Hessian entries \eqref{eqn:Hessian_entries_bound}
      is given by
      \begin{equation}
          \left(
            \delll{\stochobj}{\hyperparam_i}{\hyperparam_j}
          \right)^2
          =
          \left( \delll{}{\hyperparam_i}{\hyperparam_j}
            \sum_{\substack{\design \in \{0, 1\}^{\Nsens} \\ \wnorm{\design}{0}=\designsumval}}
            \obj(\design) 
            \CondProb{\design}{\designsum=\designsumval}
          \right)^2  
          \stackrel{(\ref{eqn:CB_Model_derivatives_bounds_3})}{\leq}
             2 C^2  M^2  \binom{\Nsens}{\designsumval} \,,
      \end{equation}
      which 
      proves \eqref{eqn:Hessian_entries_bound}.
      The boundedness of 
      \eqref{eqn:exact_gradient_bounds_degenerate} 
      follows 
      from \eqref{eqn:exact_gradient_bounds} 
      and \eqref{eqn:CB_Model_Full_derivatives_bounds}.
    \end{proof}
    \begin{proof}[Proof of
      \Cref{theorem:probabilistic_optimization_derivatives_bounds_stochastic}
      ]

      The proof of \eqref{eqn:probabilistic_optimization_derivatives_bounds_stochastic_1} 
      is obtained as follows.
      \begin{equation} 
      \begin{aligned} 
        \Expect{}{\widehat{\vec{g}}}
          &= \Expect{}{
            \frac{1}{\Nens} 
              \sum_{k=1}^{\Nens}
              \obj(\design[k]) \nabla \log \CondProb{\design[k]}{\designsum} 
          } 
          \\
          &\stackrel{(\ref{eqn:probabilistic_optimization_derivatives})}{=} 
          \frac{1}{\Nens} 
              \sum_{k=1}^{\Nens}
              \Expect{}{
              \obj(\design) \nabla \log \CondProb{\design}{\designsum} 
          }  
          \\
          &= \frac{1}{\Nens} 
              \sum_{k=1}^{\Nens}
              \sum_{\design}
                \obj(\design) \nabla \log \CondProb{\design}{\designsum} 
                  \CondProb{\design}{\designsum}
          \\
          &= \frac{1}{\Nens} 
              \sum_{k=1}^{\Nens} 
                \Expect{\CondProb{\design}{\designsum} }{\obj(\design) \nabla \log \CondProb{\design}{\designsum}}  
          \\
          &= \Expect{\CondProb{\design}{\designsum}}{\obj(\design) \nabla \log \CondProb{\design}{\designsum}} 
          \stackrel{(\ref{eqn:budget_equality_constraint_stochastic_gradient},\ref{eqn:budget_equality_constraint_inclusion_gradient})}{=}\vec{g} \,,
      \end{aligned} 
      \end{equation} 
      which proves that the stochastic approximation is an unbiased estimator. 
      Note that 
      \begin{align}
        \label{eqn:final_1}
        \Expect{}{\widehat{\vec{g}}\tran \,\widehat{\vec{g}} }
          &= \brVar{\widehat{\vec{g}}} + 
            \Expect{}{\widehat{\vec{g}}}\tran \Expect{}{\widehat{\vec{g}}}
          = \brVar{\widehat{\vec{g}}} + \vec{g}\tran \vec{g} \,, \\
        \brVar{\widehat{\vec{g}}}
          &= \frac{1}{\Nens^2} 
            \sum_{k=1}^{\Nens}
            \brVar{
              \obj(\design[k]) \nabla \log \CondProb{\design}{\designsum}
            } 
          \leq \frac{M^2}{\Nens} 
            \brVar{
              \nabla \log \CondProb{\design}{\designsum}
            } 
          \,,
      \end{align}
      where $M= \max\limits_{\design\in \designdomain}{\{ \left|\obj(\design)\right| \}}$.
      By \Cref{lemma:CB_Model_derivatives_bounds}, there is a positive finite constant 
      $\widetilde{C}$ such that 
      $ \brVar{
         \nabla \log \CondProb{\design}{\designsum}
        }\leq \widetilde{C}$, which bounds the first term in \eqref{eqn:final_1} and thus completes the proof of 
        \eqref{eqn:probabilistic_optimization_derivatives_bounds_stochastic_2} 
        and the theorem. 
    \end{proof}  


\bibliographystyle{siamplain}
  \bibliography{references}

\begin{thebibliography}{10}

\bibitem{aarset2025global}
{\sc C.~Aarset}, {\em Global optimality conditions for sensor placement, with
  extensions to binary low-rank {A}-optimal designs}, Inverse Problems, 41
  (2025), p.~065013.

\bibitem{alexanderian2020optimal}
{\sc A.~Alexanderian}, {\em Optimal experimental design for
  infinite-dimensional {B}ayesian inverse problems governed by {PDEs}: A
  review}, Inverse Problems, 37 (2021), p.~043001.

\bibitem{alexanderian2014optimal}
{\sc A.~Alexanderian, N.~Petra, G.~Stadler, and O.~Ghattas}, {\em A-optimal
  design of experiments for infinite-dimensional {B}ayesian linear inverse
  problems with regularized $\ell_0$-sparsification}, SIAM Journal on
  Scientific Computing, 36 (2014), pp.~A2122--A2148.

\bibitem{AlexanderianPetraStadlerEtAl16}
{\sc A.~Alexanderian, N.~Petra, G.~Stadler, and O.~Ghattas}, {\em A fast and
  scalable method for {A}-optimal design of experiments for
  infinite-dimensional {B}ayesian nonlinear inverse problems}, SIAM Journal on
  Scientific Computing, 38 (2016), pp.~A243--A272.

\bibitem{AlexanderianSaibaba17}
{\sc A.~Alexanderian and A.~K. Saibaba}, {\em Efficient {D}-optimal design of
  experiments for infinite-dimensional {B}ayesian linear inverse problems},
  SIAM Journal on Scientific Computing, 40 (2018), pp.~A2956--A2985.

\bibitem{asch2016data}
{\sc M.~Asch, M.~Bocquet, and M.~Nodet}, {\em Data assimilation: methods,
  algorithms, and applications}, SIAM, 2016.

\bibitem{attia2016PhD_advanced_sampling}
{\sc A.~Attia}, {\em Advanced Sampling Methods for Solving Large-Scale Inverse
  Problems}, PhD thesis, Virginia Tech, 2016.

\bibitem{attia2026probabilistic_path_oed}
{\sc A.~Attia}, {\em A probabilistic approach to trajectory-based optimal
  experimental design}, 2026, \url{https://arxiv.org/abs/2601.11473}.
\newblock Preprint, arXiv:2601.11473.

\bibitem{attia2023pyoedRepo}
{\sc A.~Attia}, {\em {PyOED}: An extensible suite for data assimilation and
  model-constrained optimal design of experiments}, 2026,
  \url{https://gitlab.com/ahmedattia/pyoed}.

\bibitem{attia2018goal}
{\sc A.~Attia, A.~Alexanderian, and A.~K. Saibaba}, {\em Goal-oriented optimal
  design of experiments for large-scale {B}ayesian linear inverse problems},
  Inverse Problems, 34 (2018), p.~095009.

\bibitem{attia2022optimal}
{\sc A.~Attia and E.~Constantinescu}, {\em Optimal experimental design for
  inverse problems in the presence of observation correlations}, SIAM Journal
  on Scientific Computing, 44 (2022), pp.~A2808--A2842.

\bibitem{attia2022stochastic}
{\sc A.~Attia, S.~Leyffer, and T.~Munson}, {\em Stochastic learning approach
  for binary optimization: Application to {B}ayesian optimal design of
  experiments}, SIAM Journal on Scientific Computing, 44 (2022),
  pp.~B395--B427.

\bibitem{attia2023robust}
{\sc A.~Attia, S.~Leyffer, and T.~Munson}, {\em Robust {A}-optimal experimental
  design for {B}ayesian inverse problems}, 2023,
  \url{https://arxiv.org/abs/2305.03855}.

\bibitem{attia2018ClHMCAtmos}
{\sc A.~Attia, A.~Moosavi, and A.~Sandu}, {\em Cluster sampling filters for
  non-{Gaussian} data assimilation}, Atmosphere, 9 (2018).

\bibitem{attia2015hmcsmoother}
{\sc A.~Attia, V.~Rao, and A.~Sandu}, {\em A hybrid {Monte Carlo} sampling
  smoother for four dimensional data assimilation}, International Journal for
  Numerical Methods in Fluids,  (2016).
\newblock fld.4259.

\bibitem{attia2015hmcfilter}
{\sc A.~Attia and A.~Sandu}, {\em A hybrid {M}onte {C}arlo sampling filter for
  non-{Gaussian} data assimilation}, AIMS Geosciences, 1 (2015), pp.~41--78.

\bibitem{attia2016reducedhmcsmoother}
{\sc A.~Attia, R.~{\c{S}}tef{\u{a}}nescu, and A.~Sandu}, {\em The reduced-order
  hybrid {Monte Carlo} sampling smoother}, International Journal for Numerical
  Methods in Fluids,  (2016).
\newblock fld.4255.

\bibitem{bannister2017review}
{\sc R.~Bannister}, {\em A review of operational methods of variational and
  ensemble-variational data assimilation}, Quarterly Journal of the Royal
  Meteorological Society, 143 (2017), pp.~607--633.

\bibitem{bertero2020introduction}
{\sc M.~Bertero and P.~Boccacci}, {\em Introduction to inverse problems in
  imaging}, CRC press, 2020.

\bibitem{BertsekasTsitsiklis96}
{\sc D.~P. Bertsekas and J.~Tsitsiklis}, {\em Neuro--Dynamic Programming},
  Athena Scientific, Belmont, Massachusetts, 1996.

\bibitem{bouttier2002data}
{\sc F.~Bouttier and P.~Courtier}, {\em Data assimilation concepts and methods
  march 1999}, Meteorological training course lecture series. ECMWF, 718
  (2002), p.~59.

\bibitem{boykov2001fast}
{\sc Y.~Boykov, O.~Veksler, and R.~Zabih}, {\em Fast approximate energy
  minimization via graph cuts}, IEEE Transactions on Pattern Analysis and
  Machine Intelligence, 23 (2001), pp.~1222--1239.

\bibitem{bui2013computational}
{\sc T.~Bui-Thanh, O.~Ghattas, J.~Martin, and G.~Stadler}, {\em A computational
  framework for infinite-dimensional {B}ayesian inverse problems {Part I}: The
  linearized case, with application to global seismic inversion}, SIAM Journal
  on Scientific Computing, 35 (2013), pp.~A2494--A2523.

\bibitem{byrd1995limited}
{\sc R.~H. Byrd, P.~Lu, J.~Nocedal, and C.~Zhu}, {\em A limited memory
  algorithm for bound constrained optimization}, SIAM Journal on Scientific
  Computing, 16 (1995), pp.~1190--1208.

\bibitem{chan2011convex}
{\sc E.~Y. Chan and D.-Y. Yeung}, {\em A convex formulation of modularity
  maximization for community detection}, in Proceedings of the Twenty-Second
  International Joint Conference on Artificial Intelligence (IJCAI), Barcelona,
  Spain, 2011, p.~2218.

\bibitem{chen2000general}
{\sc S.~X. Chen}, {\em General properties and estimation of conditional
  {B}ernoulli models}, Journal of Multivariate Analysis, 74 (2000), pp.~69--87.

\bibitem{chen1997statistical}
{\sc S.~X. Chen and J.~S. Liu}, {\em Statistical applications of the
  {P}oisson-binomial and conditional {B}ernoulli distributions}, Statistica
  Sinica,  (1997), pp.~875--892.

\bibitem{chen2003bayesian}
{\sc Z.~Chen et~al.}, {\em Bayesian filtering: From {Kalman} filters to
  particle filters, and beyond}, Statistics, 182 (2003), pp.~1--69.

\bibitem{attia2024pyoed}
{\sc A.~Chowdhary, S.~E. Ahmed, and A.~Attia}, {\em {PyOED}: An extensible
  suite for data assimilation and model-constrained optimal design of
  experiments}, ACM Trans. Math. Softw., 50 (2024),
  \url{https://doi.org/10.1145/3653071}, \url{https://doi.org/10.1145/3653071}.

\bibitem{chowdharyNonlinearRobust}
{\sc A.~Chowdhary, A.~Attia, and A.~Alexanderian}, {\em Robust optimal
  experimental design of infinite-dimensional {B}ayesian nonlinear inverse
  problems}, SIAM Journal on Scientific Computing, 48 (2026), pp.~B262--B288,
  \url{https://doi.org/10.1137/24M1693921},
  \url{https://doi.org/10.1137/24M1693921},
  \url{https://arxiv.org/abs/https://doi.org/10.1137/24M1693921}.

\bibitem{davenport20141}
{\sc M.~A. Davenport, Y.~Plan, E.~Van Den~Berg, and M.~Wootters}, {\em 1-bit
  matrix completion}, Information and Inference: A Journal of the IMA, 3
  (2014), pp.~189--223.

\bibitem{elton2019deep}
{\sc D.~C. Elton, Z.~Boukouvalas, M.~D. Fuge, and P.~W. Chung}, {\em Deep
  learning for molecular design—a review of the state of the art}, Molecular
  Systems Design \& Engineering, 4 (2019), pp.~828--849.

\bibitem{evensen2009data}
{\sc G.~Evensen}, {\em Data assimilation: the ensemble {Kalman} filter},
  Springer Science \& Business Media, 2009.

\bibitem{FedorovLee00}
{\sc V.~Fedorov and J.~Lee}, {\em Design of experiments in statistics}, in
  Handbook of semidefinite programming, R.~S. H.~Wolkowicz and L.~Vandenberghe,
  eds., vol.~27 of Internat. Ser. Oper. Res. Management Sci., Kluwer Acad.
  Publ., Boston, MA, 2000, pp.~511--532.

\bibitem{fillion2020iterative}
{\sc A.~Fillion, M.~Bocquet, S.~Gratton, S.~G\"{u}rol, and P.~Sakov}, {\em An
  iterative ensemble {Kalman} smoother in presence of additive model error},
  SIAM/ASA Journal on Uncertainty Quantification, 8 (2020), pp.~198--228.

\bibitem{FrangosMarzoukWillcoxEtAl11}
{\sc M.~Frangos, Y.~Marzouk, K.~Willcox, and B.~van Bloemen~Waanders}, {\em
  Surrogate and Reduced-Order Modeling: A Comparison of Approaches for
  Large-Scale Statistical Inverse Problems}, John Wiley \& Sons, Ltd, 2010,
  ch.~7, pp.~123--149.

\bibitem{ghil1991data}
{\sc M.~Ghil and P.~Malanotte-Rizzoli}, {\em Data assimilation in meteorology
  and oceanography}, in Advances in geophysics, vol.~33, Elsevier, 1991,
  pp.~141--266.

\bibitem{Hairer09}
{\sc M.~Hairer}, {\em Introduction to {S}tochastic {PDE}s}, lecture notes,
  2009.

\bibitem{he2016joint}
{\sc L.~He, C.-T. Lu, J.~Ma, J.~Cao, L.~Shen, and P.~S. Yu}, {\em Joint
  community and structural hole spanner detection via harmonic modularity}, in
  Proceedings of the 22nd ACM SIGKDD International Conference on Knowledge
  Discovery and Data Mining, 2016, pp.~875--884.

\bibitem{heng2020simple}
{\sc J.~Heng, P.~E. Jacob, and N.~Ju}, {\em A simple {M}arkov chain for
  independent {B}ernoulli variables conditioned on their sum}, arXiv preprint
  arXiv:2012.03103,  (2020).

\bibitem{holm2020massively}
{\sc H.~H. Holm, M.~L. S{\ae}tra, and P.~J. Van~Leeuwen}, {\em Massively
  parallel implicit equal-weights particle filter for ocean drift trajectory
  forecasting}, Journal of Computational Physics: X,  (2020), p.~100053.

\bibitem{huan2024optimal}
{\sc X.~Huan, J.~Jagalur, and Y.~Marzouk}, {\em Optimal experimental design:
  Formulations and computations}, Acta Numerica, 33 (2024), pp.~715--840.

\bibitem{kaipio2006statistical}
{\sc J.~Kaipio and E.~Somersalo}, {\em Statistical and computational inverse
  problems}, vol.~160, Springer Science \& Business Media, 2006.

\bibitem{kalnay2003atmospheric}
{\sc E.~Kalnay}, {\em Atmospheric modeling, data assimilation and
  predictability}, Cambridge University Press, 2003.

\bibitem{keuchel2003binary}
{\sc J.~Keuchel, C.~Schnorr, C.~Schellewald, and D.~Cremers}, {\em Binary
  partitioning, perceptual grouping, and restoration with semidefinite
  programming}, IEEE Transactions on Pattern Analysis and Machine Intelligence,
  25 (2003), pp.~1364--1379.

\bibitem{kingma2014adam}
{\sc D.~P. Kingma and J.~Ba}, {\em Adam: A method for stochastic optimization},
  arXiv preprint arXiv:1412.6980,  (2014).

\bibitem{koval2024non}
{\sc K.~Koval and R.~Nicholson}, {\em Non-intrusive optimal experimental design
  for large-scale nonlinear {B}ayesian inverse problems using a {B}ayesian
  approximation error approach}, arXiv preprint arXiv:2405.07412,  (2024).

\bibitem{kroese2013handbook}
{\sc D.~P. Kroese, T.~Taimre, and Z.~I. Botev}, {\em Handbook of {Monte Carlo}
  methods}, John Wiley \& Sons, 2013.

\bibitem{lawson1995solving}
{\sc C.~L. Lawson and R.~J. Hanson}, {\em Solving least squares problems},
  SIAM, 1995.

\bibitem{lindgren2011explicit}
{\sc F.~Lindgren, H.~Rue, and J.~Lindstr{\"o}m}, {\em An explicit link between
  {G}aussian fields and {G}aussian {M}arkov random fields: the stochastic
  partial differential equation approach}, Journal of the Royal Statistical
  Society Series B: Statistical Methodology, 73 (2011), pp.~423--498.

\bibitem{meng2020training}
{\sc X.~Meng, R.~Bachmann, and M.~E. Khan}, {\em Training binary neural
  networks using the {B}ayesian learning rule}, in International Conference on
  Machine Learning, PMLR, 2020, pp.~6852--6861.

\bibitem{morzfeld2018variational}
{\sc M.~Morzfeld, D.~Hodyss, and J.~Poterjoy}, {\em Variational particle
  smoothers and their localization}, Quarterly Journal of the Royal
  Meteorological Society, 144 (2018), pp.~806--825.

\bibitem{nocedal2006numerical}
{\sc J.~Nocedal and S.~Wright}, {\em Numerical optimization}, Springer Science
  \& Business Media, 2006.

\bibitem{papalexopoulos2022constrained}
{\sc T.~P. Papalexopoulos, C.~Tjandraatmadja, R.~Anderson, J.~P. Vielma, and
  D.~Belanger}, {\em Constrained discrete black-box optimization using
  mixed-integer programming}, in International Conference on Machine Learning,
  PMLR, 2022, pp.~17295--17322.

\bibitem{Pazman86}
{\sc A.~P{\'a}zman}, {\em Foundations of optimum experimental design}, D.
  Reidel Publishing Co., 1986.

\bibitem{PetraStadler11}
{\sc N.~Petra and G.~Stadler}, {\em Model variational inverse problems governed
  by partial differential equations}, Tech. Report 11-05, The Institute for
  Computational Engineering and Sciences, The University of Texas at Austin,
  2011.

\bibitem{Pukelsheim93}
{\sc F.~Pukelsheim}, {\em Optimal design of experiments}, John Wiley \& Sons,
  New York, 1993.

\bibitem{ryu2023heuristic}
{\sc M.~Ryu, A.~Attia, A.~Barnes, R.~Bent, S.~Leyffer, and A.~Mate}, {\em
  Heuristic algorithms for placing geomagnetically induced current blocking
  devices}, arXiv preprint arXiv:2310.09409,  (2023).

\bibitem{shi2000normalized}
{\sc J.~Shi and J.~Malik}, {\em Normalized cuts and image segmentation}, IEEE
  Transactions on Pattern Analysis and Machine Intelligence, 22 (2000),
  pp.~888--905.

\bibitem{shi2013multi}
{\sc X.~Shi, H.~Ling, J.~Xing, and W.~Hu}, {\em Multi-target tracking by rank-1
  tensor approximation}, in Proceedings of the IEEE Conference on Computer
  Vision and Pattern Recognition, 2013, pp.~2387--2394.

\bibitem{sunter1986solutions}
{\sc A.~Sunter}, {\em Solutions to the problem of unequal probability sampling
  without replacement}, International Statistical Review/Revue Internationale
  de Statistique,  (1986), pp.~33--50.

\bibitem{tarantola2005inverse}
{\sc A.~Tarantola}, {\em Inverse problem theory and methods for model parameter
  estimation}, SIAM, 2005.

\bibitem{Ucinski00}
{\sc D.~Uci{\'n}ski}, {\em Optimal sensor location for parameter estimation of
  distributed processes}, International Journal of Control, 73 (2000),
  pp.~1235--1248.

\bibitem{van2009particle}
{\sc P.~J. Van~Leeuwen}, {\em Particle filtering in geophysical systems},
  Monthly Weather Review, 137 (2009), pp.~4089--4114.

\bibitem{2020SciPy-NMeth}
{\sc P.~Virtanen, R.~Gommers, T.~E. Oliphant, et~al.}, {\em {SciPy} 1.0:
  Fundamental algorithms for scientific computing in {Python}}, Nature Methods,
  17 (2020), pp.~261--272, \url{https://doi.org/10.1038/s41592-019-0686-2}.

\bibitem{vogel2002computational}
{\sc C.~R. Vogel}, {\em Computational methods for inverse problems}, SIAM,
  2002.

\bibitem{wang2015learning}
{\sc J.~Wang, W.~Liu, S.~Kumar, and S.-F. Chang}, {\em Learning to hash for
  indexing big data—a survey}, Proceedings of the IEEE, 104 (2015),
  pp.~34--57.

\bibitem{wolsey2014integer}
{\sc L.~A. Wolsey and G.~L. Nemhauser}, {\em Integer and combinatorial
  optimization}, John Wiley \& Sons, 2014.

\bibitem{yang2019machine}
{\sc K.~K. Yang, Z.~Wu, and F.~H. Arnold}, {\em Machine-learning-guided
  directed evolution for protein engineering}, Nature Methods, 16 (2019),
  pp.~687--694.

\bibitem{yu2021multidimensional}
{\sc J.~Yu and M.~Anitescu}, {\em Multidimensional sum-up rounding for integer
  programming in optimal experimental design}, Mathematical Programming, 185
  (2021), pp.~37--76.

\bibitem{yu2018scalable}
{\sc J.~Yu, V.~M. Zavala, and M.~Anitescu}, {\em A scalable design of
  experiments framework for optimal sensor placement}, Journal of Process
  Control, 67 (2018), pp.~44--55.

\bibitem{yuan2013truncated}
{\sc X.-T. Yuan and T.~Zhang}, {\em Truncated power method for sparse
  eigenvalue problems}, Journal of Machine Learning Research, 14 (2013).

\end{thebibliography}

  \null \vfill
    \begin{flushright}
      \scriptsize 
      \framebox{\parbox{5.0in}{
        The submitted manuscript has been created by UChicago Argonne, LLC,
        Operator of Argonne National Laboratory (``Argonne"). Argonne, a
        U.S. Department of Energy Office of Science laboratory, is operated
        under Contract No. DE-AC02-06CH11357. The U.S. Government retains for
        itself, and others acting on its behalf, a paid-up nonexclusive,
        irrevocable worldwide license in said article to reproduce, prepare
        derivative works, distribute copies to the public, and perform
        publicly and display publicly, by or on behalf of the Government.
        The Department of Energy will provide public access to these results 
        of federally sponsored research in accordance with the DOE Public Access Plan. 
        http://energy.gov/downloads/doe-public-access-plan. 
      }}\normalsize
    \end{flushright}

\end{document}